\documentclass{amsart}
\usepackage{amsmath, amsthm, verbatim, amsfonts, amssymb, color}
\usepackage{mathrsfs}
\usepackage{dsfont}
\usepackage[a4paper,margin=3cm]{geometry}

\usepackage{graphicx}  
\usepackage{marginnote}  

\usepackage{algorithm,algorithmic}
\usepackage{booktabs}
\usepackage{multirow}
\usepackage{bigstrut}

\usepackage{float}

\usepackage{bm}
\usepackage{bbm}

\usepackage{ulem}
\usepackage{url}

\allowdisplaybreaks[4]
\usepackage{times}
\usepackage{latexsym}
\usepackage{booktabs}
\usepackage{mathtools}

\usepackage[colorlinks,pagebackref]{hyperref}
\usepackage{xcolor}
\usepackage{subfigure}
\usepackage{caption}
\usepackage{placeins}
\usepackage{booktabs}
\usepackage{multirow}

\newcommand{\R}{\mathbb{R}}

\newcommand{\sphere}{\mathbb{S}}

\newcommand{\norm}[1]{\left\lVert #1 \right\rVert}

\def\<{\langle}
\def\>{\rangle}

\def\beq{\begin{equation}}
\renewcommand{\hat}{\widehat}
\renewcommand{\tilde}{\widetilde}

\definecolor{wco}{rgb}{0.5,0.2,0.3}

\newtheorem{thm}{Theorem}[section]

\newtheorem{lem}[thm]{Lemma}

\newtheorem{rem}[thm]{Remark}
\newtheorem{assu}{Assumption}
\theoremstyle{definition}

\numberwithin{equation}{section}

\title{Tuning free Catoni type joint robust estimation}

\author[X. Li]{Xiang Li}
\address{Department of Statistics and Data Science, Southern University of Science and Technology}
\email{lixiang3@sustech.edu.cn}

\author[J. S. Liu]{Jun S. Liu}
\address{Department of Statistics and Data Science, Tsinghua University}
\email{junsliu@tsinghua.edu.cn}

\author[Q. Sun]{Qiang Sun}
\address{Department of Statistical Sciences, University of Toronto}
\email{qiang.sun@utoronto.ca}

\author[L. Xu]{Lihu Xu}
\address{Department of Mathematics, University of Macau}
\email{lihuxu@um.edu.mo}



\begin{document}
\begin{abstract}
    This paper develops a Catoni-type joint (tuning-free) estimation framework for parametric models with heavy-tailed noise, in which the target parameter and the unknown noise variance are estimated simultaneously through a system of two coupled Catoni-type estimating equations. We instantiate the framework in three canonical settings: mean estimation, linear regression, and $\ell_{2}$-penalized regression. 

Theoretically, we establish non-asymptotic, sub-Gaussian-type deviation bounds that hold jointly for the target parameter and the variance estimator, under only a finite $2\beta$-th moment assumption with $\beta\in (1,2]$. The resulting rates match — up to absolute constants — those of oracle procedures that know the variance in advance, thereby attaining optimality in the heavy-tailed regime. 

Methodologically, because the coupled equations are intrinsically non-convex and non-linear, classical convex M-estimation arguments are inapplicable. We develop a new analytical toolkit based on the Poincaré–Miranda theorem. The resulting proof strategy is of independent methodological interest, and we expect it to be applicable to a broad class of other statistical problems in which several parameters of heterogeneous nature must be estimated jointly.
\end{abstract}

\maketitle

\setcounter{tocdepth}{2}

\tableofcontents

\section{Introduction}\label{sec:intro}

An extensive and still-expanding body of empirical evidence documents that many real-world datasets   {deviate markedly from Gaussian and sub-Gaussian behaviour} and instead exhibit pronounced heavy tails, by which we mean distributions whose survival functions decay only polynomially and whose higher-order moments may fail to exist. Representative examples include financial returns and insurance losses \cite{FR03}, degree distributions of large-scale complex networks \cite{EoJo15}, and high-throughput gene-expression measurements \cite{WPL15}. In contrast with the sub-Gaussian case, heavy-tailed data frequently contain a non-trivial proportion of outliers---observations of magnitude orders larger than the bulk of the sample---which seriously degrades the finite-sample performance of classical procedures such as the sample mean or the ordinary-least-squares estimator. These empirical findings have motivated a sustained and rapidly growing research effort devoted to the development of   {robust methodologies} that retain essentially sub-Gaussian concentration guarantees even when the underlying data are themselves non-sub-Gaussian; see, for example, \cite{Ca12,SZF20}.

Robust methods provide a foundational toolkit for the principled handling of outliers, and are therefore versatile across diverse data structures, statistical models, and application domains. Their scope is by no means restricted to the classical setting of independent and identically distributed observations: robust techniques have been successfully adapted to a variety of   {dependent data regimes}, including autoregressive, exponential-smoothing, wavelet-based, and additive time-series models \cite{CMA94,CGM10,GMVX22,WLXXZ24,WZHCT21}. Robustness principles also extend far beyond simple parametric models: they have proved effective in   {high-dimensional regression} under sparsity \cite{FLW17,LM19} and in   {matrix- and tensor-valued estimation problems} \cite{AY22,FWZ21,KMRSZ19,Mi18}.

As a consequence, robust estimation has been broadly adopted in numerous fields, ranging from machine learning \cite{BGS22,FWZZ21,LL20,PKH22} to finance and econometrics \cite{FWZ19,Qu21}. For a comprehensive treatment of the theoretical foundations and methodological developments in robust statistics we refer the reader to the surveys and monographs of \cite{LM19-2,KMRSZ19}.

\subsection{Huber and Catoni estimations}\label{subsec:lit}

The modern theory of robust estimation traces its origin to the seminal contributions of   {Huber in the 1960s} \cite{Hu64,Hu73}. In place of the classical quadratic loss $x^{2}/2$, Huber proposed the   {piecewise quadratic--linear loss} $\ell_{\tau}$ defined by
\begin{equation}\label{eq:huber-loss}
\ell_{\tau}(x) \;:=\;
\begin{cases}
\frac{1}{2}\,x^{2}, & |x|\le \tau, \\[2pt]
\tau |x| - \dfrac{\tau^{2}}{2}, & |x|>\tau,
\end{cases}
\end{equation}
where $\tau>0$ is a   {robustification (tuning) parameter}. The function $\ell_{\tau}$ coincides with one-half the squared loss on the central interval $[-\tau,\tau]$ and degenerates to an affine absolute-value loss outside that interval; it is known as   {Huber's loss}, and the associated $M$-estimation procedure is referred to as   {Huber's estimation}.

The design philosophy underlying Huber's loss is to   {cap the influence of any observation whose magnitude exceeds the cutoff} $\tau$: inside the central region the estimator behaves like the ordinary least-squares estimator, whereas outside that region the score function saturates and the observation contributes only through its sign and a bounded factor. The choice of $\tau$ is therefore critical for balancing bias against robustness. A   {large} value of $\tau$ preserves more information from extreme observations and hence yields small bias but weak protection against outliers; conversely, a   {small} value of $\tau$ suppresses the tail information and thereby enhances robustness at the cost of increased bias. The asymptotic distributional theory of the Huber estimator has been comprehensively developed \cite{Hu73,Ma89,YM79}. For the calibration of $\tau$, Huber and Ronchetti \cite{HR11} suggested the heuristic rule $\tau = 1.345\,\sigma$, which yields $95\%$ asymptotic efficiency at the Gaussian model, with $\sigma$ denoting the error standard deviation. More recently, however, Sun et al. \cite{SZF20} demonstrated through a careful non-asymptotic analysis that, in the context of linear regression,  {optimal finite-sample performance requires $\tau$ to be chosen adaptively} as a function of the sample size $n$, the ambient dimension $d$, and the moments of the error distribution. Further contributions on tuning and extensions of Huber-type estimation can be found in \cite{AVP20,JP22,PSBR20,WZHCT21,ZBFL18}.

In a seminal contribution, Catoni \cite{Ca12} initiated a   {non-asymptotic deviation-theoretic} study of robust mean estimation for heavy-tailed data. His analysis revealed, for the first time, that one can obtain sub-Gaussian-type deviation guarantees under only a finite-variance assumption on the sampling distribution. More precisely, Catoni considered a class of non-decreasing influence functions $\psi_{1}$ satisfying the two-sided logarithmic envelope
\begin{equation}\label{catoni c}
-\log\!\left(1 - x + \frac{|x|^{2}}{2}\right)
\;\le\; \psi_{1}(x) \;\le\;
\log\!\left(1 + x + \frac{|x|^{2}}{2}\right),
\end{equation}
and defined his   {robust mean estimator} $\widehat{\theta}$ implicitly as the root of the estimating equation
\begin{equation}\label{catoni}
\sum_{i=1}^{n} \psi_{1}\!\bigl(\alpha (X_{i}-\theta)\bigr) \;=\; 0,
\end{equation}
where $\alpha$ is a robustification parameter calibrated as $\alpha \propto \sigma^{-1}\sqrt{\log(\epsilon^{-1})/n}$. This parameter plays the same role in the Catoni construction as $\tau$ does in Huber's loss, in that it controls the bias--robustness trade-off. According to \cite[Proposition 2.4]{Ca12}, the resulting $(1-\epsilon)$-confidence interval has length of order $O\!\left(\sigma\sqrt{\log(\epsilon^{-1})/n}\right)$, which   {matches the optimal sub-Gaussian rate} achieved by the ordinary sample mean under genuinely light-tailed data.

Compared with Huber's estimator, Catoni's construction \eqref{catoni} is considerably more flexible, since the entire family of functions $\psi_{1}$ satisfying \eqref{catoni c} is admissible. At one extreme, the   {narrowest} (tightest) admissible choice,  
\begin{equation} \label{e:ExpPsi1}
\psi_{1}(x) =
\begin{cases}
-\operatorname{sign}(x)\,\log\!\left(1-|x|+\dfrac{x^{2}}{2}\right), & |x|\le 1,\\[2pt]
\operatorname{sign}(x)\,\log 2, & |x|>1,
\end{cases}
\end{equation}
produces an estimator that behaves qualitatively like Huber's, in that it effectively truncates the contribution of extreme observations. At the opposite extreme, the   {widest} admissible choice,
\begin{equation} \label{eq:widestcatoni}
\psi_{1}(x) =
\begin{cases}
\log\!\left(1+x+\dfrac{x^{2}}{2}\right), & x\ge 0,\\[2pt]
-\log\!\left(1-x+\dfrac{x^{2}}{2}\right), & x<0,
\end{cases}
\end{equation}
preserves information from large-magnitude observations but only on a   {logarithmic scale}, so that outliers still influence the estimator without overwhelming it. This flexibility allows the practitioner to select a $\psi_{1}$ lying anywhere in this spectrum, tailored to the tail heaviness of the data at hand.

Catoni's framework has been widely adopted in a broad array of models and data scenarios (see, e.g., \cite{BCL13,FLW18,FKSZ19,JSF18,Mi18,SZF20,WLXXZ24,WR23}). For distributions possessing only a finite $\beta$-th moment with $\beta\in(1,2)$, Chen et al. \cite{CJLX21} extended Catoni's construction by replacing the quadratic envelope \eqref{catoni c} with the $\beta$-power envelope
\begin{equation} \label{catoni c2}
-\log\!\left(1 - x + \frac{|x|^{\beta}}{\beta}\right)
\;\le\; \psi_{2}(x) \;\le\;
\log\!\left(1 + x + \frac{|x|^{\beta}}{\beta}\right),
\end{equation}
thereby making Catoni-type estimators available under substantially weaker moment assumptions than the classical finite-variance case.

\subsection{Tuning-free estimation via coupled estimating equations}
\label{sec:tuning-free-review}
A common thread runs through several recent advances in robust and
high-dimensional statistics: rather than selecting a tuning parameter by an
external rule (cross-validation, plug-in scale estimates, theoretical
constants depending on unknown nuisance quantities), one designs a
{system of two coupled estimating equations} whose joint solution
returns both the target parameter and the nuisance scale at once. The
resulting estimators are tuning-free in the sense that no scale-dependent
tuning has to be chosen by the user; the data themselves determine the
scale through one of the two equations.

\textit{High-dimensional regression: scaled and square-root Lasso.}
The first influential instance of this paradigm in modern high-dimensional
statistics is the scaled Lasso of Sun and Zhang~\cite{SZ12} in 2012, who introduced
a jointly convex penalized loss in the regression coefficient vector
$\beta\in\mathbb{R}^{p}$ and the noise scale $\sigma>0$. Its minimizer
$(\hat\beta,\hat\sigma)$ is characterized by the coupled
Karush--Kuhn--Tucker system
\[
  \frac{\partial}{\partial\beta}\,L(\beta,\sigma) \;=\; 0,
  \qquad
  \frac{\partial}{\partial\sigma}\,L(\beta,\sigma) \;=\; 0,
\]
and the scale equation in $\sigma$ replaces the usual variance-dependent
tuning of the Lasso penalty. The closely related square-root Lasso of
Belloni et al. ~\cite{BCL11} achieves the same goal by
penalizing the root mean squared residual; algorithmically distinct, it is
again the joint solution of a two-equation system in $(\beta,\sigma)$.
Subsequent developments along this line, including
\cite{BMGS19,JXX19,Mol22}, refine the analysis, broaden the loss class, and
extend the framework to structured-sparsity and generalized-linear models,
but all retain the same two-equation backbone.

\textit{Heavy-tailed mean estimation}.
Lee and Valiant~\cite{LeVa22-0} construct an estimator of the
mean of a real-valued distribution that, given only finite variance
$\sigma^{2}$, achieves
$|\hat\mu-\mu|
  \;\le\;
  \sigma\sqrt{\tfrac{2\ln(1/\delta)}{n}}\,(1+o(1))$
with probability at least $1-\delta$, matching the Cram\'er--Chernoff
benchmark down to the leading constant. Their construction is a weighted
trimmed mean in which the trimming budget is fixed and the weights solve
\begin{equation} \label{e:LeeVal}
  \sum_{i=1}^{n}\min\!\bigl\{\,1,\;\lambda\,(X_{i}-\mu_{0})^{2}\,\bigr\}
  \;=\;\tfrac{1}{3}\log(1/\delta),
  \qquad
  \sum_{i=1}^{n} w_{i}\,(X_{i}-\hat\mu) \;=\; 0,
\end{equation}
where $X_{1},\dots,X_{n}$ is the given observations and the first equation determines the data-adaptive scale $\lambda$ and
the second is a weighted-mean condition on $\hat\mu$. The follow-up
\cite{LeVa22-1} extends the program to $\mathbb{R}^{d}$, again replacing
explicit knowledge of the covariance by a coupled scale-and-location
selection step. Crucially, both papers estimate the mean {without}
first tuning to the variance: the second equation in the pair plays the
role that an external estimate of $\sigma^{2}$ would otherwise have.

Sun~\cite{Su21} proposed a tuning-free method based on a new pseudo-Huber loss that enables
{simultaneous} estimation of the mean $\mu$ and the variance
$\sigma^{2}$. Sun considered the
minimization problem
\begin{equation}\label{e:PsMin}
  (\hat\mu,\hat v)
  \;=\;
  \mathop{\mathrm{argmin}}_{\mu,v}\,L(\mu,v),
  \qquad
  L(\mu,v)=\sum_{i=1}^{n}\!\left[\psi_{\mathrm{ps}}\!\left(\frac{\alpha(X_{i}-\mu)}{v}\right)v+\lambda v\right],
\end{equation}
where $\psi_{\mathrm{ps}}(x)=\sqrt{1+x^{2}}-1$ is the pseudo-Huber loss and
$\alpha>0$, $\lambda>0$ are adjustment factors. The pair $(\hat\mu,\hat v)$
is recovered as the joint solution of the first-order conditions
\[
  \frac{\partial}{\partial\mu}\,L(\mu,v) \;=\; 0,
  \qquad
  \frac{\partial}{\partial v}\,L(\mu,v) \;=\; 0,
\]
and serves as a tuning-free joint estimator of $(\mu,\sigma)$.

\textit{Heavy-tailed simultaneous estimation via Huber-type splitting.}
An alternative tuning-free proposal is due to Wang et al. ~\cite{WZZZ21}, who introduced a new Huber-type estimator of the
robustification parameter $\tau$ under the assumption that the location
parameter $\mu$ is known. Since $\tau$ is, up to constants depending on
$(n,d)$, proportional to $\sigma$, estimating $\tau$ is equivalent to
estimating $\sigma^{2}$. Building on this
construction and on the adaptive Huber mean estimator of Sun et al. ~\cite{SZF20}, they developed a {splitting algorithm} that
alternates between updating $\tau$ and $\mu$. Intuitively, the output of
this iterative scheme converges to a solution of a coupled system of
estimating equations in $(\mu,\tau)$: 
\begin{equation} \label{e:WZZZ}
\sum_{i=1}^n \psi_\tau(X_i - \theta) = 0, \quad \quad  \sum_{i=1}^n \min\{(X_i - \theta)^2, \tau^2\}/\tau^2 - z = 0, 
\end{equation}
 where $z$ can be explicitly represented by $n$ and the error tolerance $\epsilon$. 
The
same idea was extended to the linear regression setting. However, the
convergence of the splitting algorithm itself was not rigorously
established therein, and no non-asymptotic guarantees were provided for the
joint estimator. Filling this gap---establishing convergence of the
$(\mu,\tau)$ iterates and proving high-probability deviation bounds on
their joint limit---remains open.

Across these three threads, the unifying methodological principle is
the same: encode the scale tuning as a second estimating equation, solve
the resulting two-equation system jointly. The remaining technical questions in this literature are largely
about the joint analysis (existence and non-asymptotic deviation control of the coupled solution), are notoriously difficult in many cases.

\subsection{Our contributions and methods}\label{subsec:contrib}

This paper proposes a   {Catoni-type joint robust estimation} procedure for three parametric statistical models under heavy-tailed data: univariate mean estimation, classical linear regression, and $\ell_{2}$-penalized linear regression. In each model, both the trend parameter (denoted generically by $\theta$) and the error variance $\sigma^{2}$ are unknown and are estimated jointly. Our approach is based on a   {system of two coupled Catoni-type estimating equations}---one identifying the trend and the other identifying the variance (see Section \ref{Sec:2} for details). The main contributions of the paper can be summarized as follows.

\smallskip
\noindent\textbf{(i) A general framework for coupled joint (tuning-free) estimating equations.}\enspace
 Our framework provides a substantially broader formulation than existing tuning-free proposals \cite{LeVa22-0,Su21,SZ12,WZZZ21}. The methods of \cite{Su21,SZ12} proceed by minimizing a scalar loss function $L(\bm\theta,v)$ in which $\bm\theta$ estimates the target parameter and $v^{2}$ approximates the variance; since $L$ is differentiable, the first-order condition
\begin{equation}\label{eq:grad-L}
\nabla L(\bm\theta,v) \;=\; 0
\end{equation}
 where $\nabla$ denotes the joint gradient with respect to $(\bm\theta,v)$. Constructing such a scalar loss, however, is highly restrictive: the two parameters of heterogeneous nature must enter $L$ in a tightly coupled form, which restricts the possibility of applying different estimation strategies to parameters of different nature. In contrast, our two coupled equations are not required to be the gradient of any scalar loss, which substantially enlarges the admissible design space. Equally importantly, the framework imposes minimal structural restrictions on the influence function $\psi$: although we instantiate it with Catoni-type functions in this paper, the proposed methodology and the underlying analytical strategy carry over with only minor modifications to Huber-type and other Huber-like influence functions used in \cite{Su21,WZZZ21}. Moreover, our framework imposes no special structural assumptions on the estimator, making it applicable to a wider variety of models. In contrast, the method in \cite{LeVa22-0} exploits the affine invariance of the median-of-means construction, which limits its applicability beyond the mean estimation setting.

\smallskip
\noindent\textbf{(ii) Joint robust estimation under heavy-tailed noise.}\enspace
 Our framework simultaneously delivers robust estimates of both the target parameter and the noise variance under heavy-tailed data. Specifically, under a finite $2\beta$-th moment assumption with $\beta\in(1,2]$, we establish non-asymptotic, sub-Gaussian-type high-probability error bounds that hold jointly for the target parameter and the variance estimator, attaining the optimal rates—matching, up to absolute constants, those available to oracle procedures that know the variance in advance (see the main Theorems in Section \ref{Sec:2} for details). 

\smallskip
\noindent\textbf{(iii) A new analytical machinery via the Poincar\'e--Miranda theorem.}\enspace
We rigorously establish the   {consistency} of all three Catoni-type joint robust estimators using a novel analytical strategy. Because the coupled estimating equations are highly non-linear in the parameters, the classical approach that relies on convex $M$-estimation is unavailable. To circumvent this obstacle we invoke the   {Poincar\'e--Miranda theorem} \cite{Fr18}---a multidimensional generalization of the Bolzano intermediate-value theorem---to certify that any solution of the system must lie inside an   {explicit geometric region} (a cylinder or a cone) centred at the true parameter value. The verification that the Poincar\'e--Miranda hypotheses are satisfied with high probability is itself a delicate probabilistic construction that combines (a) the derivation of auxiliary algebraic comparison inequalities, (b) a refined coupling and comparison framework, and (c) sharp supremum bounds along the boundaries of the geometric regions. The resulting machinery provides a   {feasible and broadly applicable framework for the non-asymptotic analysis of coupled estimation problems}, and we expect it to be useful in many other statistical contexts.

\smallskip
The remainder of the paper is organized as follows. Section~2 introduces the Catoni-type joint robust estimators for the three settings---mean estimation, linear regression, and ridge-penalized linear regression in high dimensions---and states the corresponding non-asymptotic consistency theorems. Section~3 reports extensive simulation studies and two real-data analyses. Our numerical experiments show that the proposed Catoni-type joint robust estimators   {consistently outperform} benchmark robust procedures, including the classical Catoni estimator and the adaptive Huber estimator of \cite{SZF20}, and that this performance gap   {widens as the heaviness of the tails increases}. 
The detailed proofs of the main results in Section \ref{Sec:2} are deferred to the appendices. Appendix A collects two auxiliary lemmas: an algebraic comparison lemma used to characterize the regimes of solutions to the proposed estimators, and a maximal inequality for bounding exponential moments. Appendices \ref{Sec:AppB}–\ref{Sec:AppD} contain the proofs of Theorems \ref{thm:joint theorem}–\ref{Thm:jointly ridge}, together with outlines of their proof strategies.

\subsection{Notation}\label{subsec:notation}

We collect here the notational conventions that will be used throughout the paper. For $p\in[1,\infty]$ and a vector $\bm\theta=(\theta_{1},\dots,\theta_{d})\in\R^{d}$, the   {$\ell_{p}$-norm} of $\bm\theta$ is defined as
\begin{equation*}
\norm{\bm\theta}_{p} \;:=\; \left(\sum_{i=1}^{d} |\theta_{i}|^{p}\right)^{1/p}
\quad \text{for } p\in[1,\infty),
\qquad
\norm{\bm\theta}_{\infty} \;:=\; \max_{1\le i\le d} |\theta_{i}|.
\end{equation*}
We write $\sphere^{d-1} := \{\bm\theta\in\R^{d}:\norm{\bm\theta}_{2}=1\}$ for the Euclidean unit sphere in $\R^{d}$. The symbol $C(\cdot)$ will be used generically for a positive constant whose value may change from line to line and that depends only on the quantities listed as its arguments.

{\bf Acknowledgement:} We would like to thank Zhijun Cai for his careful reading of early drafts and for providing numerous helpful comments and corrections.


\section{Catoni type Joint robust estimations} \label{Sec:2}

In this section, we introduce three Catoni type joint robust estimation methods: joint robust mean–variance estimation, joint robust regression, and  joint robust ridge regression, presented in Sections \ref{ss:2-1}–\ref{ss:2-3}, respectively.

\subsection{Catoni type Joint robust mean and variance estimation} \label{ss:2-1}

Let $X_{1},\dots,X_{n}$ be i.i.d. random variables drawn from an unknown distribution $\mathbb{P}$ on $\mathds{R}$ satisfying
\begin{align} \label{Xmoment}
    \mathbb{E}[X_{i}]=\mu,\quad \mathbb{E}|X_{i}-\mu|^{2}=\sigma^{2} \quad \text{and} \quad \mathbb{E}|X_{i}-\mu|^{2\beta}=m_{2\beta}<\infty,
\end{align}
for some $\beta \in (1,2]$. Our goal is to estimate $\mu$ and $\sigma$ simultaneously.

For a given confidence level $\epsilon\in (0,\tfrac{1}{6})$, we consider the following coupled equations for $(\theta, v)$:
\begin{align} \label{f1f2}
    \begin{cases} &f_{1}(\theta,v)=\frac{1}{n}\sum_{i=1}^{n} \psi_{1}\left(\alpha_{1} \frac{(X_{i}-\theta)}{v}  \right)=0,  \\ &f_{2}(\theta,v)=\frac{1}{n}\sum_{i=1}^{n} \psi_{2}\left(\alpha_{2} \left(\frac{(X_{i}-\theta)^{2}}{v^{2}}-1 \right)\right)=0
    \end{cases},
\end{align}
where
\begin{align} \label{def:alpha1,2}
    \alpha_{1}:=\sqrt{\frac{2\log\left(\epsilon^{-1}\right)}{n}},\quad \alpha_{2}:=\left(\frac{\log(\epsilon^{-1})}{n}\right)^{\frac{1}{\beta}},
\end{align}
and $\psi_{1},\psi_{2}$ are continuous non-decreasing functions satisfying \eqref{catoni c} and \eqref{catoni c2}, each with a root at $0$.
If we take $\psi_1(x)=x$ and $\psi_2(x)=x$, then the corresponding equations in \eqref{f1f2} become the well-known moment equations and immediately yield the classical sample mean and variance estimators. Conditions in \eqref{catoni c} and \eqref{catoni c2}   for the two functions can help reduce the effects of the data far from the true mean, while $\alpha_1$ and $\alpha_2$ determine a scale for such reduction.

We show that \eqref{f1f2} admits a solution $(\hat \theta, \hat v)$ that consistently estimates $(\mu, \sigma)$. It is evident that there is no loss function $F(\theta,v)$ such that
\[
\nabla F(\theta,v) = [f_1(\theta,v), f_2(\theta,v)]',
\]
that is, $[f_1(\theta,v), f_2(\theta,v)]'$ is not the gradient of any function. Thus, \eqref{f1f2} cannot be solved by minimizing a loss function. Instead, we develop a delicate comparison argument to establish that \eqref{f1f2} admits a solution converging to the true value $(\mu, \sigma)$.

\begin{assu} \label{assumption:c}
    Assume that $\psi_{1}$ is twice differentiable except at finitely many points, and its second derivative satisfies $|\psi_{1}''(x)|\le M\min\{1,1/x^{2}\}$ wherever it exists for some positive constant $M$. 
\end{assu}
\begin{rem} 
    Typical examples satisfying the above assumption include the widest and narrowest Catoni influence functions (see \eqref{eq:widestcatoni} and \eqref{e:ExpPsi1}), as well as the derivative of the Huber loss. These functions are non-smooth at no more than two points. This assumption is primarily imposed to ensure the accuracy of the estimation of $\theta$, and we do not place any further smoothness restrictions on $\psi_{2}$, which is used for variance estimation.
\end{rem}

\begin{thm} \label{thm:joint theorem}
Suppose Assumption \ref{assumption:c} holds. Consider the joint estimation \eqref{f1f2} with any confidence level $\epsilon\in (0, {1}/{6})$. There exists a constant $C=C(m_{2\beta},\sigma,\beta)$ such that, for all $n \ge C[\log(\epsilon^{-1})+1]$, \eqref{f1f2} admits a solution $(\hat{\theta},\hat{v})$ satisfying, with probability at least $1-6\epsilon$:
\begin{align*}
\big|\hat{\theta}-\mu \big|\le \sigma \left(1+c_{1}\left(\tfrac{\log(\epsilon^{-1})}{n} \right)^{\tfrac{\beta-1}{\beta}}\right) \sqrt{\tfrac{2\log(\epsilon^{-1})}{n}},
\end{align*}
and
\begin{align*}
\left(1+c_{0}\left(\tfrac{\log(\epsilon^{-1})}{n} \right)^{\tfrac{\beta-1}{\beta}} \right)^{-\tfrac{1}{2}} \sigma
\le \hat{v}\le
\left(1-c_{0}\left(\tfrac{\log(\epsilon^{-1})}{n} \right)^{\tfrac{\beta-1}{\beta}} \right)^{-\tfrac{1}{2}} \sigma,
\end{align*}
where $c_{0}=c_{0}(m_{2\beta},\sigma,\beta)$ and $c_{1}=c_{1}(M,\sigma)$ do not depend on $n$ or $\epsilon$.
\end{thm}

The theorem above shows that our approach can estimate both the mean and variance with only logarithmic dependence on $1/\epsilon$, a property unattainable by non-robust methods such as the sample mean when only a $2\beta$-moment exists. Furthermore, our result demonstrates that the variance is estimated at a rate depending on the availability of higher moments, thereby exhibiting a form of adaptivity in variance estimation too. In contrast, although \cite{Su21} requires only a weaker second-moment condition, it does not provide variance estimation at any predictable rate.

The proof of Theorem \ref{thm:joint theorem} is provided in Appendix \ref{Sec:AppB}.

\subsection{Catoni type Joint robust regression} \label{ss:2-2}

Consider i.i.d. data samples $(y_{1},\boldsymbol{x}_{1}),\dots,(y_{n},\boldsymbol{x}_{n})\in \mathds{R}\times \mathds{R}^{d}$ following the linear model
\begin{align} \label{linear model}
    y_{i}=\<\boldsymbol{x}_{i},\boldsymbol{\theta}^{*}  \>+\varepsilon_{i},
\end{align}
where $\varepsilon_{1},\dots,\varepsilon_{n}$ are i.i.d. noise terms satisfying
\begin{align} \label{noise2}
    \mathbb{E}[\varepsilon_{i}|\boldsymbol{x}_i]=0,\quad \mathbb{E}[\varepsilon_{i}^{2}|\boldsymbol{x}_i]=\sigma^{2}\quad  \text{and} \quad m_{2\beta}=\mathbb{E}\left[|\varepsilon_{i}|^{2\beta}|\boldsymbol{x}_i\right]<\infty,\quad \text{for some}\quad \beta\in (1,2],
\end{align}
and $\boldsymbol{\theta}^{*}\in \mathds{R}^{d}$ denotes the true regression coefficient. Let
\begin{equation} \label{e:LMax}
 L=\max_{1 \le i \le n} \|\boldsymbol{x}_i\|_2<\infty.
\end{equation}
Our goal is to estimate $\boldsymbol{\theta}^*$ and $\sigma$ simultaneously.

For a given confidence level $\epsilon\in (0,\frac{1}{6})$, consider the following coupled equations for $\boldsymbol{\theta}$ and $v$:
\begin{align} \label{f1f2regression}
    \begin{cases}
         &\Tilde{f}_{1}(\boldsymbol{\theta},v)=\frac{1}{n}\sum_{i=1}^{n} \boldsymbol{x}_{i}\psi_{1}\left(\alpha_{1} \frac{(y_{i}-\<\boldsymbol{x}_{i},\boldsymbol{\theta}\>)}{v} \right)=\boldsymbol{0},  \\ &\Tilde{f}_{2}(\boldsymbol{\theta},v)=\frac{1}{n }\sum_{i=1}^{n}  \psi_{2}\left(\alpha_{2}\left( \frac{(y_{i}-\<\boldsymbol{x}_{i},\boldsymbol{\theta}\>)^{2}}{v^{2}}-1 \right)\right)=0,
    \end{cases}
\end{align}
where
\begin{align} \label{def:alpha1,2 re}
     \alpha_{1}:=\sqrt{\frac{2\log (d\epsilon^{-1})}{n}},\alpha_{2}:=\left(\frac{\log(\epsilon^{-1})}{n}\right)^{\frac{1}{\beta}},
\end{align}
and $\psi_{1},\psi_{2}$ are non-decreasing continuous functions satisfying \eqref{catoni c} and \eqref{catoni c2}, each with a root at $0$.
 Similar to the intuitive interpretation in Section \ref{ss:2-1}, if we take $\psi_1(x)=x$ and $\psi_2(x)=x$, then the equation \eqref{f1f2regression} immediately yields the classical linear regression. The conditions \eqref{catoni c} and \eqref{catoni c2} and the choices of $\alpha_1$ and $\alpha_2$ reduce the effect of the data far from the true values and thus produce robustness.

We will show that \eqref{f1f2regression} admits a solution $(\hat{\boldsymbol{\theta}}, \hat v)$ that consistently estimates $(\boldsymbol{\theta}^*, \sigma)$. Note that there exists no loss function $\Tilde{F}(\boldsymbol{\theta},v)$ such that $\nabla \Tilde{F}(\boldsymbol{\theta},v)=[\Tilde{f}_1(\boldsymbol{\theta},v), \Tilde{f}_2(\boldsymbol{\theta},v)]'$, so the coupled equations cannot be solved via standard loss minimization. Instead, we employ a delicate comparison argument, which can be of independent  interest.

\begin{assu} \label{A1}
The empirical Gram matrix
\begin{equation} \label{e:GramMat}
\mathbf{S}_{n}:=\frac{1}{n}\sum_{i=1}^{n} \boldsymbol{x}_{i}\boldsymbol{x}_{i}^{T}
\end{equation}
is nonsingular. Moreover, there exist constants $c_{l}, c_{u}>0$ such that
\begin{equation} \label{e:clcuS}
c_l \leq \lambda_{\min}(\mathbf{S}_n) \leq \lambda_{\max}(\mathbf{S}_n) \leq c_u,
\end{equation}
where $\lambda_{\min}\left(\mathbf{S}_n\right)$ and $\lambda_{\max}\left(\mathbf{S}_n\right)$ denote the minimal and maximal eigenvalues of $\mathbf{S}_n$.
\end{assu}

\begin{rem}

    Although Assumption~\ref{A1} is stated in a deterministic form, a high-probability or other random-design variant (e.g., \cite{VBRD14}) does not affect our main results. Since our focus is on heavy-tailed noise, this assumption avoids detailed discussions on spectral properties of $\mathbf{S}_n$ under random designs. For further reference, see \cite{Ve10,Ve18}.
\end{rem}

\begin{thm} \label{thm:jointly main regression}
Consider \eqref{f1f2regression} with confidence level $\epsilon\in (0, {1}/{6})$, and suppose Assumption \ref{assumption:c} and \ref{A1} hold. There exists a constant $C=C(m_{2\beta},\sigma,\beta,c_{l}, c_u, L,M)$ such that, for
$$n \ge C d\left[\log(\epsilon^{-1})+1\right],$$
with probability at least $1-6\epsilon$, \eqref{f1f2regression} admits a solution $(\hat{\boldsymbol{\theta}},\hat{v})$ satisfying
    \begin{align*}
        \|\hat{\boldsymbol{\theta}}-\boldsymbol{\theta}^{*}\|_{2}\le C_{1}\sigma \sqrt{\frac{d\log\left(d\epsilon^{-1} \right)}{n}},
    \end{align*}
    and
    \begin{align*}    \left(1+C_{2}\left(\sqrt{d}\alpha_1+\alpha_2^{\beta-1}\right)\right)^{-\frac{1}{2}} \sigma\le \hat{v}\le \left(1-C_{2}\left(\sqrt{d}\alpha_1+\alpha_2^{\beta-1}\right)\right)^{-\frac{1}{2}} \sigma,
    \end{align*} 
where $C_{1}=C_1(c_{l}, L)$, $C_{2}=C_2(m_{2\beta},\sigma,\beta,c_{l}, c_u, L)$, and note
    \begin{align*}
        \sqrt{d}\alpha_1+\alpha_2^{\beta-1}=\sqrt{\frac{d\log\left(d\epsilon^{-1} \right)}{n}}+\left( \frac{\log(\epsilon^{-1})}{n}\right)^{\frac{\beta-1}{\beta}}.
    \end{align*} 
\end{thm}

The proof of Theorem \ref{thm:jointly main regression} is provided in Appendix \ref{Sec:AppC}.

\subsection{Catoni type Joint robust ridge regression} \label{ss:2-3}

Recall the linear regression in Section \ref{ss:2-2}. We now consider this problem in the context of penalized regression models, which are widely used in high-dimensional settings to improve estimation accuracy and interpretability. Our proposed method can be naturally extended to such penalized frameworks. 

 Specifically, we focus on the case with an $\ell_2$-penalty (ridge regression). For a given confidence level $\epsilon\in (0,\frac{1}{6})$, 
 consider the following coupled equation about $({\boldsymbol{\theta}},{v})$: 
\begin{align} \label{f1f2ridge}
    \begin{cases}
         &\Tilde{f}_{1}^{ridge}(\boldsymbol{\theta},v)=\frac{1}{n}\sum_{i=1}^{n}  \boldsymbol{x}_{i}\psi_{1}\left(\alpha_{1}\frac{\left(y_{i}-\<\boldsymbol{x}_{i},\boldsymbol{\theta}\>\right)}{v}  \right)-\lambda  \boldsymbol{\theta}=\boldsymbol{0},  \\ &\Tilde{f}_{2}^{ridge}(\boldsymbol{\theta},v)=\frac{1}{n}\sum_{i=1}^{n} \psi_{2}\left(\alpha_{2}\left(\frac{\left(y_{i}-\<\boldsymbol{x}_{i},\boldsymbol{\theta}\>\right)^{2}}{v^{2}}-1 \right)\right)=0,
    \end{cases}
\end{align}
where $\alpha_{1}$, $\alpha_{2}$ are defined as in \eqref{def:alpha1,2 re},  $\lambda=\lambda_0 \alpha_1$ and $\lambda_0$ is the penalty coefficient which will be determined later. Our assumptions remain the same as in Section \ref{ss:2-2}.  We can intuitively interpret the equation \eqref{f1f2ridge} in a way similar to that in Section \ref{ss:2-2}, here the term $\lambda \bm{\theta}$ is obtained from differentiating the penalty $\frac{\lambda}2 \|\bm \theta\|^2_2$.

We shall show that \eqref{f1f2ridge} admits a solution $(\hat {\boldsymbol{\theta}}, \hat v)$ which is a consistent estimator for $(\boldsymbol{\theta}^*, \sigma)$. Similar to the previous two cases, we cannot solve the equation \eqref{f1f2ridge} by minimizing a loss function and need to use a very delicate comparison argument.  To more clearly illustrate the impact of the minimum eigenvalue of the Gram matrix on the estimator, we explicitly present the dependence on $c_l$ (defined in \eqref{e:clcuS}) in the following theorem.

\begin{thm} \label{Thm:jointly ridge}
    Recall \eqref{f1f2ridge} with a confidence level $\epsilon\in (0, {1}/{6})$. Let Assumption \ref{assumption:c} and \ref{A1} hold. There exists a constant $C=C(m_{2\beta},\sigma,\beta,c_u,L,M)$ such that, as long as 
    \begin{align*}
        n\ge Cd(c_l+\lambda_0)^{-1}[\log(\epsilon^{-1})+1],
    \end{align*}
     with probability at least $1-6\epsilon$, \eqref{f1f2ridge} admits a solution $(\hat{\boldsymbol{\theta}},\hat{v})$ which satisfies the following error bounds:
    \begin{align*}
        \|\boldsymbol{\theta}^*-\hat{\boldsymbol{\theta}}\|_{2}\le C_1\left( \frac{\sqrt{d}L}{\lambda_0+c_l/\sigma}\sqrt{\frac{\log (d \epsilon^{-1})}{n}}+\frac{\lambda_0}{\lambda_0+c_l/\sigma}\|\boldsymbol{\theta}^*\|_{2}\right),
    \end{align*}
    and
    \begin{align*}
        \left(1+C_2\left(\alpha_{2}^{\beta-1}+\frac{\theta_0}{\theta_0+\frac{1}{2}}\right)\right)^{-\frac{1}{2}}\sigma\le \hat{v}\le  \left(1-C_2\left(\alpha_{2}^{\beta-1}+\frac{\theta_0}{\theta_0+\frac{1}{2}}\right) \right)^{-\frac{1}{2}}\sigma,
    \end{align*}
    with $C_{1}=C_1(c_{l}, L),C_{2}=C_2(m_{2\beta},\sigma,\beta, c_u, L)$ and 
    \begin{align*}
        \theta_0 = C_1\left( \frac{\sqrt{d}L}{\lambda_0+c_l/\sigma}\sqrt{\frac{\log (d \epsilon^{-1})}{n}}+\frac{\lambda_0}{\lambda_0+c_l/\sigma}\|\boldsymbol{\theta}^*\|_{2}\right).
    \end{align*}
    In particular, if $c_l>0$, then taking $\lambda_0 = O\left( L\sqrt{\frac{d\log (d \epsilon^{-1})}{n}} \right)$ gives us
    \begin{align*}
         \|\boldsymbol{\theta}^*-\hat{\boldsymbol{\theta}}\|_{2}\le O\left(L\sqrt{\frac{d\log (d \epsilon^{-1})}{n}}\right).
    \end{align*}
\end{thm}
The proof of Theorem \ref{Thm:jointly ridge} will be given in Appendix \ref{Sec:AppD}.

\section{Simulations}\label{sec:simulations}

In this section we report three numerical experiments that examine the
finite-sample performance of the joint Catoni-type estimators introduced in
Section~\ref{Sec:2}.  The experiments are designed to parallel the
three estimation problems analysed in Theorems~\ref{thm:joint theorem}, ~\ref{thm:jointly main regression}, and ~\ref{Thm:jointly ridge}:
(i)~joint mean--variance estimation, (ii)~joint regression, and
(iii)~joint $\ell_2$-penalised (ridge) regression.  In every experiment the
performance of each estimator is assessed through two criteria:
(i)~the $\gamma$-quantile of the estimation error at levels
$\gamma\in[0.9,1]$, and (ii)~the dependence of the $99\%$ quantile of the
estimation error on a parameter characterising the tail heaviness of the
noise distribution (degrees of freedom, shape parameter, or scale).

Throughout the simulations we set the confidence level $\varepsilon=0.01$
and let the noise be generated from one of the following four
distributions:
\begin{itemize}
  \item the normal distribution;
  \item the (non-central) Student's $t$ distribution;
  \item the Generalised Pareto distribution;
  \item the Burr Type XII distribution.
\end{itemize}
The latter three are widely used heavy-tailed benchmarks.  The
configuration of each distribution will be specified case by case.  All
Monte Carlo averages reported below are based on independent repetitions
with sample size $n=500$.  

\subsection{Numerical experiment 1: mean estimation}
\label{sec:sim-mean}

In this experiment we evaluate the error stability of three Catoni-type
robust estimators of the mean:
\begin{enumerate}
  \item the Catoni estimator with the true variance $\sigma^2$;
  \item the Catoni estimator with the sample variance $\bar\sigma^2$
        replacing $\sigma^2$;
  \item our joint Catoni mean--variance estimator defined by
        equation~\eqref{f1f2}.
\end{enumerate}
For comparison we also report the classical sample mean, which highlights
the benefit of robustification.  For the first two Catoni estimators we
use the widest influence function as the following
\[
  \psi_1(x)
  \;=\;
  \begin{cases}
    \log\!\bigl(1 + x + x^2/2\bigr),      &  x \ge 0,\\[2pt]
   -\log\!\bigl(1 - x + x^2/2\bigr),      &  x < 0,
  \end{cases}
\]
while for the joint estimator we take the same $\psi_1$ and set $\psi_2$
to be as the following 
\begin{align} \label{eq:psi2widest}
    \psi_2(x)
  \;=\;
  \begin{cases}
     \log\!\bigl(1 + x + |x|^\beta/\beta\bigr), & x \ge 0,\\[2pt]
    -\log\!\bigl(1 - x + |x|^\beta/\beta\bigr), & x < 0,
  \end{cases}
\end{align}
where $\beta\in(1,2]$ is chosen according to the concrete data setting.

The robustification parameters are set as follows.  For estimator~(1),
\[
  \alpha \;=\; \tfrac{1}{\sigma}\sqrt{\tfrac{2\log(\varepsilon^{-1})}{n}};
  \qquad
  \text{for estimator (2),} \quad
  \bar\alpha \;=\; \tfrac{1}{\bar\sigma}\sqrt{\tfrac{2\log(\varepsilon^{-1})}{n}}.
\]
For the joint estimator, Theorem~\ref{thm:joint theorem} suggests
\[
  \alpha_1 \;=\; \sqrt{\tfrac{2\log(\varepsilon^{-1})}{n}},
  \qquad
  \alpha_2 \;=\; \Bigl(\tfrac{\log(\varepsilon^{-1})}{n}\Bigr)^{1/\beta}.
\]

Figure~\ref{fig:mean1} reports the quantiles of the estimation error
$|\widehat\theta - \mu|$ at levels $\gamma\in[0.9,1]$ for the four
estimators above.  Each curve is produced from $1{,}000$ independent
repetitions with $n=500$.  The parameter configurations for the four
distributions are
\begin{itemize}
  \item normal distribution: standard deviation $\sigma=1$, mean $\mu=0$;
  \item non-central $t$ distribution: degrees of freedom $\nu=2.1$,
        noncentrality parameter $\delta=1$,
        mean $\mu = \sqrt{\nu/2}\,
        \Gamma\!\bigl((\nu-1)/2\bigr)/\Gamma(\nu/2)\cdot\delta$;
  \item Generalised Pareto distribution: shape $\xi = 1/2.1$,
        scale $1$, mean $\mu = 1/(1-\xi) = 21/11$;
  \item Burr Type~XII distribution: first shape $c=1$, second shape
        $k=2.1$, mean $\mu = k\,
        \Gamma(k-1/c)\Gamma(1+1/c)/\Gamma(k+1)$.
\end{itemize}
The three heavy-tailed distributions are calibrated so that their tails
are comparable to those of a $t$-distribution with $2.1$ degrees of
freedom.  In the joint estimator we set $\beta = 2$ under the normal
distribution and $\beta = (2.1 - 0.01)/2$ under the three heavy-tailed
distributions.

From the first panel of Figure~\ref{fig:mean1} we see that under the
normal distribution the four estimators are nearly indistinguishable, as
expected.  Under the heavy-tailed distributions, however, all three
robust estimators clearly dominate the sample mean; in particular, the
joint Catoni estimator is essentially the best curve across the entire
$[0.9, 1]$ range.  A convenient summary is given by the mean
of the estimation error, reported in Table~\ref{tab:mean-summary}.

\begin{table}[t]
  \centering
  \small
  \caption{mean of $|\widehat\theta-\mu|$ from
    $1{,}000$ replications with $n=500$.}
  \label{tab:mean-summary}
  \begin{tabular}{lcccc}
    \toprule
    Distribution
      & sample mean
      & Catoni (true var.)
      & Catoni (samp.~var.)
      & joint Catoni \\
    \midrule
    Normal                                     & 0.0356 & 0.0356 & 0.0356 & 0.0356 \\
    Non-central $t$ ($\nu=2.1$)                & 0.1281 & 0.1164 & 0.1205 & 0.1147 \\
    Generalised Pareto ($\xi=1/2.1$)           & 0.1722 & 0.1537 & 0.1646 & 0.1549 \\
    Burr Type~XII ($c=1,\,k=2.1$)              & 0.0813 & 0.0734 & 0.0784 & 0.0746 \\
    \bottomrule
  \end{tabular}
\end{table}

Figure~\ref{fig:mean2} further investigates how the $99\%$ quantile of
the estimation error depends on the distributional parameter.  For the
normal distribution we vary the standard deviation over $\{1,\,1.5,\,
2,\,\ldots,\,10\}$; for the non-central $t$, Generalised Pareto and the Bur Type XII distributions we vary the degrees of freedom $\nu$, the reciprocal of shape $1/\xi$ and the second shape $k $ over $\{2.1,\,2.2,\,\ldots,\,4\}$ respectively.  In the Gaussian setting the four methods
are practically coincident for every value of $\sigma$.  Under
heavy-tailed noise, the three robust methods provide strikingly more
stable estimates; in particular, the joint Catoni estimator is the
overall best when the tail is heaviest, and remains on par with the
Catoni estimator equipped with the true variance as the tail becomes
lighter.  Interestingly, the joint estimator often outperforms the
Catoni estimator that uses the true variance at high quantile level: we conjecture that this is
because the robustification parameter of the joint approach is estimated
from the data and therefore adapts to the presence of extreme outliers,
while relying on the true variance lacks such adaptivity.  Finally,
replacing the true variance with the sample variance visibly degrades
the Catoni estimator, since the sample variance itself is highly unstable
under heavy-tailed noise.

\begin{figure}[tp]
  \centering
  \includegraphics[width=0.8\linewidth]{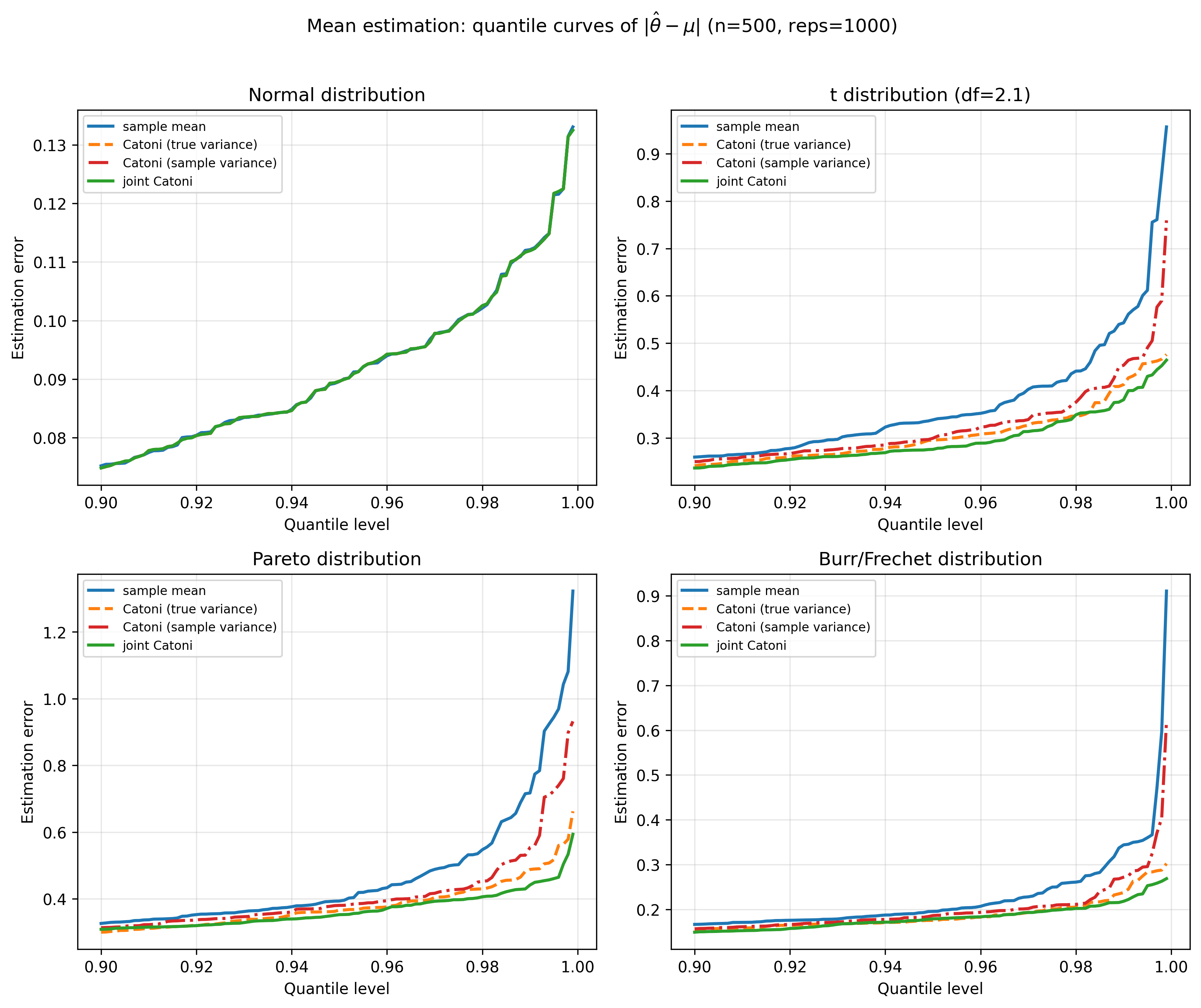}
  \caption{The $\gamma$-quantile of the estimation error $|\widehat\theta-\mu|$
    ($y$-axis) versus $\gamma$ ($x$-axis) for the sample mean and the
    three Catoni-type estimators, based on $1{,}000$ repetitions with
    $n=500$.}
  \label{fig:mean1}
\end{figure}

\begin{figure}[tp]
  \centering
  \includegraphics[width=0.8\linewidth]{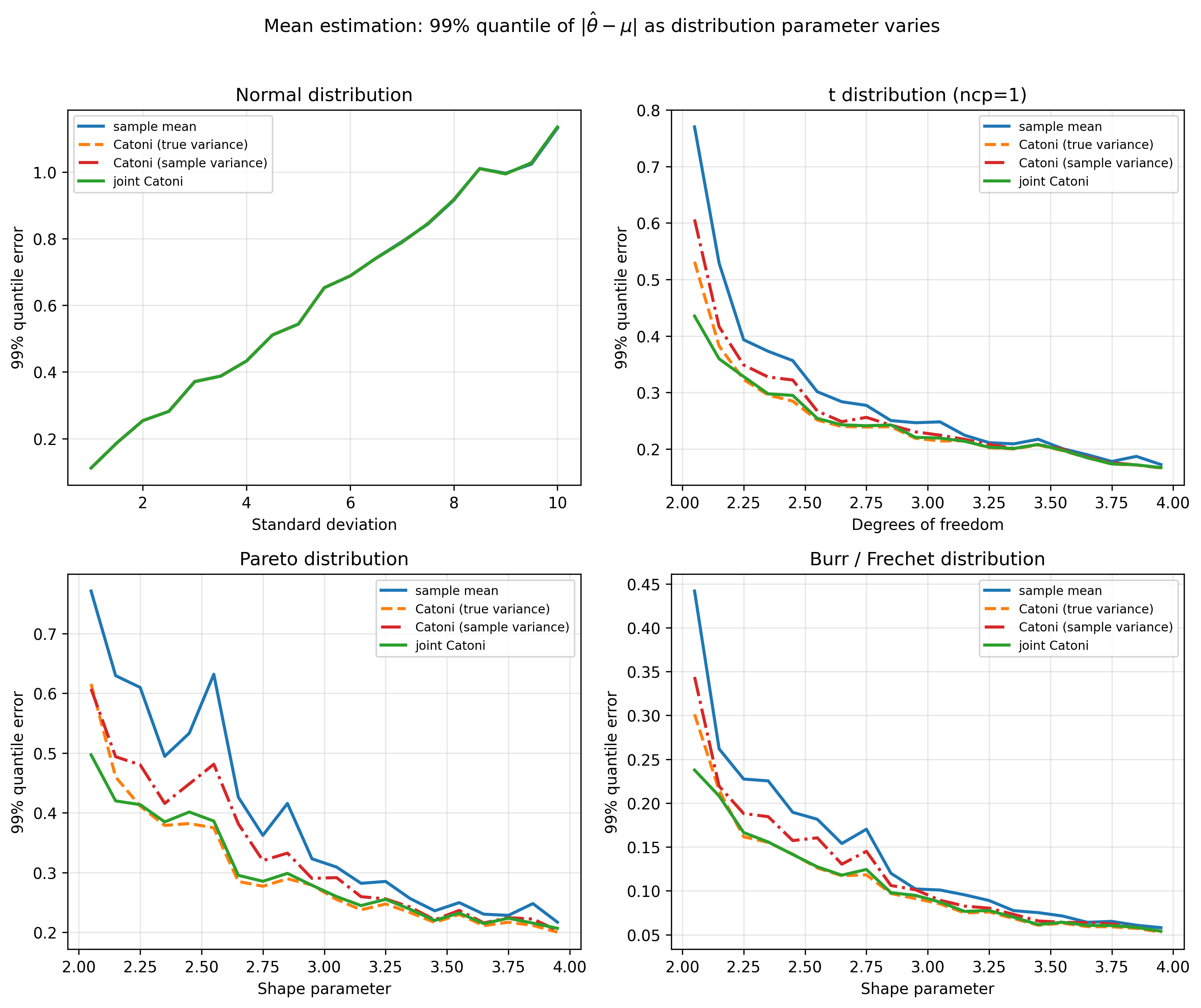}
  \caption{Empirical $99\%$ quantile of the estimation error
    ($y$-axis) for different methods versus a distribution parameter
    (standard deviation for the normal distribution, degrees of freedom
    for the non-central $t$ distribution, the reciprocal of shape for symmetric generalised Pareto and the second shape of Burr Type XII, $x$-axis).  The compared methods are
    the sample mean, the Catoni estimator with the true variance, the
    Catoni estimator with the sample variance, and the joint Catoni
    estimator.}
  \label{fig:mean2}
\end{figure}

\subsection{Numerical experiment 2: linear regression}
\label{sec:sim-regression}

In this experiment we compare the error performance of
\begin{enumerate}
  \item the ordinary least squares (OLS) estimator;
  \item the adaptive Huber estimator ~\cite{SZF20};
  \item the joint Catoni estimator with $\psi_1$being the
        narrowest choices, i.e.
\begin{equation}
  \psi_1(x) \;=\;
  \begin{cases}
    -\operatorname{sign}(x)\log\bigl(1-|x|+x^2/2\bigr), & |x|\le 1,\\
    \operatorname{sign}(x)\log 2,                      & |x|>1,
  \end{cases}
  \label{eq:psi1-narrow}
\end{equation}
and $\psi_2$ be the same as in \eqref{eq:psi2widest}.
\end{enumerate}
The data $\{(y_i,\bm x_i):i=1,\ldots,n\}$ are generated from the linear
model~\eqref{f1f2regression} with covariate dimension $d=5$.  The true
regression coefficient is fixed at $\bm\theta^*=(5,0,-8,0,2)^\top$ and
the covariates $\bm x_i$ are i.i.d.\ from $\mathcal{N}(\bm 0, I_d)$.

For the adaptive Huber estimator we follow the robustification parameter
choice of 
\[
  \widehat\tau \;=\; \widehat c_1\, \bar\sigma\sqrt{\tfrac{n}{\log n}},
  \qquad
  \bar\sigma^2 \;=\; \tfrac{1}{n}\sum_{i=1}^n (y_i - \bm x_i^\top
  \widehat\theta^{\,\text{OLS}})^2,
\]
where we use an OLS residual-based estimate of the error variance so as
to avoid confounding with the signal; for the joint Catoni estimator we
take, in accordance with Theorem \ref{thm:jointly main regression},
\[
  \alpha_1 \;=\; \widehat c_2 \sqrt{\tfrac{2\log(d\varepsilon^{-1})}{n}},
  \qquad
  \alpha_2 \;=\; \widehat c_3
              \Bigl(\tfrac{\log(\varepsilon^{-1})}{n}\Bigr)^{1/\beta},
\]
with $\beta$ set as in the mean-estimation case ($\beta=2$ under
Gaussian noise and $\beta=(2.1-0.01)/2$ under heavy-tailed noise), and
the tuning constants $\widehat c_1, \widehat c_2, \widehat c_3$ selected
from the set $\{0.5,\,1,\,1.5\}$ via three-fold cross-validation.

Figure~\ref{fig:reg1} displays the estimation error
$\|\widehat{\bm\theta}-\bm\theta^*\|_2$ across quantile levels from
$0.9$ to $1$ for the three estimators, based on $500$ replications
with $n=500$.  Four error distributions are considered: standard normal
($\sigma=1$), Student's $t$ with $\nu=2.1$, symmetric generalised
Pareto (scale $=1$, shape $\xi=1/2.1$, obtained by attaching a random sign to
a generalised Pareto distribution), and symmetric Burr Type XII
(first shape $c=1$, second shape $k=2.1$, obtained by attaching a random
sign to a Burr Type XII distribution); all four noises are symmetric.

Under the light-tailed normal distribution (first panel of
Figure~\ref{fig:reg1}) the three estimators perform almost identically.
Under the heavy-tailed cases, however, OLS deteriorates sharply at
quantile levels beyond $0.95$, reflecting the well-known vulnerability of
least squares to outliers.  The adaptive Huber estimator offers a
substantial improvement but still exhibits mild fluctuations stemming
from the variability of the preliminary $\bar\sigma$.  The joint Catoni
estimator is the most stable curve across the board.  The mean of estimation error summary is given in Table~\ref{tab:reg-summary}.

\begin{table}[t]
  \centering
  \small
  \caption{mean of
    $\|\widehat{\bm\theta}-\bm\theta^*\|_2$ from $500$ replications
    with $n=500$ and $d=5$.}
  \label{tab:reg-summary}
  \begin{tabular}{lccc}
    \toprule
    Distribution                              & OLS    & adaptive Huber & joint Catoni \\
    \midrule
    Standard normal                           & 0.0956  & 0.0956          & 0.0956 \\
    $t$ ($\nu=2.1$)                           & 0.2700  & 0.2061          & 0.1866 \\
    Symmetric Generalised Pareto ($\xi=1/2.1$)     & 0.4878  & 0.3327          & 0.2756 \\
    Symmetric Burr Type XII ($\text{c}=1,\text{k}=2.1$)  & 0.2253  & 0.1553          & 0.1324 \\
    \bottomrule
  \end{tabular}
\end{table}

We also report in Figure~\ref{fig:reg2} how the $99\%$ quantile of the
estimation error varies with a distribution parameter.  Under the normal
distribution the three methods remain indistinguishable for every value
of $\sigma\in\{1,2,\ldots,10\}$.  Under heavy-tailed noise, the
robust methods provide a significant advantage; in particular, the joint
Catoni estimator---which simultaneously estimates both the regression
coefficient $\bm\theta^*$ and the noise scale---outperforms the
adaptive Huber estimator, when the noise is very
heavy-tailed.

\begin{figure}[tp]
  \centering
  \includegraphics[width=0.8\linewidth]{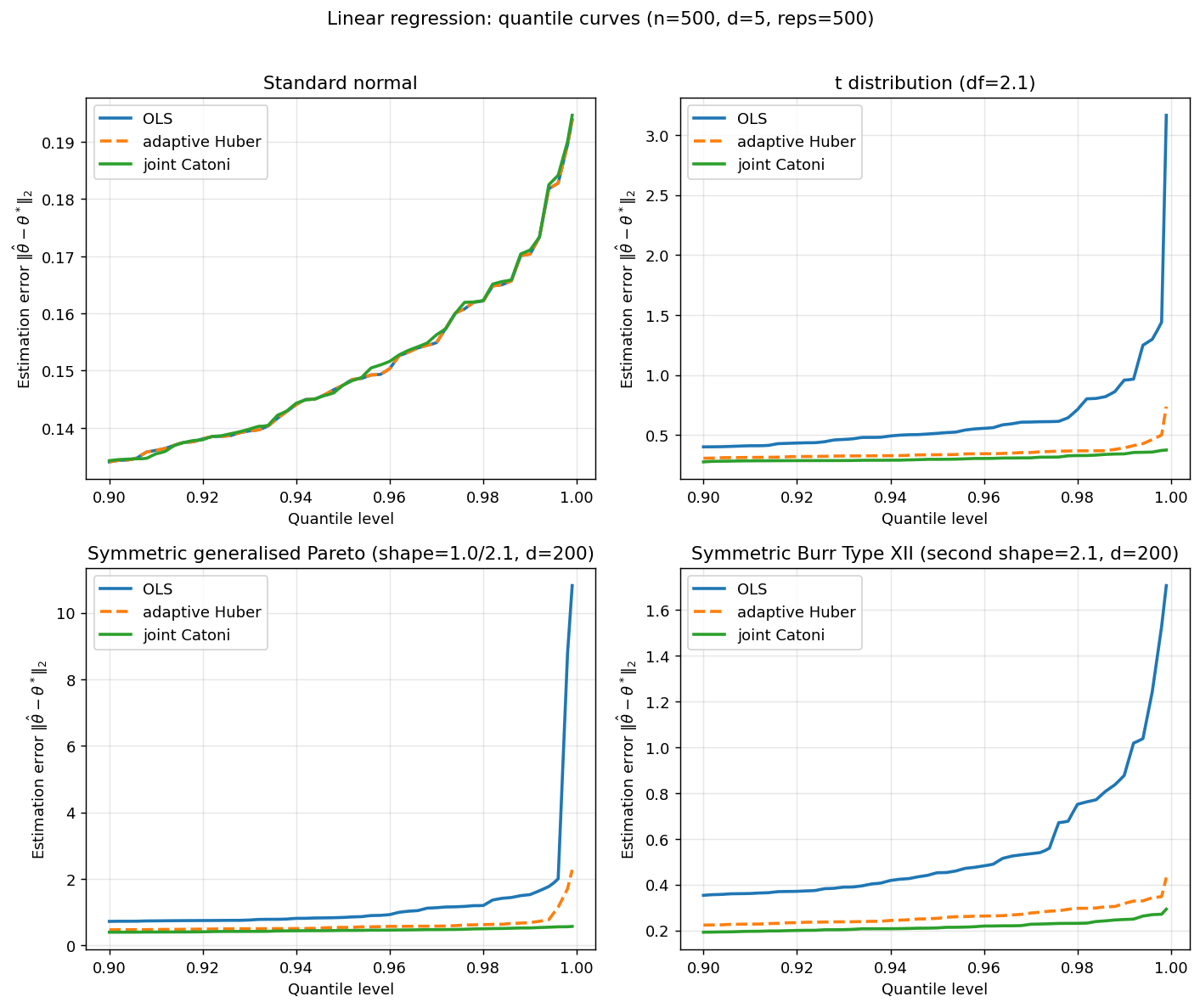}
  \caption{The $\gamma$-quantile of the estimation error
    $\|\widehat{\bm\theta}-\bm\theta^*\|_2$ ($y$-axis) versus $\gamma$
    ($x$-axis) for the OLS, adaptive Huber, and joint Catoni estimators
    with $n=500$ and $d=5$, based on $500$ simulations.}
  \label{fig:reg1}
\end{figure}

\begin{figure}[tp]
  \centering
  \includegraphics[width=0.8\linewidth]{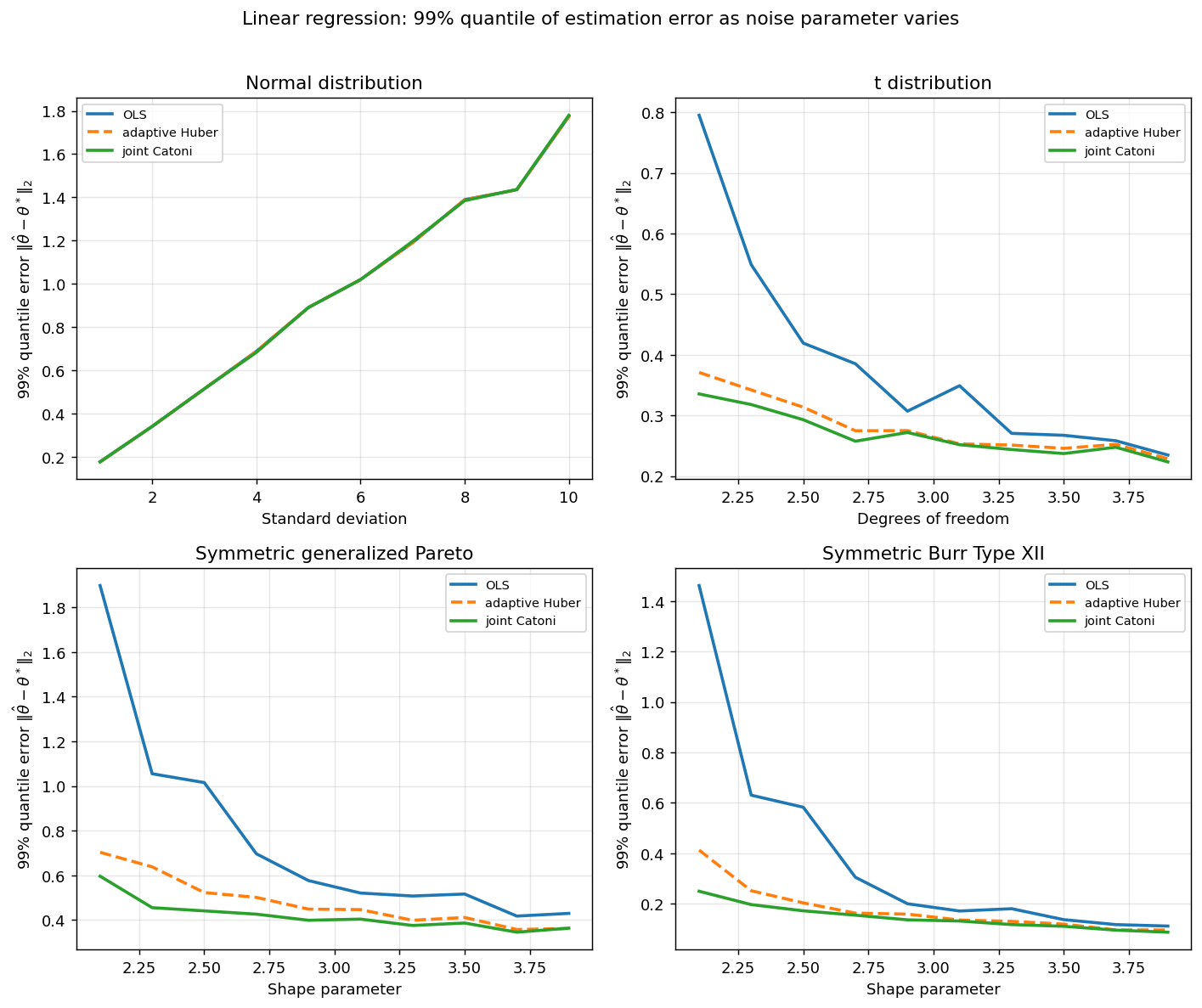}
  \caption{Empirical $99\%$ quantile of the estimation error
    $\|\widehat{\bm\theta}-\bm\theta^*\|_2$ ($y$-axis) versus a
    distribution parameter (standard deviation for the normal, degrees
    of freedom for the $t$, the reciprocal of shape for symmetric generalised Pareto and the second shape of Burr Type XII
    , $x$-axis).  The compared methods are OLS, adaptive Huber,
    and joint Catoni.}
  \label{fig:reg2}
\end{figure}

\subsection{Numerical experiment 3: \texorpdfstring{$\ell_2$}{l2}-penalised regression}
\label{sec:sim-ridge}

We now consider the high-dimensional version of
Section~\ref{sec:sim-regression} and compare three $\ell_2$-penalised
estimators for the linear model~\eqref{f1f2ridge}:
\begin{enumerate}
  \item the ridge estimator;
  \item the $\ell_2$-penalised adaptive Huber estimator;
  \item the $\ell_2$-penalised joint Catoni estimator.
\end{enumerate}
The true coefficient vector is
$\bm\theta^* = (5,0,-8,0,2,0,\ldots,0)^\top\in\mathbb{R}^d$ with $d=200$,
i.e.\ only the first, third, and fifth components are non-zero.  The
covariates $\bm x_i$ are i.i.d.\ from the standard multivariate normal
$\mathcal{N}(\bm 0, I_d)$.

Robustification parameters are selected similarly to the low-dimensional
case.  For the adaptive Huber estimator,
\[
  \widehat\tau =   \bar\sigma \sqrt{\tfrac{n}{\log n \log d}},
  \qquad
  \bar\sigma^2 = \tfrac{1}{n}\sum_{i=1}^n (y_i - \bm x_i^\top
  \widehat\theta^{\,\text{ridge}})^2.
\]
For the joint Catoni estimator,
\[
  \alpha_1 = \sqrt{\tfrac{2\log(d\varepsilon^{-1})}{n}},
  \qquad
  \alpha_2 =\Bigl(\tfrac{\log(\varepsilon^{-1})}{n}\Bigr)^{1/\beta},
\]
where $\varepsilon,\beta$ are
chosen as in Section~\ref{sec:sim-regression}.

The penalty coefficients are selected as follows.  The ridge penalty
$\widehat\lambda_{\text{ridge}}$ is determined via five-fold
cross-validation on a logarithmic grid.  For the adaptive Huber and
joint Catoni estimators, the penalty parameters are obtained by
rescaling $\widehat\lambda_{\text{ridge}}$:
\[
  \widehat\lambda_{\text{AH}} = \widehat c_{\text{AH}}\,
  \widehat\lambda_{\text{ridge}},
  \qquad
  \widehat\lambda_{\text{JC}} = \widehat c_{\text{JC}}\,
  (\alpha_{1}/\hat{v})\widehat\lambda_{\text{ridge}},
\]
where $\widehat c_{\text{AH}}$ and $\widehat c_{\text{JC}}$ are selected from $\{0.5,\,1,\,1.5\}$, both through three-fold
cross-validation using a Huber-loss out-of-fold metric that is robust to
extreme test residuals.  The factor $\alpha_{1}/\hat{v}$ in $\hat{\lambda}_{JC}$ serves to rescale the penalty so that it is commensurate with $\tilde{f}_{1}^{ridge}$. 

In our experiments we fix $n=500$ and $d=200$ and perform $500$
independent repetitions.  Figure~\ref{fig:ridge1} shows the empirical
quantiles of the estimation error as a function of the quantile level,
under the four error distributions described in
Section~\ref{sec:sim-regression}.  Both the $\ell_2$-penalised adaptive
Huber and joint Catoni estimators display robust behaviour at high
quantile levels compared with ridge regression; in particular, under the
heaviest tails, ridge deteriorates by a factor of $2$--$3$ at the
$99\%$-quantile.  The mean of estimation error summary is collected in
Table~\ref{tab:ridge-summary}.

\begin{table}[t]
  \centering
  \small
  \caption{Empirical $99\%$ quantile of
    $\|\widehat{\bm\theta}-\bm\theta^*\|_2$ from $1{,}000$ replications
    with $n=500$ and $d=200$.}
  \label{tab:ridge-summary}
  \begin{tabular}{lccc}
    \toprule
    Distribution                              & ridge  & $\ell_2$-Huber & $\ell_2$-joint Catoni \\
    \midrule
    Standard normal                           & 0.8226  & 0.8185          & 0.8211 \\
    $t$ ($\nu=2.1$)                           & 2.1338  & 1.7458          & 1.6723 \\
    Symmetric Pareto ($\text{shape}=2.1$)     & 3.4722  & 2.7665          &2.6335 \\
    Symmetric Fr\'echet ($\text{shape}=2.1$)  & 1.7575  & 1.3669          & 1.2783 \\
    \bottomrule
  \end{tabular}
\end{table}

In a similar spirit to Figure~\ref{fig:reg2}, we also examine how the
$99\%$ quantile of the estimation error depends on the noise
distribution's shape parameter (The number of replications per parameter setting is reduced to 100).  Results are presented in
Figure~\ref{fig:ridge2}.  When the noise is normally distributed, the
three methods perform comparably across all values of the scale
parameter. However, under extremely heavy-tailed noise, the joint Catoni estimator outperforms the adaptive Huber estimator, and both robust methods exhibit substantially lower estimation errors than ridge regression, demonstrating their superior stability in the presence of heavy-tailed noise.

\begin{figure}[tp]
  \centering
  \includegraphics[width=0.8\linewidth]{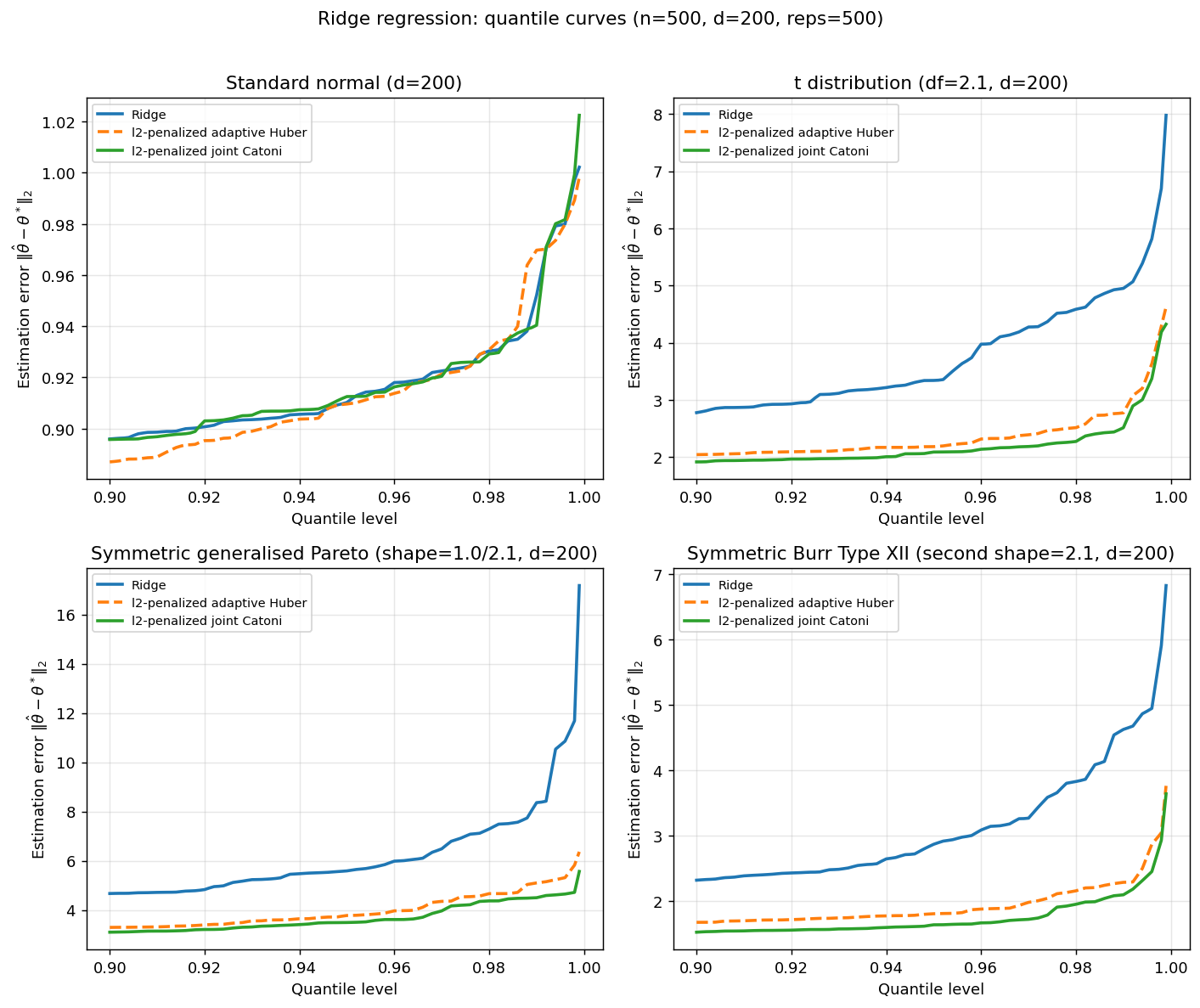}
  \caption{The $\gamma$-quantile of the estimation error
    $\|\widehat{\bm\theta}-\bm\theta^*\|_2$ ($y$-axis) versus $\gamma$
    ($x$-axis) for ridge, $\ell_2$-penalised adaptive Huber, and
    $\ell_2$-penalised joint Catoni estimators with $n=500$ and $d=200$,
    based on $500$ simulations.}
  \label{fig:ridge1}
\end{figure}

\begin{figure}[tp]
  \centering
  \includegraphics[width=0.8\linewidth]{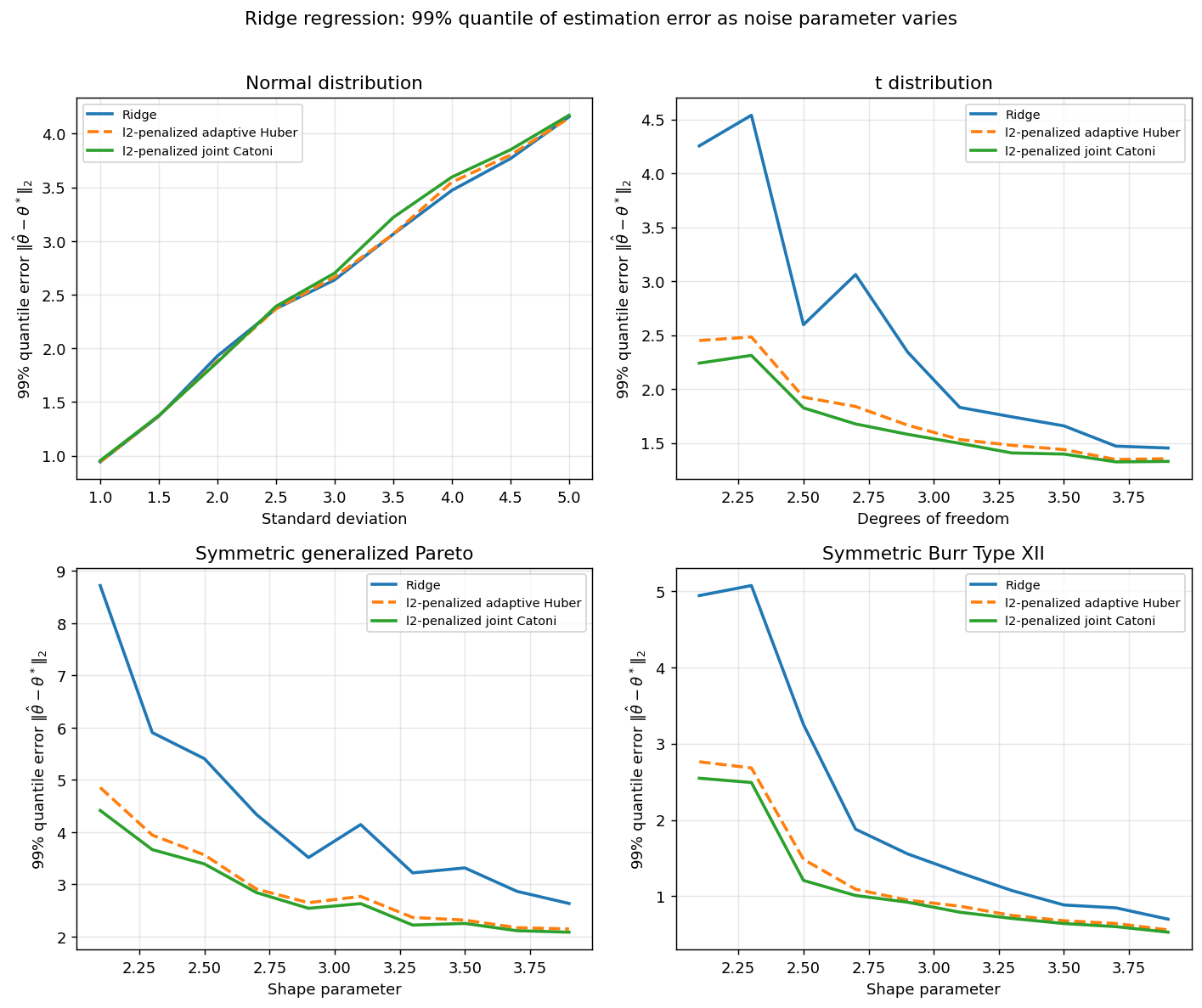}
  \caption{Empirical $99\%$ quantile of the estimation error
    $\|\widehat{\bm\theta}-\bm\theta^*\|_2$ ($y$-axis) versus a
    distribution parameter (standard deviation for the normal, degrees
    of freedom for the $t$, the reciprocal of shape for symmetric generalised Pareto and the second shape of Burr Type XII, $x$-axis) for ridge, $\ell_2$-penalised adaptive Huber,
    and $\ell_2$-penalised joint Catoni estimators.}
  \label{fig:ridge2}
\end{figure}

\subsection{Summary of the numerical findings}
\label{sec:sim-summary}

Combining the three experiments, we draw the following empirical
conclusions:
\begin{enumerate}
  \item Under Gaussian noise, the joint Catoni estimator is essentially
        indistinguishable from the best non-robust baseline (the sample
        mean, OLS, or ridge).  In particular, robustification incurs no
        observable cost in the light-tailed regime.
  \item Under heavy-tailed noise, the joint Catoni estimator---which
        simultaneously estimates the location/coefficient and the
        noise scale---consistently outperforms both the non-robust
        baseline and the plug-in Catoni estimator that uses the sample
        variance; it is at least comparable with, and often superior
        to, the adaptive Huber estimator in low dimension.
  \item In the high-dimensional ridge setting the joint Catoni
        estimator and the $\ell_2$-penalised adaptive Huber are both
        dramatically more stable than ridge regression at high quantile
        levels; the two robust methods are of comparable quality, and
        the choice between them can be guided by the heaviness of the
        tail and the available computational budget.
\end{enumerate}
These findings confirm the theoretical guarantees established in
Theorems~\ref{thm:joint theorem}, ~\ref{thm:jointly main regression}, and ~\ref{Thm:jointly ridge}:
the proposed joint estimators deliver sub-Gaussian-type deviation
behaviour under only a finite $2\beta$-moment assumption on the noise,
and do so without any tuning of the scale.

\subsection{A brief discussion on the selection of Catoni type functions} \label{ss:discussioncatoni}

 In the previous simulations analysis, we chose special $\psi_1$ and $\psi_2$ and did not give reasons for such choices. The theory about selecting Catoni functions is an important and delicate topic, we leave it for the future research and demonstrate how to practically choose such functions through the example of mean estimation.

We conduct numerical experiments to compare joint Catoni-type estimators under three choices of the $\psi_1$ function across two pairs of data-generating distributions (symmetric vs. asymmetric variants). Specifically, we consider:
\begin{itemize}
    \item [(1)] The widest function satisfying Catoni’s conditions \eqref{catoni c}:
    \begin{align*}
        \psi_{1}^{wide}(x):=\begin{cases}\log \left(1+x+x^2 / 2\right), & x \geq 0 \\ -\log \left(1-x+x^2 / 2\right), & x < 0\end{cases}
    \end{align*}
    \item[(2)] The narrowest function satisfying Catoni’s conditions:
    \begin{align*}
        \psi_{1}^{narrow}(x) = \begin{cases}
            -\operatorname{sign}(x)\log\left(1-|x|+\frac{x^{2}}{2} \right), & |x|\le 1,\\
            \operatorname{sign}(x)\log(2), &|x|>1,
        \end{cases}
    \end{align*}
    \item[(3)] An intermediate profile constructed by mixing the widest and narrowest function:
    \begin{align*}
        \psi_{1}^{mid}(x):=\frac{1}{2}\psi_{1}^{wide}(x)+\frac{1}{2}\psi_{1}^{narrow}(x).
    \end{align*}
\end{itemize}
For all three joint Catoni methods, $\psi_2$ is set to the widest choice 
\begin{align*}
     \psi_{2}(x)= \begin{cases}
    \log \left(1+x+|x|^{\beta} / \beta\right), & x \geq 0 \\ -\log \left(1-x+|x|^{\beta} / \beta\right), & x < 0\end{cases}
\end{align*}
and the parameter settings are set as: $\beta=(2.1-0.01)/2$,
\begin{align*}
    \alpha_{1} = \sqrt{\frac{2\log \left(\epsilon^{-1}\right)}{n}} \text{ and }
\alpha_{2} = \left( \frac{\log(\epsilon^{-1})}{n} \right)^{1/\beta}.
\end{align*}
Under this design, we evaluate the estimation error of the three joint Catoni methods when samples follow the following distribution pairs:
\begin{itemize}
    \item [(i)] Student t distribution vs. one-sided (half) t distribution with degree of freedom $=2.1$;
    \item[(ii)] symmetric double Pareto distribution vs. (one-sided) Pareto distribution with shape parameter $=2.1$.
\end{itemize}
For each distribution pair, we perform $1,000$ independent trials with $n = 500$ observations per trial. The quantiles of three Catoni joint estimation method at quantile levels between $0.9$ and $1$ are reported in Figure \ref{fig:different Catoni}. From the figure, the three methods exhibit different relative performance under symmetric versus asymmetric settings. When the sample distribution is symmetric, the narrowest Catoni variant consistently performs best, while the widest Catoni variant yields the largest errors. In contrast, when the distribution is asymmetric, this ordering reverses: the widest Catoni becomes the best among the three. The middle Catoni performs in between the narrowest and the widest variants across all cases.

We interpret these patterns as follows. The narrowest Catoni attenuates outliers most aggressively, which most effectively reduces variance under symmetry. However, when the distribution is asymmetric, discarding the magnitude information of outliers introduces substantial bias, thereby worsening overall error. Based on this understanding, we suggest choosing the Catoni function according to the degree of symmetry in the data (e.g., the gap between the mean and the median). If nothing is known about the target distribution’s symmetry, the middle Catoni may be a robust default that provides a balanced trade-off.

Moreover, in linear regression models the noise is typically assumed to be symmetric. However, we argue that symmetry of the noise alone is insufficient to determine which choice is preferable; one must also account for the effect of $\<\boldsymbol{x}_{i},\boldsymbol{\theta}^{*}- \boldsymbol{\theta} \> $. This point needs further investigation.
\begin{figure}[tp]
    \centering
    \begin{minipage}{0.49\linewidth}
        \centering
        \includegraphics[width=\linewidth]{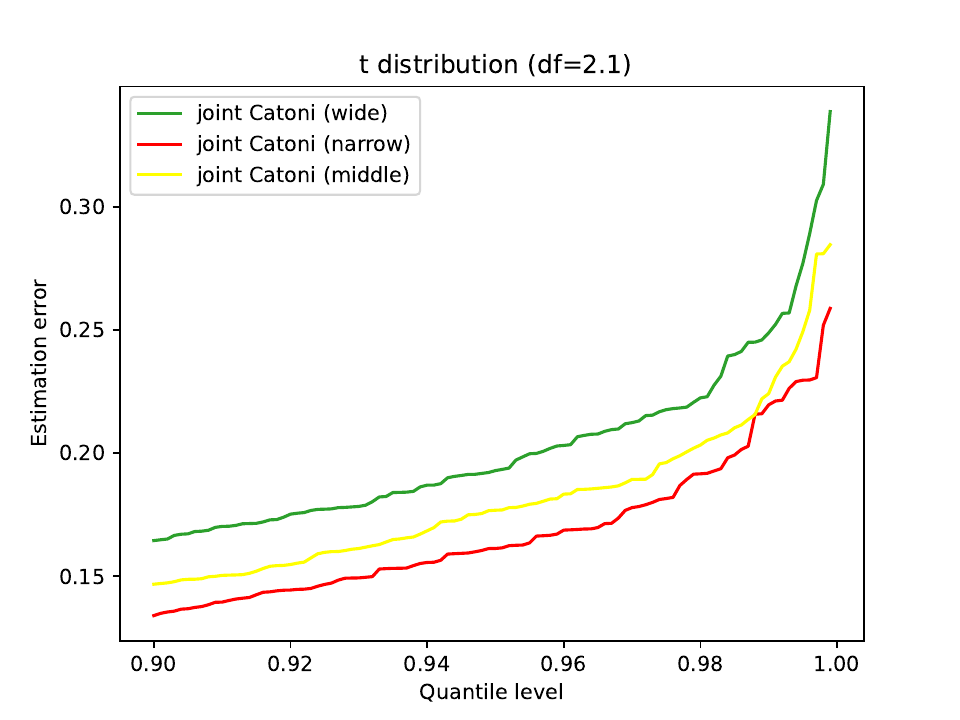}
    \end{minipage} 
    \begin{minipage}{0.49\linewidth}
        \centering
        \includegraphics[width=\linewidth]{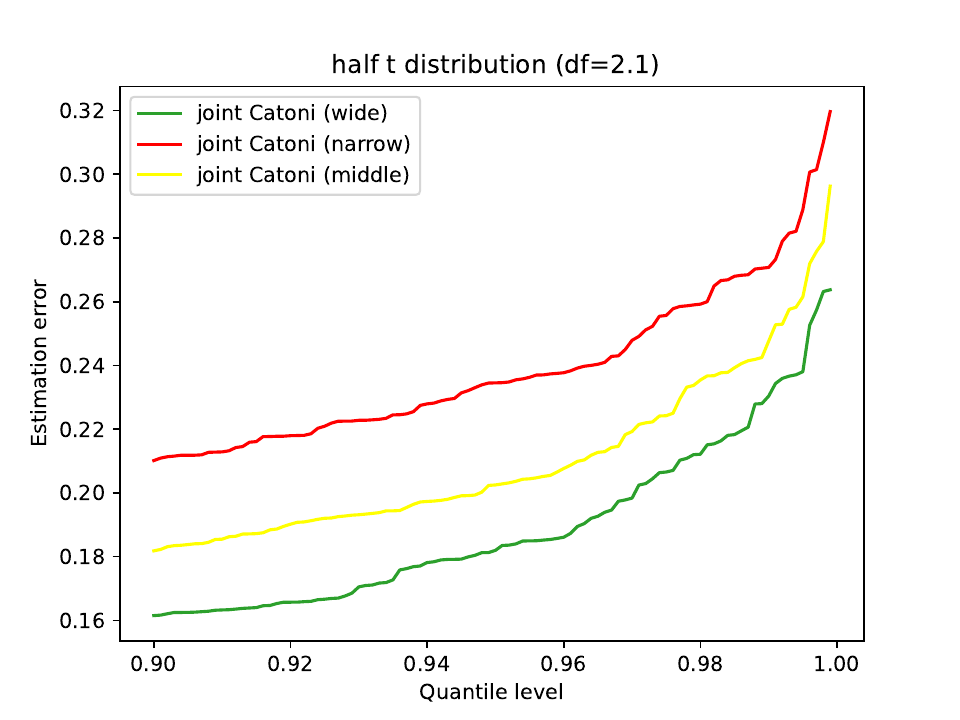}
    \end{minipage} 
    \\
    \begin{minipage}{0.49\linewidth}
        \centering
        \includegraphics[width=\linewidth]{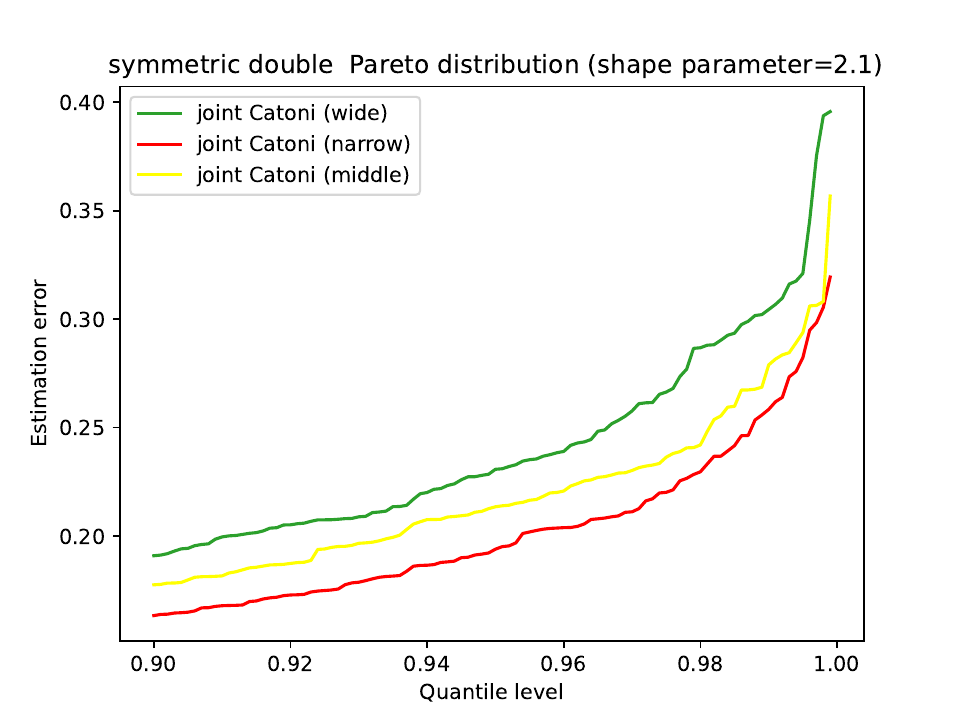}
    \end{minipage} 
    \begin{minipage}{0.49\linewidth}
        \centering
        \includegraphics[width=\linewidth]{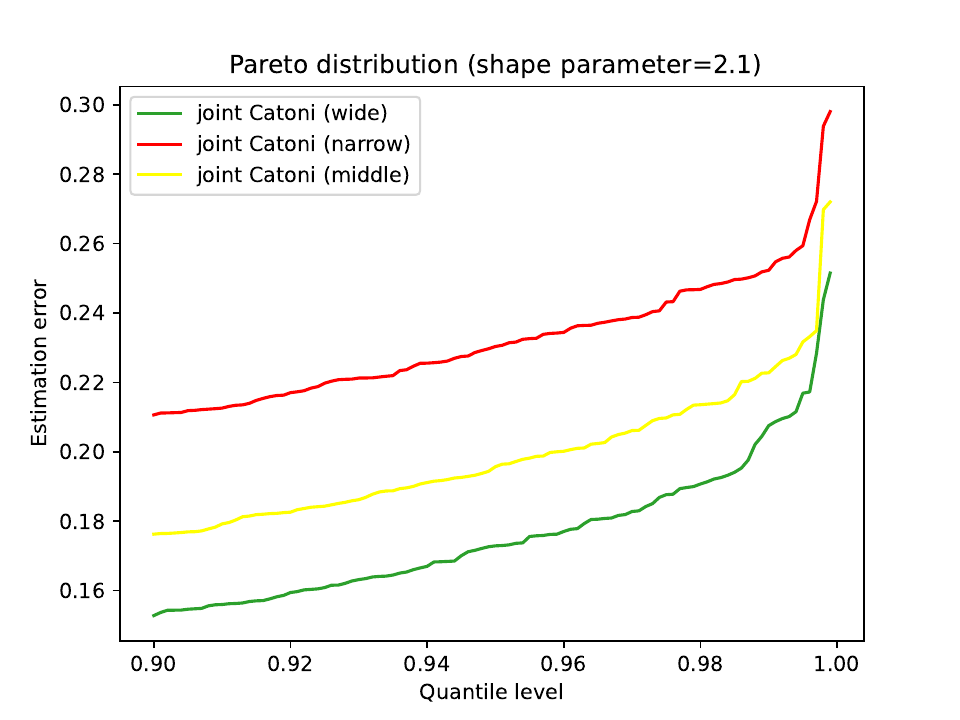}
    \end{minipage} 
    \caption{The $\gamma$-quantile of the estimation error $|\mu-\hat{\theta}|$ (estimation error, $y$-axis) versus $\gamma$ (quantile level, $x$-axis) for three  Catoni-type estimators based on 1000 simulations and $n=500$ for each simulations. }
    \label{fig:different Catoni}
\end{figure}

\clearpage

\appendix

\section{Auxiliary lemmas}
The following lemmas will be frequently used in the proofs of the main results, so we put it in a separate section. 
\begin{lem} \label{lem:root bound}
    Let $g(x):=a|x|^{\beta}+bx+c$ for some $\beta\in (1,2]$. The following results hold:
    \begin{itemize}
        \item[(i)]  Assume $a>0$, $b>0$ and $c>0$. If there exists some $r>\frac{1}{b}$ such that $b-a(cr)^{\beta-1}>\frac{1}{r}$, then the equation $g(x)=0$ has two solutions in which the  larger one $x_{+}$ satisfies
        \begin{align*}
            x_{+}\ge -\frac{c}{b-a(cr)^{\beta-1}}.
        \end{align*}
        \item[(ii)]  Assume $a<0$, $b>0$ and $c<0$. If there exists some $r>\frac{1}{b}$ such that $b+a(-cr)^{\beta-1}>\frac{1}{r}$, then the equation $g(x)=0$ has two solutions in which the  smaller one $x_{-}$ satisfies
        \begin{align*}
            x_{-}\le -\frac{c}{b+a(-cr)^{\beta-1}}.
        \end{align*}
    \end{itemize}
\end{lem}
\begin{proof}
The proofs of (i) and (ii) are similar, so we only prove (i). Write $\Tilde{g}(x)=g(-x)=a|x|^{\beta}-bx+c$, it is equivalent to show that $\Tilde{g}$ has two solutions in which the smallest one is smaller or equal to $\frac{c}{b-a(cr)^{\beta-1}}$. It is clear to see that 
    $$\tilde{g}(x) \ge c>0, \ \ \ x \le 0,$$ so all the solutions of $\tilde g(x)=0$ must be positive. Notice that $b-a(cr)^{\beta-1}>\frac{1}{r}$ implies 
    \begin{align*}
        0<\frac{1}{b-a(cr)^{\beta-1}}<r.
    \end{align*}
    Hence,
    \begin{align*}
        \Tilde{g}\left(\frac{c}{b-a(cr)^{\beta-1}} \right)=\frac{ac^{\beta}}{b-a(cr)^{\beta-1}}\left[ \left(\frac{1}{b-a(cr)^{\beta-1}} \right)^{\beta-1}-r^{\beta-1}\right]<0,
    \end{align*}
   and there exists one solution between $0$ and $\frac{c}{b-a(cr)^{\beta-1}}$ and the smallest solution must be smaller than $\frac{c}{b-a(cr)^{\beta-1}}$. On the other hand,  since $\tilde{g}'(x)=a \beta x^{\beta-1}-b$ is increasing for $x \ge 0$ and $\tilde{g}'(0)=-b<0$, $\tilde g(x)=0$ only has two solutions, one between $0$ and $\frac{c}{b-a(cr)^{\beta-1}}$ and the other larger than $\frac{c}{b-a(cr)^{\beta-1}}$.
\end{proof}

 Recall that $\psi_1$ and $\psi_2$ are non-decreasing continuous functions satisfying \eqref{catoni c} and \eqref{catoni c2} respectively. we shall frequently use the following inequality: 
\begin{equation} \label{e:XPsi1Y}
x \psi_1(y) \le |x| \log\left(1+\operatorname{sign}(x) y+\frac{y^2}2\right),\quad \forall  x,y\in \mathds{R}.
\end{equation} 
To see why this inequality holds, note that by \eqref{catoni c}, if $x>0$, we have
    \begin{align*}
        x\psi_{1}(y)\le x\log\left(1+y+\frac{y^{2}}{2}\right),
    \end{align*}
    while if $x\le 0$, we have
    \begin{align*}
        x\psi_{1}(y) \le -x\log\left(1-y+\frac{y^{2}}{2}\right).
    \end{align*}
    
We shall also frequently use the following inequality: 
\begin{align} \label{ineq:1+x+y}
        (1+x+y)\le (1+x)\exp\left(\frac{y}{1+x} \right),\quad \forall  x>-1,y\in \mathds{R},
    \end{align}
 which immediately follows from the elementary inequality $1+a \le e^{a}$ for $a \in \mathds{R}$ with $a=\frac{y}{1+x}$.

\section{Catoni type joint robust mean and variance estimation} \label{Sec:AppB}
 We shall prove Theorem \ref{thm:joint theorem} in this section, which is divided into three subsections: the first outlines the strategy of the proof, and the other two are devoted to proving two auxiliary lemmas.
\subsection{The strategy for the proof of Theorem \ref{thm:joint theorem}}
Our approach primarily relies on the Poincar\'e–Miranda theorem(we refer to \cite{Fr18} for a formulation 
with a viability condition), which is a generalization of the intermediate value theorem.\par
For $\epsilon\in (0,\frac{1}{6})$ and $\alpha_{1}$ and $\alpha_{2}$ defined in \eqref{def:alpha1,2}, let us first define the rectangle $\mathcal{K}$ as the following:  
\begin{align} \label{def:K}
         \mathcal{K}:=\left\{(\theta,v): \mu-\theta_{0} \le \theta\le \mu+\theta_{0}, \ \ v_{0}\le v\le V_{0}   \right\},
\end{align}
where
\begin{align} \label{e:vVTh}
    v_{0}:=(1+c_{0} \alpha_{2}^{\beta-1})^{-\frac{1}{2}}\sigma, \quad \quad V_{0}:= \left(1-c_{0} \alpha_{2}^{\beta-1} \right)^{-\frac{1}{2}}\sigma,\quad \quad \theta_{0}:=\sigma(1+\delta)\alpha_{1},
\end{align}
with
$c_{0}=27+\frac{2^{5\beta-3}}{\beta \sigma^{2\beta}} \left(\frac{m_{2\beta}}{\sigma^{2\beta}}+1 \right)$ and $$\delta:=\max\left\{16M(\sigma^{2}+1)\frac{V_0-v_0}{v_0},\left(\frac{\sigma^{2}}{v_0^{2}}-1\right)(1+4\alpha_{1}^{2}) \right\}=O(\alpha_{2}^{\beta-1}).$$ 
By the definitions in \eqref{def:alpha1,2}, $\alpha_{1}\to 0$ and $\alpha_{2}^{\beta-1}\to 0$ as $\frac{\log(\epsilon^{-1})}{n}\to 0$. Consequently, there always exists a sufficiently large absolute constant $C$ such that $\theta_0\le 1,\delta\le 1$ and $c_0\alpha_{2}^{\beta-1}\le \frac{1}{2}$ provided $n\ge C(1+\log(\epsilon^{-1}))$. In what follows, we shall assume this is the case.

Based on our definitions of $\alpha_1$ and $\alpha_2$, once we prove that \eqref{f1f2} has a solution in $\mathcal{K}$, the error estimation for this solution naturally follows. The boundary $\partial \mathcal K$ of the rectangle $\mathcal K$,
the external normal unit vector $N_{\mathcal K}(\mu,v)$ at the boundary $\partial \mathcal K$ is 
$$N_{\mathcal K}(\theta,v)=\begin{cases} (-1,0), & \ \ \ \theta=\mu-\theta_0, v \in [v_0,V_0];  \\
(1,0), & \ \ \ \theta=\mu+\theta_0, v \in [v_0,V_0];  \\
(0,-1), & \ \ \ \theta \in [\mu-\theta_0,\mu+\theta_0], v=v_0;  \\
(0,1), & \ \ \ \theta \in [\mu-\theta_0,\mu+\theta_0], v=V_0;  \\
\end{cases}$$

\par

{Write $F(\theta,v):=(f_1(\theta,v),f_2(\theta,v))$, by the Poincaré–Miranda theorem, as long as 
\begin{equation} \label{e:BoundaryNegative}
\langle N_{\mathcal K}(\theta,v), F(\theta,v)\rangle \le 0, \ \ \ (\theta, v) \in \partial \mathcal K, 
\end{equation}
then $F(\theta,v)=0$ (i.e. \eqref{f1f2}) has a solution in $\mathcal K$. The condition \eqref{e:BoundaryNegative} is represented as 
the following relations: 
\begin{align} \label{ineq:f1}
        f_{1}(\mu-\theta_{0},v)\ge 0,\quad f_{1}(\mu+\theta_{0},v)\le 0, \quad \forall v\in [v_{0},V_{0}];
\end{align}
 \begin{align} \label{ineq:f2}
        f_{2}(\theta,v_{0})\ge 0,\quad f_{2}\left(\theta,V_{0} \right)\le 0,\quad \forall \theta\in [\mu-\theta_{0},\mu+\theta_{0}] .
    \end{align}
Hence, to prove that there exists a solution to \eqref{f1f2} in the rectangular region $\mathcal{K}$ with probability at least $1-6\epsilon$, it suffices to show that
    \begin{itemize}
        \item[(i)] \eqref{ineq:f1} holds with probability at least $1-4\epsilon$;
        \item[(ii)] \eqref{ineq:f2} holds with probability at least $1-2\epsilon$.
    \end{itemize}
In this two dimensional setting, the conditions \eqref{ineq:f1} and \eqref{ineq:f2} can be illustrated by Figure \ref{f:PMTh} which is very helpful for us to understand the Poincar\'e–Miranda Theorem. 
\begin{figure}[htbp] 
    \centering
    \includegraphics[width=0.7\textwidth]{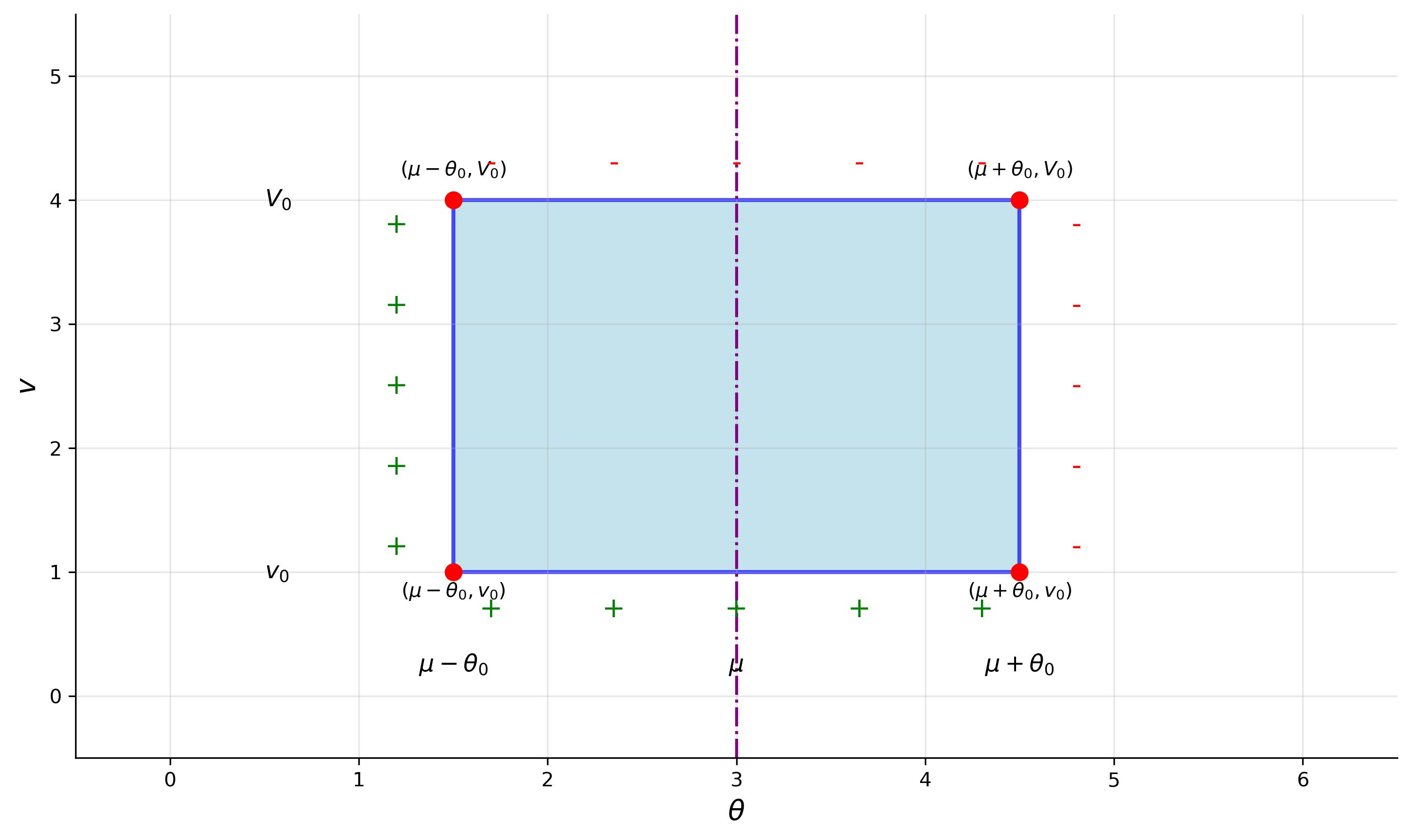}
    \caption{Visualization of $\mathcal K(\theta, v)$ for Poincar\'e–Miranda Theorem}
    \label{f:PMTh}
\end{figure}

\par
 Let us define the following events
\begin{align} \label{def:E1E2}
    \begin{split}
    &\mathcal{E}_{1}^{+}:=\left\{ \sup_{v\in [v_{0},V_{0}]} f_{1}(\mu+\theta_{0},v)\le 0\right\},\ \ \ \ \mathcal{E}_{1}^{-}:=\left\{ \inf_{v\in [v_{0},V_{0}]} f_{1}(\mu-\theta_{0},v)\ge 0\right\},\\
    &\mathcal{E}_{2}^{+}:=\left\{\sup_{\theta\in [\mu-\theta_{0},\mu+\theta_{0}]}f_{2}(\theta,V_{0}) \le 0\right\},\ \ \ \ \mathcal{E}_{2}^{-}:=\left\{ \inf_{\theta\in [\mu-\theta_{0},\mu+\theta_{0}]}f_{2}(\theta,v_{0})\ge 0\right\}.
    \end{split}
\end{align}
The verification of (i) and (ii) is divided into the following two lemmas.
\begin{lem} \label{lem:probE1}
    Under the assumptions and settings of Theorem \ref{thm:joint theorem}, for any $\epsilon\in (0,\frac{1}{6})$, there exist a constant $C:=C(m_{2\beta},\sigma,\beta)$ such that as $n\ge C[\log(\epsilon^{-1})+1]$, the following statement holds:
    \begin{align*}
        \mathbb{P}\left( \mathcal{E}_{1}^{+}\right)\ge 1-2\epsilon,\quad \mathbb{P}\left( \mathcal{E}_{1}^{-}\right)\ge 1-2\epsilon.
    \end{align*}
\end{lem}

\begin{lem} \label{lem:probE2}
     Under the assumptions and settings of Theorem \ref{thm:joint theorem},  for any $\epsilon\in (0,\frac{1}{6})$, there exist a constant $C:=C(m_{2\beta},\sigma,\beta)$ such that as $n\ge C[\log(\epsilon^{-1})+1]$, the following statement holds:
    \begin{align*}
        \mathbb{P}\left( \mathcal{E}_{2}^{+}\right)\ge 1-\epsilon,\quad \mathbb{P}\left( \mathcal{E}_{2}^{-}\right)\ge 1-\epsilon.
    \end{align*}
\end{lem}
The proofs of these two lemmas are placed in the following subsections.
\begin{proof}[Proof of Theorem \ref{thm:joint theorem}]
 Lemmas \ref{lem:probE1} and \ref{lem:probE2} imply (i) and (ii) above respectively.   
Hence, by the Poincar\'e–Miranda Theorem, the conclusions in the theorem hold true.
\end{proof}

\subsection{Proof of Lemma \ref{lem:probE1}} 
We begin with a lemma that serves as a key tool for dealing with the supremum in the probability.
\begin{lem} \label{lem:vfdiff}
    Under the assumptions of Theorem \ref{thm:joint theorem}, for any $|\theta-\mu|\le \theta_0$ it holds with probability at least $1-\epsilon$ that
    \begin{align*}
        \sup_{v_1,v_{2}\in [v_0,V_0]} \left|v_{1}f_{1}(\theta,v_{1})-v_2f_{1}(\theta,v_{2})\right|\le  8M(\sigma^{2}+1)\alpha_{1}^{2}|V_0-v_0|.
    \end{align*}
\end{lem}
\begin{proof}
    Let $g(\theta,v)=vf_{1}(\theta,v)=v\cdot\frac1n\sum_{i}\psi_{1}(\alpha_{1}(X_{i}-\theta)/v)$. Using $\psi_{1}(0)=0$, we have
\[
\partial_{v}g(\theta,v)=-\tfrac{1}{n}\sum_{i=1}^{n}\!\int_{0}^{1}\!\!\!\int_{t}^{1}\tfrac{\alpha_{1}^{2}(X_{i}-\theta)^{2}}{v^{2}}\,\psi_{1}''\!\bigl(s\alpha_{1}\tfrac{X_{i}-\theta}{v}\bigr)\,\mathrm{d}s\,\mathrm{d}t.
\]
 Combining this with Assumption \ref{assumption:c}, we then have, for any $(\theta,v)\in \mathcal{K}$,
    \begin{align} \label{ineq:g'}
        \left| \frac{\partial}{\partial v}g(\theta,v)\right|
        &\le \int_{0}^{1}\int_{t}^{1}  \frac{M}{n}\sum_{i=1}^{n}\left[\frac{\alpha_{1}^{2}(X_i-\theta)^{2}}{v_{0}^{2}}1_{\{s\alpha_{1}|X_i-\theta|\le \sigma\}} +\frac{1}{s^{2}}1_{\{s\alpha_{1}|X_i-\theta|>\sigma\}}\right]\mathrm{d}s\mathrm{d}t.
    \end{align}
    Note that, for any $u>0$,
    \begin{align*}
        \int_{0}^{1}\int_{t}^{1} 1_{\{su\le 1\}} \mathrm{d}s\mathrm{d}t=\frac{1}{2}1_{\{0<u\le 1\}}+\frac{1}{2u^{2}}1_{\{u>1\}},
    \end{align*}
    and
    \begin{align*}
         \int_{0}^{1}\int_{t}^{1} \frac{1}{s^{2}}1_{\{su> 1\}} \mathrm{d}s\mathrm{d}t&=1_{\{u>1\}} \int_{1/u}^{1}\int_{t}^{1} \frac{1}{s^{2}} \mathrm{d}s\mathrm{d}t
         \le \log(u)1_{\{u>1\}}\le \log(1+u^{2}).
    \end{align*}
    Combining this with \eqref{ineq:g'} gives us, for any $|\theta-\mu|\le \theta_0$,
    \begin{align} \label{ineq:g'-2}
        \begin{split}
            \sup_{v\in [v_0,V_{0}]} \left| \frac{\partial}{\partial v}g(\theta,v)\right|
        &\le \frac{M}{n}\sum_{i=1}^{n} \left[\frac{\alpha_{1}^{2}(X_i-\theta)^{2}}{2v_0^{2}}1_{\left\{\alpha |X_i-\theta|\le \sigma\right\} }+\frac{\sigma^{2}}{2v_0^{2}}1_{\left\{\alpha_{1}|X_i-\theta|>\sigma\right\} } \right.\\
        &\qquad\qquad \quad\left.+\log\left(1+\frac{\alpha_{1}^{2}|X_i-\theta|^{2}}{\sigma^{2}}\right)\right].
        \end{split}
    \end{align}
    
 Let us bound the three terms on the right hand of the above inequality.  
Write $Y_{i}=\alpha^{2}(X_i-\theta)^{2}1_{\{\alpha|X_i-\theta|\le \sigma\}}$ and $Z_{i}=1_{\{\alpha|X_i-\theta|>\sigma\}}$. Whenever $\theta_{0}/\sigma=O(\alpha_{1})\le 1$, for any $|\theta-\mu|\le \theta_0$, we have
    \begin{align*}
        \mathbb{E}[Y_i]\le 2\alpha_{1}^{2}\sigma^{2}, \ \ 
        \operatorname{Var}(Y_{i})\le 2\alpha_{1}^{2}\sigma^{4}, \ \ 
        \mathbb{E}[Z_i] \le 2\alpha_{1}^{2},\ \ 
        \operatorname{Var}(Z_i) \le 2\alpha_{1}^{2}.
    \end{align*}
    It follows from Bernstein's inequality with $n\alpha_1^{2}=2 \log (\epsilon^{-1})$ that
    \begin{align*}
            \mathbb{P}\left(\sum_{i=1}^{n}Y_i-\mathbb{E}[Y_i]>3n\alpha_{1}^{2}\sigma^{2} \right)\le \epsilon/3, \ \ \ \ 
        \mathbb{P}\left(\sum_{i=1}^{n}Z_i-\mathbb{E}[Z_i]>3n\alpha_{1}^{2}\right)\le \epsilon/3.
    \end{align*}
   By Markov's inequality and the inequality $1+x \le e^x$, we have
    \begin{align*}
        \mathbb{P}\left(\sum_{i=1}^{n}\log\left(1+\frac{\alpha_{1}^{2}|X_i-\theta|^{2}}{\sigma^{2}}\right)\ge 4n\alpha_{1}^{2}\right) \le \epsilon/3.
    \end{align*}
    Combining these three high-probability upper bounds with \eqref{ineq:g'-2}, we obtain, with probability at least $1-\epsilon$,
    \begin{align*}
         \sup_{v\in [v_0,V_{0}]}\left| \frac{\partial}{\partial v}g(\theta,v)\right|\le 8M(\sigma^{2}+1)\alpha_{1}^{2}.
    \end{align*}
    Since Assumption \ref{assumption:c} implies that for any $\theta$, $g(\theta,v)$ is absolutely continuous with respect to $v$, the desired result follows.
\end{proof}

\begin{proof}[Proof of Lemma \ref{lem:probE1}]
It follows from \eqref{catoni c}, the definition of $f_{1}$ in \eqref{f1f2} and \eqref{Xmoment} that
\begin{align*}
    \mathbb{E}\left[\exp\left(nf_{1}(\theta,v_0)\right)\right]
    &\le \mathbb{E}\left[\prod_{i=1}^{n} \exp\left(\psi_{1}\left(\alpha_{1} \frac{\left( X_{i}-\theta\right)}{v_0} \right) \right)\right]\\
    &\le \left(1+\mathbb{E}\left[\alpha_{1} \frac{(X_{1}-\theta)}{v_0} +\alpha_{1}^{2}\frac{(X_{1}-\theta)^{2}}{ 2v_0^{2}}\right] \right)^{n}\\
    &\le \exp\left[n\alpha_{1}\left(\frac{\mu-\theta}{v_0}+\frac{\alpha_{1}\sigma^{2}}{2v_0^{2}}+\frac{\alpha_{1}|\mu-\theta|^{2}}{2v_0^{2}}\right)   \right].
\end{align*}
Hence, by the Markov's inequality and the definitions of $\theta_0,v_0$ and $\delta$, we have
\begin{align*}
    \mathbb{P}\left(f_{1}(\mu+\theta_{0},v_0)\ge -\frac{\delta}{2} \alpha_{1}^{2}\right)
   & \le \frac{\mathbb{E}\left[e^{nf_{1}(\mu+\theta_0,v_0)}\right]}{e^{-\frac{n\delta}{2} \alpha_{1}^{2}}}\\
    &\le  \exp\left[n\alpha_{1}^{2}\left(-\frac{1}{2}-\frac{\delta}{2}+(\frac{\sigma^{2}}{2v_0^{2}}-\frac{1}{2})(1+4\alpha_{1}^{2})\right) \right]\\
    &\le \exp\left(-\frac{1}{2}n\alpha_{1}^{2}\right)=\epsilon,
\end{align*}
provided $n\ge C(1+\log(\epsilon^{-1}))$ for some $C>0$.
Combining this with Lemma \eqref{lem:vfdiff}, we obtain that, with probability at least $1-2\epsilon$, 
\begin{align*}
    \sup_{v\in [v_0,V_0]} vf_{1}(\mu+\theta_0,v)
    &\le v_{0}f_{1}(\mu+\theta_0,v_0)+\sup_{v\in [v_0,V_0]}\left|v_{0}f_{1}(\mu+\theta_0,v_0)-vf_{1}(\mu+\theta_0,v)\right| \\
    &\le [-\frac{v_0\delta}{2}+8M(\sigma^{2}+1)|V_0-v_0|]\alpha_{1}^{2} \le 0.
\end{align*}
Since $v>0$, the first result follows. The second result of the lemma follows by applying the lower bound in \eqref{catoni c} to 
$-nf_{1}(\theta,v)$ and a similar argument. This completes the proof.
\end{proof}

\subsection{Proof of Lemma \ref{lem:probE2}}
\begin{proof}[Proof of Lemma \ref{lem:probE2}]
 We split the proof into the following two claims: 

\textbf{Claim 1}: There exist a constant $C=C(m_{2\beta},\sigma,\beta)$, such that, as $n \ge C [\log(\epsilon^{-1})+1]$,
\begin{align} \label{ineq:f2upper}
    \mathbb{P} \left(\sup_{\theta\in [\mu-\theta_{0},\mu+\theta_{0}] }f_{2}(\theta,v)\le \alpha_{2}B^{v}_{+}(v)\right)\ge 1-\epsilon,  \quad \forall v>0,
\end{align}
\begin{align} \label{ineq:f2lower}
    \mathbb{P} \left(\inf_{\theta\in [\mu-\theta_{0},\mu+\theta_{0}] }f_{2}(\theta,v)\ge \alpha_{2}B^{v}_{-}(v)\right)\ge 1-\epsilon, \quad \forall v>0,
\end{align}
with
\begin{align} \label{def:Bv}
    \begin{split}
    B^{v}_{+}(v)&:= \frac{\sigma^{2}}{v^{2}}-1+\frac{2\sigma \theta_{0}+\theta_{0}^{2}}{v^{2}}+ \frac{2^{4\beta-3}\alpha_{2}^{\beta-1}}{\beta}\left( \frac{m_{2\beta}+\theta_{0}^{2\beta}+\sigma^{2\beta}}{v^{2\beta}} +\left| \frac{\sigma^{2}}{v^{2}}-1\right|^{\beta}\right)+ \alpha^{\beta-1}_2,\\
    B^{v}_{-}(v)&:= \frac{\sigma^{2}}{v^{2}}-1-\frac{2\sigma \theta_{0}+\theta_{0}^{2}}{v^{2}}- \frac{2^{4\beta-3}\alpha_{2}^{\beta-1}}{\beta}\left( \frac{m_{2\beta}+\theta_{0}^{2\beta}+\sigma^{2\beta}}{v^{2\beta}} +\left| \frac{\sigma^{2}}{v^{2}}-1\right|^{\beta}\right)- \alpha^{\beta-1}_2.
    \end{split}
\end{align}

\textbf{Claim 2}: There exists a $v_+ \in (0,V_0]$ such that, as $n \ge C [\log(\epsilon^{-1})+1]$,
\begin{equation} \label{e:B+v=0}
B^v_{+}(v_{+}) \le 0,
\end{equation}
and there exists a $v_{-} \ge v_0$ such that,  as $n \ge C [\log(\epsilon^{-1})+1]$,
\begin{equation} \label{e:B-v=0}
B^v_{-}(v_{-}) \ge 0.
\end{equation}

 Assume that the above two claims hold true. Let us use \eqref{ineq:f2upper} and \eqref{e:B+v=0} to prove the first inequality in the lemma.     Since $f_{2}(\theta,v)$ is monotonically decreasing with respect to $v$ and $v_+ \le V_0$, we have 
   \begin{align*}
       f_{2}(\theta,V_{0})\le f_{2}(\theta,v_{+}), \quad \forall \theta\in [\mu-\theta_{0},\mu+\theta_{0}],
   \end{align*}
which immediately implies 
$$\mathbb{P}\bigg( \sup_{\theta\in [\mu-\theta_{0},\mu+\theta_{0}]}f_{2}(\theta,V_{0})\le 0  \bigg)\ge \mathbb{P}\left(\sup_{\theta\in [\mu-\theta_{0},\mu+\theta_{0}]}f_{2}(\theta,v_{+})\le 0 \right).$$
By \eqref{e:B+v=0}, we have 
   \begin{align*}
       \mathbb{P}\bigg( \sup_{\theta\in [\mu-\theta_{0},\mu+\theta_{0}]}f_{2}(\theta,V_{0}) \le 0  \bigg)& \ge \mathbb{P}\left(\sup_{\theta\in [\mu-\theta_{0},\mu+\theta_{0}]}f_{2}(\theta,v_{+})\le 0 \right) \\
       & \ge \mathbb{P}\left(\sup_{\theta\in [\mu-\theta_{0},\mu+\theta_{0}]}f_{2}(\theta,v_{+})\le  B^{v}_{+}(v_{+})\right)  \ge 1-\epsilon,
   \end{align*}
where the last inequality is by \eqref{ineq:f2upper}. 

The second inequality in the lemma can be proved by \eqref{ineq:f2lower} and \eqref{e:B-v=0} in the same way as above. It remains to prove the two claims.

\underline{\textbf{Proof of Claim 1}}:
Let us first prove \eqref{ineq:f2upper}.  It follows from \eqref{Xmoment}, \eqref{catoni c2} and the definition of $f_{2}$ in \eqref{f1f2} that
\begin{align} \label{ineq:f2upper'}
    \begin{split}
        &\mathbb{E}\left[\sup_{\theta\in [\mu-\theta_{0},\mu+\theta_{0}]} \left\{\exp\left( n f_{2}(\theta,v)  \right) \right\}\right]\\
    &= \mathbb{E}\left[\sup_{\theta\in [\mu-\theta_{0},\mu+\theta_{0}]}\left\{ \prod_{i=1}^{n} \exp\left( \psi_{2}\left[ \alpha_{2}\left(\frac{(X_{i}-\theta)^{2}}{v^{2}}-1\right)\right] \right) \right\}  \right]\\
    &\le \prod_{i=1}^{n}\mathbb{E} \left[\sup_{\theta\in [\mu-\theta_{0},\mu+\theta_{0}]}\left\{1+\alpha_{2}\left(\frac{(X_{i}-\theta)^{2}}{v^{2}}-1\right)+\frac{\alpha_{2}^{\beta}}{\beta } \left|\frac{(X_{i}-\theta)^{2}}{v^{2}}-1 \right|^{\beta}\right\} \right].
    \end{split}
\end{align} 
By the inequality $(X-\theta)^2 \le  (X-\mu)^2+(\theta-\mu)^2+2 |X-\mu||\theta-\mu|$, we have 
\begin{align} \label{e:SupV1}
\begin{split}
    & \ \ \ \mathbb{E}\left[ \sup_{\theta\in [\mu-\theta_{0},\mu+\theta_{0}]} \left\{\frac{(X_{i}-\theta)^{2}}{v^{2}}-1  \right\} \right] \\
    &\le \mathbb{E}\left[ \sup_{\theta\in [\mu-\theta_{0},\mu+\theta_{0}]} \frac{(X_{i}-\mu)^{2}+(\mu-\theta)^{2}+2|X_i-\mu||\mu- \theta|}{v^{2}}  \right]-1 \\
    & \le \frac{2 \sigma \theta_0+\theta_0^2}{v^2}+\frac{\sigma^{2}}{v^{2}}-1, 
\end{split}
\end{align}
where in the second inequality we have used $\mathbb{E} |X_i-\mu| \le \sigma$. 
On the other hand, by repeatedly using the convexity inequality 
\begin{align} \label{ineq:convexity}
    |x+y|^{\beta}\le 2^{\beta-1} x^{\beta}+2^{\beta-1} y^{\beta}, \forall x,y\in \mathbb{R},
\end{align}
one can deduce that
\begin{align*}
     \left|\frac{(X_{i}-\theta)^{2}}{v^{2}}-1  \right|^{\beta}
    &\le 2^{2\beta-2} \left|\frac{|X_{i}-\theta|^{2\beta}+\sigma^{2\beta}}{v^{2\beta}} \right|+2^{\beta-1}\left|\frac{\sigma^{2}}{v^{2}}-1 \right|^{\beta}\\
    &\le 2^{4\beta-3} \frac{|X_{i}-\mu|^{2\beta}+|\mu-\theta|^{2\beta}}{v^{2\beta}} +\frac{2^{2\beta-2}\sigma^{2\beta}}{v^{2\beta}}+2^{\beta-1}\left|\frac{\sigma^{2}}{v^{2}}-1 \right|^{\beta},
\end{align*}
and thus 
\begin{align*}
    &\mathbb{E}\left[ \sup_{\theta\in [\mu-\theta_{0},\mu+\theta_{0}]} \left|\frac{(X_{i}-\theta)^{2}}{v^{2}}-1  \right|^{\beta}\right]\le 2^{4\beta-3} \left( \frac{m_{2\beta}+\theta_{0}^{2\beta}+\sigma^{2\beta}}{v^{2\beta}} +\left| \frac{\sigma^{2}}{v^{2}}-1\right|^{\beta} \right).
\end{align*}
Combining these two inequalities with \eqref{ineq:f2upper'} and the definition of $B^{v}_{+}$ in \eqref{def:Bv} gives us
\begin{align*}
    &\mathbb{E}\left[\sup_{\theta\in [\mu-\theta_{0},\mu+\theta_{0}]} \left\{\exp\left( n f_{2}(\theta,v)  \right) \right\}\right]\\
    &\le \left[1+\alpha_{2}\left(\frac{2\sigma \theta_{0}+\theta_{0}^{2}}{v^{2}}+\frac{\sigma^{2}}{v^{2}}-1 \right)+\frac{2^{4\beta-3}\alpha_{2}^{\beta}}{\beta}\left( \frac{m_{2\beta}+\theta_{0}^{2\beta}+\sigma^{2\beta}}{v^{2\beta}} +\left| \frac{\sigma^{2}}{v^{2}}-1\right|^{\beta}\right)   \right]^{n}\\
    &\le \exp\left( n\alpha_{2}B^{v}_{+}(v)-\log(\epsilon^{-1}) \right), 
\end{align*}
where the last inequality is by $1+x \le e^{x}$ and $\alpha_2^{\beta-1}=\frac{\log(\epsilon^{-1})}{n\alpha_2}$.
By Markov's inequality, 
 \begin{align*}
     \mathbb{P}\left(\sup_{\theta\in [\mu-\theta_{0},\mu+\theta_{0}]} \left\{\exp\left( n f_{2}(\theta,v)  \right) \right\}\ge \exp\left( n\alpha_{2}B^{v}_{+}(v) \right)\right)\le \epsilon,
 \end{align*}
 which implies \eqref{ineq:f2upper}. 

For the other inequality \eqref{ineq:f2lower}, it can be proved by applying the same argument for $-f_{2}$ in which \eqref{e:SupV1} is replaced with 
\begin{align} \label{e:SupV2}
\begin{split}
    &\mathbb{E}\left[ \sup_{\theta\in [\mu-\theta_{0},\mu+\theta_{0}]} \left\{-\frac{(X_{i}-\theta)^{2}}{v^{2}}+1  \right\} \right] \le \frac{2\sigma \theta_{0}+\theta_{0}^{2}}{v^{2}}+1-\frac{\sigma^{2}}{v^{2}},
\end{split}
\end{align}
 and \eqref{e:SupV2} can be proved in the same way as for proving \eqref{e:SupV1}. So the proof of Claim 1 is finished.\par

\underline{\textbf{Proof of Claim 2}}: 
Writing $t=\frac{\sigma^{2}}{v^{2}}-1$, we have
    \begin{align*}
        B^{v}_{+}(v)
        &\le \frac{2^{4\beta-3}\alpha_{2}^{\beta-1}}{\beta}\left(1+ \frac{2^{\beta-1}\left(m_{2\beta}+\theta_{0}^{2\beta}+\sigma^{2\beta}\right)}{\sigma^{2\beta}}\right)\left| t\right|^{\beta}+\left(1+ \frac{2\sigma \theta_{0}+\theta_{0}^{2}}{\sigma^{2}} \right)t\\
        &\quad +\frac{2^{5\beta-4}\alpha_{2}^{\beta-1}\left( m_{2\beta}+\theta_{0}^{2\beta}+\sigma^{2\beta}\right)}{\beta \sigma^{2\beta}}+\frac{2\sigma \theta_{0}+\theta_{0}^{2}}{\sigma^{2}}+ \alpha_2^{\beta-1}\\
        &:=a|t|^{\beta}+bt+c.
    \end{align*}
 Denote 
$$\bar{B}^v_{+}(t)=a|t|^{\beta}+bt+c,$$
where $t=\frac{\sigma^{2}}{v^{2}}-1$, we have shown 
\begin{equation} \label{e:BvBt}
B^{v}_{+}(v) \le \bar{B}^v_{+}(t).
\end{equation}
Let us find the largest root of $\bar{B}^v_{+}$. We have $a>0,b>0, c>0$, and it is clear that
    \begin{align*}
        a = O\left(\left(\frac{\log(\epsilon^{-1})}{n} \right)^{\frac{\beta-1}{\beta}}\right), \quad b =1+O\left( \sqrt{\frac{\log(n^{\frac{1}{\beta}-\frac{1}{2}}\epsilon^{-1})}{n}} \right), \quad c=O\left( \left(\frac{\log(\epsilon^{-1})}{n} \right)^{\frac{\beta-1}{\beta}} \right).
    \end{align*}
   Hence, there must exist a constant $C=C(m_{2\beta},\sigma,\beta)$ such that, as $n\ge C\left[\log\left(\epsilon^{-1} \right) +1\right]$, 
    \begin{align*}
        b-a\left( 2c\right)^{\beta-1}>\frac{1}{2}.
    \end{align*}
    In this case, Lemma \ref{lem:root bound}(i) with $r=2$ implies that the equation $\bar{B}^{v}(t)=0$ has two roots. Denote the largest one by $t_{+}$ and $$v_{+}:=\frac{\sigma}{\sqrt{1+t_{+}}}$$ (note that the two roots are both large than $-1$ by the definition of $t$). Also by Lemma \ref{lem:root bound}(i) and the definition of $c$ above, one has
    \begin{align*}
        t_{+}&\ge -\frac{c}{b-a( 2c)^{\beta-1}}\ge -2c
        \ge  -\left(2+ \frac{2^{5\beta-3}\left( m_{2\beta}+\theta_{0}^{2\beta}+\sigma^{2\beta}\right)}{\beta \sigma^{2\beta}} \right)\alpha_{2}^{\beta-1}-24\alpha_{1}-72\alpha_{1}^{2},
    \end{align*}
 where we used the definition of $\alpha_1,\alpha_2$ and $\theta_0$ in the last inequality. Since $\alpha_1$ can be bounded in terms of $\alpha_2^{\beta-1}$ up to a constant by their definitions in \eqref{def:alpha1,2},  there exist a constant $C=C(m_{2\beta},\sigma,\beta)$ (larger than the previous $C$) such that, as $n\ge C[\log(\epsilon^{-1})+1]$,
   \begin{align*}
       t_{+}\ge -c_{0}\alpha_{2}^{\beta-1},
   \end{align*}
   which yields 
   \begin{align*}
       v_{+}\le \left(1-c_{0}\alpha_{2}^{\beta-1} \right)^{-\frac{1}{2}}\sigma =V_{0}.
   \end{align*}
By \eqref{e:BvBt}, we have 
$$B^{v}_{+}(v_{+})\le \bar{B}^{v}_{+}(t_{+})=0,$$
 which is exactly equation \eqref{e:B+v=0}.
The \eqref{e:B-v=0} can be proved in the same way. So Claim 2 holds true and the proof is complete.
\end{proof}

\section{Catoni type joint robust regression estimation} \label{Sec:AppC}
We shall prove Theorem \ref{thm:jointly main regression} in this section, which consists of three subsections, one giving the strategy of the proof, the other two proving two auxiliary lemmas.    
\subsection{The strategy for the proof of Theorem \ref{thm:jointly main regression}} 
 For $\epsilon\in (0,\frac{1}{6})$ and 
 $\alpha_1$ and $\alpha_2$ defined in \eqref{def:alpha1,2 re}, let us first define the high dimensional cylinder $\mathcal{K}^{re}$ as the following:
\begin{align} \label{def:Kregression}
    \mathcal{K}^{re}:=\left\{ (\boldsymbol{\theta},v): \|\boldsymbol{\theta}-\boldsymbol{\theta}^{*}\|_{2}\le \theta_{0}, v_{0}\le v\le V_{0} \right\},
\end{align}
with
\begin{equation}\label{def:vVtheta regression}
    \begin{gathered} 
    v_{0}:=\left(1+c_{0} \left(\sqrt{d}\alpha_1+\alpha_{2}^{\beta-1}\right)\right)^{-\frac{1}{2}}\sigma,\quad V_{0}:=\left(1-c_{0} \left(\sqrt{d}\alpha_1+\alpha_{2}^{\beta-1}\right)\right)^{-\frac{1}{2}}\sigma,
 \end{gathered}
\end{equation} 
\begin{align} \label{e:c0}
       \theta_{0}:= \frac{54L\sqrt{d}V_0}{c_{l}}\alpha_{1}, \ \ \ \ c_{0}:=\frac{264c_{u}L}{c_{l}}+\frac{2^{5\beta-2}}{\beta}\left(\frac{m_{2\beta}}{\sigma^{2\beta}}+1 \right).
\end{align}
 As in Section \ref{Sec:AppB}, we assume throughout that our assumptions ensure $\theta_0\le 1$ and $$c_0\left(\sqrt{d}\alpha_{1}+\alpha_{2}^{\beta-1}\right)\le \frac{1}{2}.$$

 Denote by  $\partial \mathcal{K}^{re}$ the boundary of $\mathcal{K}^{re}$, and for any $(\boldsymbol{\theta},v)\in \partial\mathcal{K}^{re}$, denote by $N_{\mathcal{K}^{re}}(\boldsymbol{\theta},v)$ its unit normal vector:
\begin{align*}
        N_{\mathcal{K}^{re}}(\boldsymbol{\theta},v)=\begin{cases}
            (\frac{\boldsymbol{\theta}-\boldsymbol{\theta}^{*}}{\|\boldsymbol{\theta}-\boldsymbol{\theta}^{*}\|_{2}},0), & v\in (v_{0},V_{0}), \|\boldsymbol{\theta}-\boldsymbol{\theta}^{*}\|_{2}=\theta_0;\\
            (\boldsymbol{0},1),& v=V_{0}, \|\boldsymbol{\theta}-\boldsymbol{\theta}^{*}\|_{2}\le \theta_0;\\
            (\boldsymbol{0},-1), &v=v_{0}, \|\boldsymbol{\theta}-\boldsymbol{\theta}^{*}\|_{2} \le \theta_0.
        \end{cases}
    \end{align*}
By the definition of the set $\mathcal{K}^{re}$, to prove Theorem \ref{thm:jointly main regression}, it suffices to prove that  \eqref{f1f2regression} admits a solution in $\mathcal{K}^{re}$ with high probability. To this end, we shall use 
the generalized Poincar\'e–Miranda theorem in \cite[Theorem 2.4]{Fr18}, which is a multidimensional intermediate value theorem. 

 Write $$\tilde{F}(\boldsymbol{\theta},v):=(\tilde{f}_{1}(\boldsymbol{\theta},v),\tilde{f}_{2}(\boldsymbol{\theta},v)), \ \ \ \forall \ (\boldsymbol{\theta},v)\in \mathcal{K}^{re},$$
where $\tilde{f}_{1}$ and $\tilde{f}_{2}$ are defined in \eqref{f1f2regression}.  
According to the generalized Poincar\'e–Miranda theorem, if we can verify that
    \begin{align} \label{e:NormCon}
        \<N_{\mathcal{K}^{re}}(\boldsymbol{\theta},v), \tilde{F}(\boldsymbol{\theta},v)   \>\le 0, \ \ \ \ \ \forall \ (\boldsymbol{\theta},v)\in \partial \mathcal{K}^{re},
    \end{align}
    then $\tilde{F}(\boldsymbol{\theta},v)=0$ has solutions in $\mathcal{K}^{re}$, i.e., the solutions to the regression problem \eqref{f1f2regression}.
 The condition \eqref{e:NormCon} in our setting reads as 
        \begin{align} \label{ineq:tf1}
           \<\tilde{f}_{1}(\boldsymbol{\theta},v), \boldsymbol{\theta}-\boldsymbol{\theta}^{*}\>\le 0,\quad \forall v\in (v_{0},V_{0}), \boldsymbol{\theta}\in \partial\mathcal{B}_{\theta_{0}}(\boldsymbol{\theta}^{*}),
        \end{align}
        \begin{align} \label{ineq:tf2}
            \tilde{f}_{2}(\boldsymbol{\theta},v_{0})\ge 0,\quad \tilde{f}_{2}(\boldsymbol{\theta},V_{0})\le 0,\quad \forall \boldsymbol{\theta}\in \mathcal{B}_{\theta_{0}}(\boldsymbol{\theta}^{*}),
        \end{align}
    where $\mathcal{B}_{\theta_{0}}(\boldsymbol{\theta}^{*}):=\{\boldsymbol{\theta}:\|\boldsymbol{\theta}-\boldsymbol{\theta}^{*}\|_{2}\le \theta_{0}\}$ and  $\partial\mathcal{B}_{\theta_{0}}(\boldsymbol{\theta}^{*})$ is the boundary of $\mathcal{B}_{\theta_{0}}(\boldsymbol{\theta}^{*})$.
    \vskip 2mm
 It is easy to see that \eqref{ineq:tf1} is equivalent to the following event $\mathcal{E}_{1}$ and that two inequalities in \eqref{ineq:tf2} are equivalent to the following events $\mathcal{E}_{2}^{+}$ and $\mathcal{E}_{2}^{-}$.  
\begin{align} \label{def:E1E2regression}
    \begin{split}
    &\mathcal{E}_{1}:=\left\{ \sup_{v\in [v_{0},V_{0}]} \sup_{\boldsymbol{\theta}\in \partial \mathcal{B}_{\theta_{0}}(\boldsymbol{\theta}^{*})} \<\tilde{f}_{1}(\boldsymbol{\theta},v), \boldsymbol{\theta}-\boldsymbol{\theta}^{*}\>\le 0\right\},\\
    &\mathcal{E}_{2}^{+}:=\left\{\sup_{\boldsymbol{\theta}\in \mathcal{B}_{\theta_{0}}(\boldsymbol{\theta}^{*})}\tilde{f}_{2}(\boldsymbol{\theta},V_{0}) \le 0\right\}, \ \ \ \ \mathcal{E}_{2}^{-}:=\left\{ \inf_{\boldsymbol{\theta}\in \mathcal{B}_{\theta_{0}}(\boldsymbol{\theta}^{*})}\tilde{f}_{2}(\boldsymbol{\theta},v_{0})\ge 0\right\}.
    \end{split}
\end{align}
We shall prove that the conditions \eqref{ineq:tf1} and \eqref{ineq:tf2} hold with high probability, which are stated in the following two lemmas more precisely. 
\begin{lem} \label{lem:probE1 re}
    Under the assumptions and settings of Theorem \ref{thm:jointly main regression},  for any $\epsilon\in (0,\frac{1}{6})$, there exist a constant $C:=C(m_{2\beta},\sigma,\beta,c_{l},  c_{u}, L,M)$ such that as 
      $n\ge Cd [\log(\epsilon^{-1})+1]$, 
     we have
    \begin{align*}
        \mathbb{P}\left( \mathcal{E}_{1}\right)\ge 1-4\epsilon.
    \end{align*}
\end{lem}

\begin{lem} \label{lem:probE2 re}
     Under the assumptions and settings of Theorem \ref{thm:jointly main regression},  for any $\epsilon\in (0,\frac{1}{6})$, there exist a constant $C:=C(m_{2\beta},\sigma,\beta,c_{l},  c_{u}, L)$ such that as 
      $n\ge Cd [\log(\epsilon^{-1})+1],$ 
     we have
    \begin{align*}
        \mathbb{P}\left( \mathcal{E}_{2}^{+}\right)\ge 1-\epsilon,\quad \mathbb{P}\left( \mathcal{E}_{2}^{-}\right)\ge 1-\epsilon.
    \end{align*}
\end{lem}  
The proofs of these two lemmas are placed in the following subsections. With these two lemmas, we are at the position to prove the main result. 
\begin{proof}[Proof of Theorem \ref{thm:jointly main regression}]
By \eqref{def:vVtheta regression} and \eqref{e:c0}, we can easily see that 
\begin{align} \label{e:the0}
       \theta_0 \le C_{1}\sigma \sqrt{\frac{d\log\left(d\epsilon^{-1} \right)}{n}},
\end{align}
    and 
    \begin{align} \label{e:c0bound}
   C_{2}\left(\sqrt{d}\alpha_1+\alpha_{2}^{\beta-1}\right)=C_{2} \left[\sqrt{\frac{d\log\left(d\epsilon^{-1} \right)}{n}}+\left( \frac{\log(\epsilon^{-1})}{n}\right)^{\frac{\beta-1}{\beta}} \right],
    \end{align}
where $C_1$ and $C_2$ are constants not depending on $n,\epsilon, d$. 

By Lemma \ref{lem:probE1 re} and Lemma \ref{lem:probE2 re} and applying the generalized Poincar\'e–Miranda theorem \cite[Theorem 2.4]{Fr18} with the consideration of \eqref{e:the0} and \eqref{e:c0bound}, we immediately know that the conclusions in Theorem \ref{thm:jointly main regression} hold true with probability at least $1-6\epsilon$.
\end{proof}

\subsection{Proof of Lemma \ref{lem:probE1 re}}
To estimate $\mathbb{P}(\mathcal{E}_{1})$,  we define the following auxiliary events and will show in the later proof of Lemma \ref{lem:probE1 re} that $\mathcal{E}_{1,1} \cap \mathcal{E}_{1,2} \cap \mathcal{E}_{1,3}$ implies $\mathcal{E}_{1}$.
\begin{align} \label{def:E11E12E13}
    \begin{split}
    & \mathcal{E}_{1,1}:= \left\{\frac{1}{n}\sum_{i=1}^{n}\log\left( 1+\frac{2\alpha_{1}^{2}\varepsilon_{i}^{2}}{v_{0}^{2}} \right)\le 5\alpha_{1}^{2} \right\}; \ \ \ \ 
    \mathcal{E}_{1,2}:=\left\{  \frac{1}{n}\sum_{i=1}^{n}1_{\{|\varepsilon_{i}|\ge \frac{v_{0}}{\alpha_{1}}\}}\le  8\alpha_{1}^{2}\right\};\\
    & \mathcal{E}_{1,3}:= \left\{ \sup_{v\in [v_{0},V_{0}]} \left\|\frac{1}{n}\sum_{i=1}^{n} \boldsymbol{x}_{i}\psi_{1}\left(\frac{\alpha_{1}\varepsilon_{i}}{v} \right) \right\|_{\infty}\le   L\alpha_{1}^{2}\left[\frac{17|V_0-v_0|}{v_0}+3\right]\right\}.
    \end{split}
\end{align}
\begin{lem} \label{lem:probE11,E12,E13}
    Let $\epsilon\in (0,\frac{1}{6})$. Under the assumptions and settings of Theorem \ref{thm:jointly main regression}, it holds that
    $$\mathbb{P}\left(\mathcal{E}_{1,1}\cap\mathcal{E}_{1,2}\cap\mathcal{E}_{1,3}\right)\ge 1-4\epsilon.$$
\end{lem}

\begin{proof}
   we first show that $\mathcal{E}_{1,1}$ and $\mathcal{E}_{1,2}$ are events that hold with probability at least $1-\epsilon$. For $\mathcal{E}_{1,1}$, since $$\mathbb{E}\bigl[\prod_{i}(1+2\alpha_{1}^{2}\varepsilon_{i}^{2}/v_{0}^{2})\bigr]=(1+2\alpha_{1}^{2}\sigma^{2}/v_{0}^{2})^{n}\le e^{2n\alpha_{1}^{2}\sigma^{2}/v_{0}^{2}},$$ 
   Markov's inequality with $\sigma^{2}/v_{0}^{2}\le 2$ and $n\alpha_{1}^{2}=2\log(d\epsilon^{-1})$ implies $$\mathbb{P}(\mathcal{E}_{1,1}^{c})\le \epsilon.$$

    For $\mathcal{E}_{1,2}$, note that
    \begin{align*}
        \mathbb{E}[1_{\{|\varepsilon_{i}|\ge \frac{v_0}{\alpha_{1}}\}  }]\le \frac{\alpha_{1}^{2}\sigma^{2}}{v_0^{2}} \text{ and } \operatorname{Var}\left(1_{\{|\varepsilon_{i}|\ge \frac{v_0}{\alpha_{1}}\}  }\right)\le \mathbb{P}\left(|\varepsilon_{i}|\ge \frac{v_0}{\alpha_{1}}\right)\le \frac{\alpha_{1}^{2}\sigma^{2}}{v_0^{2}}.
    \end{align*}
    It follows from Bernstein’s inequality, $\frac{\sigma^{2}}{v_0^{2}}\le 2$ and the relation $n\alpha_{1}^{2}=2\log(d\epsilon^{-1})$ that
    \begin{align*}
        \mathbb{P}\left(\mathcal{E}_{1,2}^{c}\right)
        &\le \mathbb{P}\left(\sum_{i=1}^{n}1_{\{|\varepsilon_{i}|\ge \frac{v_0}{\alpha_{1}}\}  }-\mathbb{E}\left[1_{\{|\varepsilon_{i}|\ge \frac{v_0}{\alpha_{1}}\}  }\right]
        \ge \frac{3n\alpha_{1}^{2}\sigma^{2}}{v_0^{2}} \right) \le \epsilon. 
    \end{align*} 
    
    Next, as in the proof of Lemma \ref{lem:probE1}, we introduce two additional high-probability events that will be used to ensure $\mathcal{E}_{1,3}$ holds:
    \begin{align*}
        \mathcal{E}_{1,4}&:=\left\{ \frac{1}{n}\sum_{i=1}^{n} \frac{\alpha_{1}^{2}\varepsilon_{i}^{2}}{v_0^{2}}1_{\{\alpha_{1}|\varepsilon_i|\le v_0\} }\le 8\alpha_{1}^{2}  \right\};\\
        \mathcal{E}_{1,5}&:\left\{ \left\| \frac{1}{nL}\sum_{i=1}^{n}\boldsymbol{x}_i\psi_{1}\left(\frac{\alpha_{1}\varepsilon_i}{v_0}\right)\right\|_{\infty}\le   3\alpha_{1}^{2} \right\}.
    \end{align*}
    
    Let $g_i(v):=v\psi_{1}\left(\frac{\alpha_{1}\varepsilon_{i}}{v}\right).$
    Following the same argument as in the proof of \eqref{ineq:g'-2}, it follows that, for $1\le i\le n$,
    \begin{align*}
       \sup_{v\in [v_0,V_0]}\left|\frac{\partial}{\partial v}g_{i}(v)\right|\le M\left[\frac{\alpha_{1}^{2}\varepsilon_i^{2}}{2v_0^{2}}1_{\{\alpha_{1}|\varepsilon_i|\le v_0\} }+\frac{\sigma^{2}}{2v_0^{2}}1_{\{\alpha_{1}|\varepsilon_i|>v_0\} }+\log\left(1+\frac{\alpha_{1}^{2}\varepsilon_i^{2}}{v_0^{2}}\right)\right].
    \end{align*}
    Since $g_i$ are absolutely continuous by Assumption \ref{assumption:c}, we then have, for any $v_1,v_2\in [v_0,V_0]$,
    \begin{align}  \label{ineq:supxpsi}
    \begin{split}
        &\sup_{v_1,v_2\in [v_0,V_0]}\left \|\frac{1}{n}\sum_{i=1}^{n}\boldsymbol{x}_i\left[v_1\psi_{1}\left(\frac{\alpha_{1}\varepsilon_i}{v_1}\right)-v_2\psi_{1}\left(\frac{\alpha_{1} \varepsilon_i}{v_2}\right)\right]\right \|_{\infty}\\
        &\le \frac{1}{n}\sum_{i=1}^{n}\|\boldsymbol{x}_i\|_{\infty} \left|v_1\psi_{1}\left(\frac{\alpha_{1}\varepsilon_i}{v_1}\right)-v_2\psi_{1}\left(\frac{\alpha_{1} \varepsilon_i}{v_2}\right)\right|\\
        &\le \frac{ML|v_1-v_2|}{n}\sum_{i=1}^{n} \left[  \frac{\alpha_{1}^{2}\varepsilon_i^{2}}{2v_0^{2}}1_{\{\alpha_{1}|\varepsilon_i|\le v_0\} }+\frac{\sigma^{2}}{2v_0^{2}}1_{\{\alpha_{1}|\varepsilon_i|>v_0\} }+\log\left(1+\frac{\alpha_{1}^{2}\varepsilon_i^{2}}{v_0^{2}}\right) \right].
    \end{split}
    \end{align}  
    Note that
    \begin{align*}
        \mathbb{E}\left[\frac{\alpha_{1}^{2}\varepsilon_{i}^{2}}{v_0^{2}}1_{\{\alpha_{1}|\varepsilon_i|\le v_0\} }\right]\le \frac{\alpha_{1}^{2}\sigma^{2}}{v_0^{2}} \text{ and }
        \operatorname{Var}\left(\frac{\alpha_{1}^{2}\varepsilon_{i}^{2}}{v_0^{2}}1_{\{\alpha_{1}|\varepsilon_i|\le v_0\} }\right)\le \frac{\alpha_{1}^{2}\sigma^{2}}{v_0^{2}}.
    \end{align*}
    Again, by Bernstein's inequality and $\frac{\sigma^{2}}{v_0^{2}}\le 2$, we obtain
    \begin{align} \label{e:E14Large}
        \mathbb{P}\left( \mathcal{E}_{1,4}\right)\ge 1-\epsilon. 
    \end{align}
    Hence, under the events $\mathcal{E}_{1,1},\mathcal{E}_{1,2}$ and $\mathcal{E}_{1,4}$ , it follows from \eqref{ineq:supxpsi} and $n\alpha_{1}^{2}=2\log(d\epsilon^{-1})$ that
    \begin{align} \label{ineq:supxpsi-2}
         \begin{split}
        \sup_{v_1,v_2\in [v_0,V_0]}\left \|\frac{1}{n}\sum_{i=1}^{n}\boldsymbol{x}_i\left[v_1\psi_{1}\left(\frac{\alpha_{1}\varepsilon_i}{v_1}\right)-v_2\psi_{1}\left(\frac{\alpha_{1} \varepsilon_i}{v_2}\right)\right]\right \|_{\infty}\le 17ML\alpha_{1}^{2}|v_1-v_2|.
    \end{split}
    \end{align}
    On the other hand, for any $1\le k\le d$ and $1 \le i \le n$ , by \eqref{catoni c}, we have
     \begin{align*}
       \exp\left(\frac{x_{ik} }{L}\psi_{1}\left(\frac{\alpha_{1}\varepsilon_{i}}{v_{0}} \right) \right)
        &\le   \left(1+\frac{\alpha_{1}\varepsilon_{i}}{v_{0}}+\frac{\alpha_{1}^{2}\varepsilon_{i}^{2}}{2v_{0}^{2}} \right)^{\frac{|x_{ik}|}{L}} 
    \end{align*}
     where $x_{ik}$ is the $k$-th component of $\boldsymbol{x}_{i}$. This, together with the independence of $\varepsilon_i$, Jensen's inequality  and $\frac{|x_{ik}|}{L} \le 1$ implies that
        \begin{align*}
        \mathbb{E}\left[\exp\left(\sum_{i=1}^{n} \frac{x_{ik} }{L}\psi_{1}\left(\frac{\alpha_{1}\varepsilon_{i}}{v_0} \right) \right)\right] \le  \prod_{i=1}^{n} \left(1+\frac{\alpha_{1}^{2}\sigma^{2}}{2v_{0}^{2}} \right)^{\frac{|x_{ik}|}{L}} \le \exp\left[\frac{n\alpha_{1}^{2}\sigma^{2}}{2v_0^{2}}\right].
    \end{align*}
    It then follows from the Markov's inequality that
    \begin{align} \label{ineq:xpsiupper}
        \mathbb{P}\left(\sum_{i=1}^{n} \frac{x_{ik} }{L}\psi_{1}\left(\frac{\alpha_{1}\varepsilon_{i}}{v_0} \right)  \ge \frac{ n\alpha_{1}^{2}\sigma^{2}}{2v_0^{2}}   +\log\left(2d\epsilon^{-1}\right)\right)
        \le \frac{\epsilon}{2d}.
    \end{align}
 Similarly, for any $1\le k\le d$, 
    \begin{align} \label{ineq:xpsilower}
         \mathbb{P}\left(-\sum_{i=1}^{n} \frac{x_{ik} }{L}\psi_{1}\left(\frac{\alpha_{1}\varepsilon_{i}}{v_0} \right)  \ge   \frac{ n\alpha_{1}^{2}\sigma^{2}}{2v_0^{2}}   +\log\left(2d\epsilon^{-1}\right)   \right)
        \le \frac{\epsilon}{2d}.
    \end{align}
    Combining \eqref{ineq:xpsilower} and \eqref{ineq:xpsiupper} with $\frac{\sigma^{2}}{v_0^{2}}\le 2$ and $n\alpha_{1}^{2}=2\log(d\epsilon^{-1})$ gives us that
    \begin{align} \label{e:E15Large}
        \mathbb{P}\left( \mathcal{E}_{1,5} \right) \ge 1-\epsilon.
    \end{align}
    This, together with \eqref{ineq:supxpsi-2} implies that, 
    under the events $\mathcal{E}_{1,1},\mathcal{E}_{1,2},\mathcal{E}_{1,4}$ and $\mathcal{E}_{1,5}$,
    \begin{align*}
        \sup_{v\in [v_{0},V_{0}]} \left\|\frac{1}{n}\sum_{i=1}^{n} \boldsymbol{x}_{i}\psi_{1}\left(\frac{\alpha_{1}\varepsilon_{i}}{v} \right) \right\|_{\infty}
        &\le \sup_{v\in [v_0,V_0]}\left \|\frac{1}{nv_0}\sum_{i=1}^{n}\boldsymbol{x}_i\left[v\psi_{1}\left(\frac{\alpha_{1}\varepsilon_i}{v}\right)-v_0\psi_{1}\left(\frac{\alpha_{1} \varepsilon_i}{v_0}\right)\right]\right \|_{\infty}\\
        &\quad + \left\|\frac{1}{nv_0}\sum_{i=1}^{n} \boldsymbol{x}_{i}v_0\psi_{1}\left(\frac{\alpha_{1}\varepsilon_{i}}{v_0} \right) \right\|_{\infty}\\
        &\le L\alpha_{1}^{2}\left[\frac{17M|V_0-v_0|}{v_0}+3\right],
    \end{align*}
 which is exactly the event $\mathcal{E}_{1,3}$. This completes the proof. 
\end{proof}

\begin{proof}[Proof of Lemma \ref{lem:probE1 re}]
    Let $\mathcal{E}_{1,1},\mathcal{E}_{1,2}$ and $\mathcal{E}_{1,3}$ be defined as in \eqref{def:E11E12E13}. By Lemma \ref{lem:probE11,E12,E13}, it suffices to show that  there exists a constant $C=C(m_{2\beta},\sigma,\beta,c_{l}, c_{u}, L)$ such that as long as $n\ge Cd[\log(\epsilon^{-1})+1]$,
    \begin{equation} \label{e:E123ToE}
    {\rm the \ event} \ \mathcal{E}_{1,1}\cap \mathcal{E}_{1,2}\cap\mathcal{E}_{1,3} \ {\rm implies} \  \mathcal{E}_{1}.
    \end{equation}

By the monotonicity of $\psi_{1}$, $-x[\psi_{1}(x+y)-\psi_{1}(y)]\le 0$ for $x,y\in\mathbb{R}$. Denote
\[
r_{i}(\boldsymbol{\theta},v):=\langle\boldsymbol{x}_{i},\boldsymbol{\theta}-\boldsymbol{\theta}^{*}\rangle\,\psi_{1}\!\Bigl(\tfrac{\alpha_{1}(\varepsilon_{i}+\langle\boldsymbol{x}_{i},\boldsymbol{\theta}^{*}-\boldsymbol{\theta}\rangle)}{v}\Bigr), \quad i=1,...,n.\] 
Clearly, $\<\tilde{f}_{1}(\boldsymbol{\theta},v), \boldsymbol{\theta}-\boldsymbol{\theta}^{*}\>=\frac{1}{n}\sum_{i=1}^{n}r_{i}(\boldsymbol{\theta},v).$
We may decompose $r_{i}(\boldsymbol{\theta},v)\le r_{i}^{(1)}(\boldsymbol{\theta},v)+r_{i}^{(2)}(\boldsymbol{\theta},v)$
with
$$
r_{i}^{(1)}(\boldsymbol{\theta},v)=\langle\boldsymbol{x}_{i},\boldsymbol{\theta}-\boldsymbol{\theta}^{*}\rangle\!\Bigl[\psi_{1}\!\bigl(\tfrac{\alpha_{1}(\varepsilon_{i}+\langle\boldsymbol{x}_{i},\boldsymbol{\theta}^{*}-\boldsymbol{\theta}\rangle)}{v}\bigr)-\psi_{1}\!\bigl(\tfrac{\alpha_{1}\varepsilon_{i}}{v}\bigr)\Bigr]\mathbf{1}_{\{|\varepsilon_{i}|\le v/\alpha_{1}\}},$$
$$r_{i}^{(2)}(\boldsymbol{\theta},v)=\langle\boldsymbol{x}_{i},\boldsymbol{\theta}-\boldsymbol{\theta}^{*}\rangle\,\psi_{1}\left(\frac{\alpha_{1}\varepsilon_{i}}v\right).$$

\emph{Bounding $r_{i}^{(1)}(\boldsymbol{\theta},v)$}.
 Applying \eqref{e:XPsi1Y} to $\<\boldsymbol{x}_{i}, \boldsymbol{\theta}-\boldsymbol{\theta}^{*} \>\psi_{1}\left(\frac{\alpha_{1}\left(\varepsilon_{i}+\<\boldsymbol{x}_{i}, \boldsymbol{\theta}^{*}-\boldsymbol{\theta}\> \right)}{v}\right)$ and $-\<\boldsymbol{x}_{i},\boldsymbol{\theta}-\boldsymbol{\theta}^{*}\> \psi_{1}\left( \frac{\alpha_{1}\varepsilon_{i}}{v} \right)$, we have
 
\begin{equation} \label{e:R1i}
 r^{(1)}_{i}(\boldsymbol{\theta},v) \le |\<\boldsymbol{x}_{i}, \boldsymbol{\theta}-\boldsymbol{\theta}^{*} \>| \Psi_i 1_{\{|\varepsilon_i| \le \frac{v}{\alpha_1}\}}
\end{equation}
with $\psi_{+}(x):=\log(1+x+\frac{x^{2}}{2})$ and
$$\Psi_i=\psi_{+}\left(\operatorname{sign}\left(\<\boldsymbol{x}_{i}, \boldsymbol{\theta}-\boldsymbol{\theta}^{*} \> \right)  \frac{\alpha_{1}\left(\varepsilon_{i}+\<\boldsymbol{x}_{i}, \boldsymbol{\theta}^{*}-\boldsymbol{\theta}\> \right)}{v}\right)+\psi_{+}\left( -\operatorname{sign}\left(\<\boldsymbol{x}_{i}, \boldsymbol{\theta}-\boldsymbol{\theta}^{*} \> \right)\frac{\alpha_{1}\varepsilon_{i}}{v} \right).$$
 Moreover, whenever $|\varepsilon_{i}|\le \frac{v}{\alpha_{1}}$, we have 
    \begin{align*} 
        \begin{split}
        \Psi_i
        &\le \log \bigg{(} 1-\frac{\alpha_{1}|\<\boldsymbol{x}_{i}, \boldsymbol{\theta}-\boldsymbol{\theta}^{*} \>|}{v} +\frac{\alpha_{1}^{2}|\<\boldsymbol{x}_{i}, \boldsymbol{\theta}-\boldsymbol{\theta}^{*} \>|^{2}}{2v^{2}}+\frac{\alpha_{1}^{2} \left(\varepsilon_{i}+\<\boldsymbol{x}_{i}, \boldsymbol{\theta}^{*}-\boldsymbol{\theta}\> \right)^{2}}{4v^{2}}
        \\
        &\qquad \qquad +\frac{\alpha_{1}^{2}}{2v^{2}}\bigg{[}|\varepsilon_{i}||\<\boldsymbol{x}_{i}, \boldsymbol{\theta}-\boldsymbol{\theta}^{*} \>|+|\<\boldsymbol{x}_{i}, \boldsymbol{\theta}-\boldsymbol{\theta}^{*} \>|^{2} \bigg{]}\bigg{)}\\
        &\le \log \bigg{(} 1-\frac{\alpha_{1}|\<\boldsymbol{x}_{i}, \boldsymbol{\theta}-\boldsymbol{\theta}^{*} \>|}{v} +\frac{2\alpha_{1}^{2}|\<\boldsymbol{x}_{i}, \boldsymbol{\theta}-\boldsymbol{\theta}^{*} \>|^{2}}{v^{2}}+\frac{2\alpha_{1}^{2}\varepsilon_{i}^{2} }{v^{2}} \bigg{)}.
        \end{split}
    \end{align*}
    where the last inequality is by $|\varepsilon_{i}||\<\boldsymbol{x}_{i}, \boldsymbol{\theta}-\boldsymbol{\theta}^{*} \>| \le \frac12 |\varepsilon_{i}|^2+\frac 12 |\<\boldsymbol{x}_{i}, \boldsymbol{\theta}-\boldsymbol{\theta}^{*} \>|^2.$
    Using the inequality \eqref{ineq:1+x+y} 
    with $x=\frac{2\alpha_{1}^{2}\varepsilon_{i}^{2} }{v^{2}}$ and
        $y=-\frac{\alpha_{1}|\<\boldsymbol{x}_{i}, \boldsymbol{\theta}-\boldsymbol{\theta}^{*} \>|}{v} +\frac{2\alpha_{1}^{2}|\<\boldsymbol{x}_{i}, \boldsymbol{\theta}-\boldsymbol{\theta}^{*} \>|^{2}}{v^{2}}$ therein, 
    yields that as long as $|\varepsilon_{i}|\le \frac{v}{\alpha_{1}}$,
    \begin{align} \label{ineq:r(1)upper-3}
        \begin{split}
       \Psi_i \le  -\frac{\alpha_{1}|\<\boldsymbol{x}_{i}, \boldsymbol{\theta}-\boldsymbol{\theta}^{*} \>|}{3v} +\frac{2\alpha_{1}^{2}|\<\boldsymbol{x}_{i}, \boldsymbol{\theta}-\boldsymbol{\theta}^{*} \>|^{2}}{v^{2}} +\log \bigg{(} 1+\frac{2\alpha_{1}^{2}\varepsilon_{i}^{2} }{v^{2}} \bigg{)}.
        \end{split}
    \end{align}
    Combining \eqref{ineq:r(1)upper-3} with \eqref{e:R1i} gives that, for any $1\le i\le n$,
    \begin{align*} 
        \begin{split}
        r^{(1)}_{i}(\boldsymbol{\theta,v})
        \le   &\bigg{[} -\frac{\alpha_{1}|\<\boldsymbol{x}_{i}, \boldsymbol{\theta}-\boldsymbol{\theta}^{*} \>|}{3v} +\frac{2\alpha_{1}^{2}|\<\boldsymbol{x}_{i}, \boldsymbol{\theta}-\boldsymbol{\theta}^{*} \>|^{2}}{v^{2}}+\log \bigg{(} 1+\frac{2\alpha_{1}^{2}\varepsilon_{i}^{2} }{v^{2}} \bigg{)}\bigg{]} |\<\boldsymbol{x}_{i}, \boldsymbol{\theta}-\boldsymbol{\theta}^{*} \>|1_{\{|\varepsilon_{i}|\le \frac{v}{\alpha_{1}}\}}.
        \end{split}
    \end{align*} 
     Further using $L=\max_{1\le i\le n} \|\boldsymbol{x}_{i}\|_{2}$ and $|\<\boldsymbol{x}_{i}, \boldsymbol{\theta}-\boldsymbol{\theta}^{*} \>|\le L\|\boldsymbol{\theta}-\boldsymbol{\theta}^{*}\|_{2}$ yields that
       \begin{align*} 
        \begin{split}
            r^{(1)}_{i}(\boldsymbol{\theta,v})
        &\le -\frac{\alpha_{1}|\<\boldsymbol{x}_{i}, \boldsymbol{\theta}-\boldsymbol{\theta}^{*} \>|^{2}}{3v} +\frac{L^{2}\alpha_{1} \|\boldsymbol{\theta}-\boldsymbol{\theta}^{*}\|_{2}^{2}}{3v}1_{\{|\varepsilon_{i}|\ge \frac{v}{\alpha_{1}}\}}\\
        &\quad +L\|\boldsymbol{\theta}-\boldsymbol{\theta}^{*}\|_{2}\left(\frac{2\alpha_{1}^{2}|\<\boldsymbol{x}_{i}, \boldsymbol{\theta}-\boldsymbol{\theta}^{*} \>|^{2}}{v^{2}}+\log \bigg{(} 1+\frac{2\alpha_{1}^{2}\varepsilon_{i}^{2} }{v^{2}} \bigg{)} \right)
        ,
        \end{split}
    \end{align*}
     Note that Assumption \ref{A1} implies that
    \begin{align*}
        nc_{l}\|\boldsymbol{\theta}-\boldsymbol{\theta}^{*}\|_{2}^{2}\le \sum_{i=1}^{n}\<\boldsymbol{x}_{i},\boldsymbol{\theta}^{*}-\boldsymbol{\theta}\>^{2}\le nc_{u}\|\boldsymbol{\theta}-\boldsymbol{\theta}^{*}\|_{2}^{2}.
    \end{align*}
    So summing up all the $r^{(1)}_{i}(\boldsymbol{\theta},v)$ over $i$ gives us
    \begin{align*}
        \frac{1}{n}\sum_{i=1}^{n} r^{(1)}_{i}(\boldsymbol{\theta},v)
        &\le -\frac{\alpha_{1}c_{l}\|\boldsymbol{\theta}-\boldsymbol{\theta}^{*}\|_{2}^{2}}{3v} 
        +\frac{L^{2}\alpha_{1} \|\boldsymbol{\theta}-\boldsymbol{\theta}^{*}\|_{2}^{2}}{3nv}\sum_{i=1}^{n}1_{\{|\varepsilon_{i}|\ge \frac{v}{\alpha_{1}}\}}\\
        &\quad 
        + \frac{L\|\boldsymbol{\theta}-\boldsymbol{\theta}^{*}\|_{2}}{n} \sum_{i=1}^{n}\log \bigg{(} 1+\frac{2\alpha_{1}^{2}\varepsilon_{i}^{2} }{v^{2}} \bigg{)}+\frac{2\alpha_{1}^{2}c_{u}L\|\boldsymbol{\theta}-\boldsymbol{\theta}^{*}\|_{2}^{3}}{v^{2}}
        .
    \end{align*}
    Since $\|\boldsymbol{\theta}-\boldsymbol{\theta}^{*}\|_{2}=\theta_{0}$ for any $\boldsymbol{\theta}\in \partial\mathcal{B}_{\theta_{0}}(\boldsymbol{\theta}^{*})$, taking the supremum with respect to $\boldsymbol{\theta}$ and $v$, we obtain
    \begin{align*}
        &\sup_{v\in [v_{0},V_{0}]} \sup_{\boldsymbol{\theta}\in \partial \mathcal{B}_{\theta_{0}}(\boldsymbol{\theta}^{*})}\left\{\frac{1}{n}\sum_{i=1}^{n} r^{(1)}_{i}(\boldsymbol{\theta},v)\right\}\\
        &\le L\theta_{0}\left[-\frac{\alpha_{1}}{3L}\left(\frac{c_{l}}{V_{0}}-\frac{L^{2} }{nv_{0}}\sum_{i=1}^{n}1_{\{|\varepsilon_{i}|\ge \frac{v_{0}}{\alpha_{1}}\}}\right)\theta_{0} 
        + \frac{1}{n} \sum_{i=1}^{n}\log \bigg{(} 1+\frac{2\alpha_{1}^{2}\varepsilon_{i}^{2} }{v_{0}^{2}} \bigg{)}
        +\frac{2\alpha_{1}^{2}c_{u}}{v_{0}^{2}}\theta_{0}^{2}\right].
    \end{align*} 
    Moreover, it follows from the occurrence of $\mathcal{E}_{1,1}\cap\mathcal{E}_{1,2}\cap\mathcal{E}_{1,3}$ and the definitions of these events in \eqref{def:E11E12E13} that
    \begin{align} \label{ineq:sumr(1)i}
    \begin{split}
        \sup_{v\in [v_{0},V_{0}]} \sup_{\boldsymbol{\theta}\in \partial \mathcal{B}_{\theta_{0}}(\boldsymbol{\theta}^{*})}\left\{\frac{1}{n}\sum_{i=1}^{n} r^{(1)}_{i}(\boldsymbol{\theta},v)\right\}\le L\theta_{0}\alpha_{1}\left[-\frac{1}{3L}\left(\frac{c_{l}}{V_{0}}-\frac{8L^{2}\alpha_{1}^{2} }{v_{0}}\right)\theta_{0} +\frac{2\alpha_{1}c_{u}}{v_{0}^{2}}\theta_{0}^{2}+ 5\alpha_{1} 
        \right],
    \end{split}
\end{align}
\par

\emph{Bounding $r_i^{(2)}(\boldsymbol{\theta},v)$}. By $\|\boldsymbol{\theta}-\boldsymbol{\theta}^{*}\|_{1}\le \sqrt{d}\|\boldsymbol{\theta}-\boldsymbol{\theta}^{*}\|_{2}$, we have 
    \begin{align*}
        \frac{1}{n}\sum_{i=1}^{n}r^{(2)}_{i}(\boldsymbol{\theta},v)
        \le \left\|\frac{1}{n}\sum_{i=1}^{n}\boldsymbol{x}_{i}\psi_{1}\left( \frac{\alpha_{1}\varepsilon_{i}}{v} \right) \right\|_{\infty} \|\boldsymbol{\theta}-\boldsymbol{\theta}^{*}\|_{1} \le \sqrt{d}\|\boldsymbol{\theta}-\boldsymbol{\theta}^{*}\|_{2}\left\|\frac{1}{n}\sum_{i=1}^{n}\boldsymbol{x}_{i}\psi_{1}\left( \frac{\alpha_{1}\varepsilon_{i}}{v} \right) \right\|_{\infty}.
    \end{align*}
     Since $\|\boldsymbol{\theta}-\boldsymbol{\theta}^{*}\|_{2}=\theta_{0}$ for any $\boldsymbol{\theta}\in \partial\mathcal{B}_{\theta_{0}}(\boldsymbol{\theta}^{*})$, taking the supremum with respect to $v$ and $\boldsymbol{\theta}$ gives 
    \begin{align*}
        \sup_{v\in [v_{0},V_{0}]} \sup_{\boldsymbol{\theta}\in \partial \mathcal{B}_{\theta_{0}}(\boldsymbol{\theta}^{*})}\left\{\frac{1}{n}\sum_{i=1}^{n}r^{(2)}_{i}(\boldsymbol{\theta},v)\right\}
        &\le \sqrt{d}\theta_{0} \sup_{v\in [v_{0},V_{0}]}\left\|\frac{1}{n}\sum_{i=1}^{n}\boldsymbol{x}_{i}\psi_{1}\left( \frac{\alpha_{1}\varepsilon_{i}}{v} \right) \right\|_{\infty}.
    \end{align*}
    Further more, by the occurrence of $\mathcal{E}_{1,1}\cap\mathcal{E}_{1,2}\cap\mathcal{E}_{1,3}$ and the definition of $\mathcal{E}_{1,3}$ in \eqref{def:E11E12E13},
    \begin{align} \label{ineq:sumr(2)i}
    \sup_{v\in [v_{0},V_{0}]} \sup_{\boldsymbol{\theta}\in \partial \mathcal{B}_{\theta_{0}}(\boldsymbol{\theta}^{*})}\left\{\frac{1}{n}\sum_{i=1}^{n}r^{(2)}_{i}(\boldsymbol{\theta},v)\right\}
    \le \sqrt{d}L\theta_{0}\alpha_{1}^{2}\left( \frac{17M|V_0-v_0|}{v_0}+3\right).
\end{align}

Combining \eqref{ineq:sumr(1)i} and \eqref{ineq:sumr(2)i}, we have
\begin{align*}
    &\sup_{v\in [v_{0},V_{0}]} \sup_{\boldsymbol{\theta}\in \partial \mathcal{B}_{\theta_{0}}(\boldsymbol{\theta}^{*})} \left[\frac1n\sum_{i}r_{i}(\boldsymbol{\theta},v)\right]\\
    &\le L\theta_{0}\alpha_{1}\bigg{[}-\frac{1}{3L}\left(\frac{c_{l}}{V_{0}}-\frac{8L^{2}\alpha_{1}^{2} }{v_{0}}\right)\theta_{0} +\frac{2\alpha_{1}c_{u}}{v_{0}^{2}}\theta_{0}^{2}+ 5\alpha_{1}+\sqrt{d}\alpha_{1}\left(\frac{17M|V_0-v_0|}{v_0}+3\right)\bigg{]}.
\end{align*}
Recall that $\alpha_{1}=O\left(\sqrt{\frac{\log(d\epsilon^{-1})}{n}}\right)$, $\alpha_{2}=O\left(\left(\frac{\log(\epsilon^{-1})}{n}\right)^{\frac{1}{\beta}}\right)$ and $V_{0}-v_0=O(\sigma \alpha_{2}^{\beta-1})$. There must exist a $C=C(m_{2\beta},\sigma,\beta,c_{l},  c_{u}, L,M)>0$ such that 
\begin{align} \label{ineq:suptf1}
    &\sup_{v\in [v_{0},V_{0}]} \sup_{\boldsymbol{\theta}\in \partial \mathcal{B}_{\theta_{0}}(\boldsymbol{\theta}^{*})} \left[\frac1n\sum_{i}r_{i}(\boldsymbol{\theta},v)\right]
    \le L\theta_{0}\alpha_{1}\bigg{[}-\frac{c_l}{6V_0L}\theta_{0}+9\sqrt{d}\alpha_{1}\bigg{]}=0,
\end{align}
provided $n\ge Cd[\log(\epsilon^{-1})+1]$. So the proof is complete.
\end{proof}

\subsection{Proof of Lemma \ref{lem:probE2 re}}

\begin{proof}[Proof of Lemma \ref{lem:probE2 re}]
We shall show that for any $\epsilon \in (0,\frac 16)$, there exist some $C=C(m_{2\beta},\sigma,\beta,c_{l}, c_{u}, L)$ such that as $n\ge Cd[\log(\epsilon^{-1})+1]$ the following two claims hold:

\textbf{Claim 1 (concentration)}: We have
\begin{equation} \label{ineq:tf2upper re}
        \mathbb{P}\left(\sup_{\boldsymbol{\theta}\in \mathcal{B}_{\theta_0}(\boldsymbol{\theta}^*)}\frac{\tilde{f}_{2}(\boldsymbol{\theta},v)}{\alpha_{2}}
        \le \tilde B^{v}_{+}(v)\right) \ge 1-\epsilon,  \ \ \ \ \ \ \forall v>0; 
    \end{equation}
where
    \begin{align} \label{def:tBv+}
    \begin{split}
   & \tilde{B}^{v}_{+}(v):=\frac{\left(1+\theta_{0} \right)\sigma^{2}}{v^{2}} -1+\frac{(1+\frac{1}{\theta_{0}}) c_{u}\theta_{0}^{2}}{(1-\alpha_{2})v^{2}}\\
     &\quad \ \  +\frac{2^{4\beta-3}\alpha_{2}^{\beta-1}}{\beta}\left( \frac{m_{2\beta}+\sigma^{2\beta}}{v^{2\beta}}+\frac{ c_{u}L^{2\beta-2} }\theta_{0}^{2\beta}{(1-\alpha_{2})v^{2\beta}}+\left|\frac{\sigma^{2}}{v^{2}}-1 \right|^{\beta} \right)+\frac{\log(\epsilon^{-1})}{n\alpha_{2}};
    \end{split}
    \end{align}
and
    \begin{equation} \label{ineq:tf2lower re}
    \mathbb{P}\left(\inf_{\boldsymbol{\theta}\in \mathcal{B}_{\theta_0}(\boldsymbol{\theta}^*)}\frac{\tilde{f}_{2}(\boldsymbol{\theta},v)}{\alpha_{2}}
        \ge \tilde B^{v}_{-}(v)\right) \ge 1-\epsilon, \ \ \ {\forall v>0,}
    \end{equation}
    where
    \begin{align} \label{def:tBv-}
    \begin{split}
    \tilde{B}^{v}_{-}(v) &:=\frac{\left(1-\theta_{0} \right)\sigma^{2}}{v^{2}}-1 -\frac{(\frac{1}{\theta_{0}}-1)\beta c_{u}\theta_{0}^{2}}{v^{2}}\\
     &\quad -\frac{2^{4\beta-3}\alpha_{2}^{\beta-1}}{\beta}\left( \frac{m_{2\beta}+\sigma^{2\beta}}{v^{2\beta}}+\frac{ \beta c_{u}L^{2\beta-2}}\theta_{0}^{2\beta}{v^{2\beta}}+\left|\frac{\sigma^{2}}{v^{2}}-1 \right|^{\beta} \right)-\frac{\log(\epsilon^{-1})}{n\alpha_{2}}. 
    \end{split}
    \end{align}

\textbf{Claim 2 (existence of $v_{\pm}$)}: There exist some $0<v_+<V_0$ and some $v_{-}>v_0$ such that 
\begin{equation} \label{e:TilBv+=0}
\tilde{B}^{v}_{+}(v_+) \le 0,
\end{equation}
\begin{equation} \label{e:TilBv-=0}
\tilde{B}^{v}_{-}(v_{-}) \ge 0.
\end{equation}
 Assume that the above two claims hold true.
 Let us now use \eqref{ineq:tf2upper re} and \eqref{e:TilBv+=0} to prove the first inequality in the lemma.  Since for any $\boldsymbol{\theta}$, $\tilde{f}_{2}(\boldsymbol{\theta},v)$ is monotonically decreasing with respect to $v$, by \eqref{ineq:tf2upper re}, it holds with probability at least $1-\epsilon$ that,
   \begin{align*}
       \sup_{\boldsymbol{\theta}\in \mathcal{B}_{\theta_{0}}(\boldsymbol{\theta}^{*})}\tilde{f}_{2}(\boldsymbol{\theta},V_{0}) \le \sup_{\boldsymbol{\theta}\in \mathcal{B}_{\theta_{0}}(\boldsymbol{\theta}^{*})}\tilde{f}_{2}(\boldsymbol{\theta},v_{+})\le \alpha_{2} \tilde{B}^{v}_{+}(v_{+}) \le 0.
   \end{align*}
For the second inequality in the lemma, we can use \eqref{ineq:tf2lower re} and \eqref{e:TilBv-=0} to prove it in the same way. It remains to prove the two claims.

   \underline{\textbf{Proof of Claim 1}} It follows from the definition of $\tilde{f}_{2}$ in \eqref{f1f2regression} and \eqref{catoni c2} that
    \begin{align} \label{ineq:tf2upper}
    \begin{split}
        &\mathbb{E}\left[\sup_{\boldsymbol{\theta}\in \mathcal{B}_{\theta_0}\left(\boldsymbol{\theta}^*\right)} \left\{\exp\left( n \tilde{f}_{2}(\boldsymbol{\theta},v)  \right) \right\} \right]
        \le \mathbb{E}\left[ \sup_{\boldsymbol{\theta}\in \mathcal{B}_{\theta_0}\left(\boldsymbol{\theta}^*\right)}\left\{\prod_{i=1}^{n}I_{i}\right\}\right] ,\quad \forall v>0,
    \end{split}
    \end{align}
    with
    \begin{align} \label{def:Ii}
        I_{i}:=1+\alpha_{2} \left(\frac{\left(\varepsilon_{i}+\< \boldsymbol{x}_{i},\boldsymbol{\theta}^{*}-\boldsymbol{\theta}\>\right)^{2}}{v^{2}}-1\right)+\frac{\alpha_{2}^{\beta}}{\beta}\left| \frac{\left(\varepsilon_{i}+\<\boldsymbol{x}_{i},\boldsymbol{\theta}^{*}-\boldsymbol{\theta}\> \right)^{2}}{v^{2}}-1\right|^{\beta} .
    \end{align}
    By the basic inequality $2xy\le x^{2}+y^{2}$, 
    \begin{align} \label{ineq:tf2-1}
        \left(\varepsilon_{i}+\< \boldsymbol{x}_{i},\boldsymbol{\theta}^{*}-\boldsymbol{\theta}\> \right)^{2}\le \left(1+\theta_{0} \right)\varepsilon_{i}^{2}+\left(1+\frac{1}{\theta_{0}} \right)\< \boldsymbol{x}_{i},\boldsymbol{\theta}^{*}-\boldsymbol{\theta}\>^{2},
    \end{align}
    and by the convexity inequality \eqref{ineq:convexity} twice,
    \begin{align} \label{ineq:tf2-2}
        \begin{split}
        \left| \frac{\left(\varepsilon_{i}+\<\boldsymbol{x}_{i},\boldsymbol{\theta}^{*}-\boldsymbol{\theta}\> \right)^{2}}{v^{2}}-1\right|^{\beta} 
    \le 2^{4\beta-3}\left( \frac{\left|\varepsilon_{i}\right|^{2\beta}+\left|\<\boldsymbol{x}_{i},\boldsymbol{\theta}^{*}-\boldsymbol{\theta}\> \right|^{2\beta}+\sigma^{2\beta}}{v^{2\beta}}+\left|\frac{\sigma^{2}}{v^{2}}-1 \right|^{\beta} \right).
        \end{split}
    \end{align}
Hence, 
    \begin{align*}
        \begin{split}
         I_{i}
        &\le  1+I^{(1)}_{i}+I^{(2)}_{i}, \quad 
        \quad 1 \le i \le n,
        \end{split}
    \end{align*}
    where 
    \begin{align*}
        I^{(1)}_{i}:=\alpha_{2}\left(\frac{\left(1+\theta_{0} \right)\varepsilon_{i}^{2}}{v^{2}} -1\right) +\frac{2^{4\beta-3}\alpha_{2}^{\beta}}{\beta}\left( \frac{\left|\varepsilon_{i}\right|^{2\beta}+\sigma^{2\beta}}{v^{2\beta}}+\left|\frac{\sigma^{2}}{v^{2}}-1 \right|^{\beta} \right),
    \end{align*}
    \begin{align*}
        I^{(2)}_{i}:=\frac{\alpha_{2}\left(1+\frac{1}{\theta_{0}} \right)\< \boldsymbol{x}_{i},\boldsymbol{\theta}^{*}-\boldsymbol{\theta}\>^{2}}{v^{2}}+\frac{2^{4\beta-3}\alpha_{2}^{\beta}\left|\<\boldsymbol{x}_{i},\boldsymbol{\theta}^{*}-\boldsymbol{\theta}\> \right|^{2\beta}}{\beta v^{2\beta}}. 
    \end{align*}
Note that  $I^{(1)}_{i}$  is stochastic and $I^{(2)}_{i}$ is related to $\boldsymbol{\theta}$. To separate the random terms from the terms containing $\boldsymbol{\theta}$, using \eqref{ineq:1+x+y} and the fact that $I^{(1)}_{i}\ge -\alpha_{2}$ gives us
    \begin{align} \label{ineq:Ii}
        \begin{split}
        I_{i}\le \left(1+I^{(1)}_{i}+I^{(2)}_{i}\right)
        \le (1+I^{(1)}_{i})\exp\left(\frac{I^{(2)}_{i}}{1+I^{(1)}_{i}} \right)\le (1+I^{(1)}_{i})\exp\left(\frac{I^{(2)}_{i}}{1-\alpha_{2}} \right).
        \end{split}
    \end{align}
    Combining this with \eqref{ineq:tf2upper} and using the independence of $\varepsilon_{i}$ yields
    \begin{align*} 
    \begin{split}
        \mathbb{E}\left[\sup_{\boldsymbol{\theta}\in \mathcal{B}_{\theta_0}\left(\boldsymbol{\theta}^*\right)} \left\{\exp\left( n \tilde{f}_{2}(\boldsymbol{\theta},v)  \right) \right\} \right]
        &\le \left\{\prod_{i=1}^{n}\mathbb{E} \left[  (1+I^{(1)}_{i})\right]\right\} 
         \exp\left(\frac{\sup_{\boldsymbol{\theta}\in \mathcal{B}_{\theta_0}\left(\boldsymbol{\theta}^*\right)} \left\{\sum_{i=1}^{n}I^{(2)}_{i}\right\}}{1-\alpha_{2}} \right).
    \end{split}
    \end{align*}
    On the one hand, we have  
    $ 
        \mathbb{E}\left[ I^{(1)}_{i} \right]=\alpha_{2}\left(\frac{\left(1+\theta_{0} \right)\sigma^{2}}{v^{2}} -1\right) +\frac{2^{4\beta-3}\alpha_{2}^{\beta}}{\beta}\left( \frac{m_{2\beta}+\sigma^{2\beta}}{v^{2\beta}}+\left|\frac{\sigma^{2}}{v^{2}}-1 \right|^{\beta} \right)
    $ and thus
\begin{align*} 
        \begin{split}
        \prod_{i=1}^{n}\mathbb{E} \left[  (1+I^{(1)}_{i})\right] 
      \le  \exp\left[n\alpha_{2}\left(\frac{\left(1+\theta_{0} \right)\sigma^{2}}{v^{2}} -1\right) +\frac{2^{4\beta-3}\alpha_{2}^{\beta}n}{\beta}\left( \frac{m_{2\beta}+\sigma^{2\beta}}{v^{2\beta}}+\left|\frac{\sigma^{2}}{v^{2}}-1 \right|^{\beta} \right)\right]
        \end{split}
\end{align*}
    On the other hand, for $I^{(2)}_{i}$, 
     it follows from Assumption \ref{A1}  and 
     $\left|\< \boldsymbol{x}_{i},\boldsymbol{\theta}^{*}-\boldsymbol{\theta}\>\right|\le L\|\boldsymbol{\theta}^{*}-\boldsymbol{\theta}\|_{2}$ 
     that
 $ \sum_{i=1}^{n} \< \boldsymbol{x}_{i},\boldsymbol{\theta}^{*}-\boldsymbol{\theta}\>^{2}\le nc_{u}\|\boldsymbol{\theta}^{*}-\boldsymbol{\theta}\|_{2}^{2}$ and $\sum_{i=1}^{n} \< \boldsymbol{x}_{i},\boldsymbol{\theta}^{*}-\boldsymbol{\theta}\>^{2\beta}\le  nc_{u}L^{2\beta-2}\|\boldsymbol{\theta}^{*}-\boldsymbol{\theta}\|_{2}^{2\beta}$, hence
    \begin{align*} 
        \begin{split}
        &\sup_{_{\boldsymbol{\theta}\in \mathcal{B}_{\theta_0}\left(\boldsymbol{\theta}^*\right)}}\left\{ \sum_{i=1}^{n} I^{(2)}_{i}\right\}\le \frac{\alpha_{2}\left(1+\frac{1}{\theta_{0}} \right)nc_{u}\theta_{0}^{2}}{v^{2}}+\frac{2^{4\beta-3}\alpha_{2}^{\beta} nc_{u} L^{2\beta-2}\theta_{0}^{2\beta}}{\beta v^{2\beta}} .
        \end{split}
    \end{align*} 
Therefore,
    \begin{align*}
        &\mathbb{E}\left[\sup_{\boldsymbol{\theta}\in \mathcal{B}_{\theta_0}\left(\boldsymbol{\theta}^*\right)} \left\{\exp\left( n \tilde{f}_{2}(\boldsymbol{\theta},v)  \right) \right\} \right] \le  \exp\left(n\alpha_{2}\tilde B^{v}_{+}(v)-\log(\epsilon^{-1})  \right),
    \end{align*}
    where $\tilde B^{v}_{+}(v)$ is defined as \eqref{def:tBv+}. 
    Using the Markov's inequality as before yields \eqref{ineq:tf2upper re}. 
        
    Repeating the same argument for $-f_{2}$ in which \eqref{ineq:tf2-1} is replaced with 
    \begin{align*}
        -\left(\varepsilon_{i}+\< \boldsymbol{x}_{i},\boldsymbol{\theta}^{*}-\boldsymbol{\theta}\> \right)^{2}\le -\left(1-\theta_{0} \right)\varepsilon_{i}^{2}+\left(\frac{1}{\theta_{0}} -1\right)\< \boldsymbol{x}_{i},\boldsymbol{\theta}^{*}-\boldsymbol{\theta}\>^{2},
    \end{align*}
    one can also show that \eqref{ineq:tf2lower re} holds true with $\tilde{B}^{v}_{-}(v)$ defined in \eqref{def:tBv-}. Hence, the proof of the Claim 1 is complete.

    \underline{\textbf{Proof of Claim 2}} The analysis for $\tilde{B}^{v}_{+}(v)$ and $\tilde{B}^{v}_{-}(v)$ is very similar to $B^v_{+}$ and $B^v_{-}$ in Section \ref{Sec:AppB}. Therefore, we will only present the main steps here and omit the tedious details. 
    
    By changing the variable with $t=\frac{\sigma^{2}}{v^{2}}-1$, we can obtain
    \begin{align} \label{e:TilBvBtCom}
    \tilde{B}^{v}_{+}(v) &\le a|t|^{\beta}+bt+c:=\bar{B}^{v}_{+}(t),
    \end{align}
    for some positive  
    \begin{align*}
        a=O\left( \alpha_{2}^{\beta-1}\right),\quad b=1+O(\alpha_{1}),
    \end{align*}
    and
    \begin{align*}
        c&=\theta_{0}+\frac{(1+\frac{1}{\theta_{0}})c_{u}\theta_{0}^{2}}{(1-\alpha_{2})\sigma^{2}}+\frac{2^{5\beta-4}\alpha_{2}^{\beta-1}}{\beta}\left(\frac{m_{2\beta}}{\sigma^{2\beta}}+1+\frac{ c_{u}L^{2\beta-2}}\theta_{0}^{2\beta}{(1-\alpha_{2})\sigma^{2\beta}} \right)+\frac{\log(\epsilon^{-1})}{n\alpha_{2}}\\
        &=O(\sqrt{d}\alpha_{1}+\alpha_{2}^{\beta-1}).
    \end{align*}
     Recalling
     $ \alpha_{1}=\sqrt{\frac{2\log\left(d \epsilon^{-1}\right)}{n}}, \alpha_{2}=\left( \frac{\log(\epsilon^{-1})}{n}\right)^{\frac{1}{\beta}},$
  it is clear that we can find some positive constant  $C=C(m_{2\beta},\sigma,\beta,c_{l}, c_{u}, L)$ such that
   \begin{align*}
       b-a( 2c)^{\beta-1}>\frac{1}{2},
   \end{align*}
   provided $n\ge Cd[\log(\epsilon^{-1})+1]$.
   In this case, Lemma \eqref{lem:root bound}(i) implies that the equation $\bar{B}^{v}_{+}(t)=0$ has two roots. Denote the largest one by $t_{+}$ and $v_{+}:=\frac{\sigma}{\sqrt{t_{+}+1}}$. Lemma \eqref{lem:root bound}(i) also tells us that
   \begin{align*}
       t_{+}\ge -\frac{c}{b-a( 2c})^{\beta-1} \ge -2c \ge -c_{0} \left(\sqrt{d}\alpha_1+\alpha_{2}^{\beta-1}\right),
   \end{align*}
   whenever $n\ge Cd[\log(\epsilon^{-1})+1]$ for some  (possibly larger) positive constant $C=C(m_{2\beta},\sigma,\beta,c_{l}, c_{u}, L)$, which implies 
   \begin{align*}
       v_{+}\le \left(1-c_{0} \left(\sqrt{d}\alpha_1+\alpha_{2}^{\beta-1}\right)\right)^{-\frac{1}{2}}\sigma =V_{0}.
   \end{align*}
By \eqref{e:TilBvBtCom}, we have
   \begin{align*}
      \tilde{B}^{v}_{+}(v_{+}) \le  \bar{B}_+^{v}(t_{+})=0,
   \end{align*}
 which is exactly \eqref{e:TilBv+=0}. The \eqref{e:TilBv-=0} can be proved in the same way, so the proof of Claim 2 is finished.
\end{proof}

\section{Catoni type joint robust \texorpdfstring{$\ell_2$}{l2}-penalized regression} \label{Sec:AppD}
The argument parallels Section \ref{Sec:AppC}; we highlight only the parameter changes and the new boundary computation for the penalty term.

\subsection{The strategy of the proof of Theorem \ref{Thm:jointly ridge}}\label{ss:appD strategy}

With $\alpha_{1},\alpha_{2}$ as in \eqref{def:alpha1,2 re}, set
\begin{equation}\label{def:vVtheta ridge}
v_{0}:=\!\Bigl(1+c_{0}\bigl(\alpha_{2}^{\beta-1}+\tfrac{\theta_{0}}{\theta_{0}+1/2}\bigr)\Bigr)^{-1/2}\!\sigma,\quad
V_{0}:=\!\Bigl(1-c_{0}\bigl(\alpha_{2}^{\beta-1}+\tfrac{\theta_{0}}{\theta_{0}+1/2}\bigr)\Bigr)^{-1/2}\!\sigma,\quad
\theta_{0}:=\tfrac{9\sqrt{d}L\alpha_{1}+\lambda_{0}\|\boldsymbol{\theta}^{*}\|_{2}}{\lambda_{0}+c_{l}/(6V_{0})},
\end{equation}
with $c_{0}:=1+\frac{2c_{u}}{\sigma}+\frac{2^{5\beta-2}}{\beta}\bigl(\frac{m_{2\beta}}{\sigma^{2\beta}}+1\bigr)$ and $\lambda=\lambda_{0}\alpha_{1}$. Define the cylinder
$\mathcal{K}^{ridge}:=\{(\boldsymbol{\theta},v):\|\boldsymbol{\theta}-\boldsymbol{\theta}^{*}\|_{2}\le\theta_{0},\ v_{0}\le v\le V_{0}\}$. As in Section \ref{Sec:AppC}, we assume $\theta_{0}\le 1$ and $c_{0}\bigl(\alpha_{2}^{\beta-1}+\theta_{0}/(\theta_{0}+1/2)\bigr)\le 1/2$. The reduction via the generalized Poincar\'e--Miranda theorem is identical to that in Section \ref{Sec:AppC} with $\tilde f_{1}^{ridge},\tilde f_{2}^{ridge}$ in place of $\tilde f_{1},\tilde f_{2}$, and yields the following two lemmas.

\begin{lem}\label{lem:probE1 ridge}
Under the assumptions of Theorem \ref{Thm:jointly ridge}, for every $\epsilon\in(0,1/6)$ there is $C=C(m_{2\beta},\sigma,\beta,c_{u},L,M)$ such that, for $n\ge Cd(c_{l}+\lambda_{0})^{-1}[1+\log(\epsilon^{-1})]$, the analogue of $\mathcal{E}_{1}$ for $\tilde f_{1}^{ridge}$ has probability $\ge 1-4\epsilon$.
\end{lem}

\begin{lem}\label{lem:probE2 ridge}
Under the same assumptions and sample size bound, the analogues of $\mathcal{E}_{2}^{\pm}$ for $\tilde f_{2}^{ridge}$ each have probability $\ge 1-\epsilon$.
\end{lem}

\begin{proof}[Proof of Theorem \ref{Thm:jointly ridge}]
Combine Lemmas \ref{lem:probE1 ridge}--\ref{lem:probE2 ridge} with the generalized Poincar\'e--Miranda theorem \cite[Theorem~2.4]{Fr18}; the error bounds follow from \eqref{def:vVtheta ridge}.
\end{proof}

\subsection{Proof of Lemma \ref{lem:probE1 ridge}}\label{ss:appD probE1}

\begin{proof}
The events $\mathcal{E}_{1,1},\mathcal{E}_{1,2},\mathcal{E}_{1,3}$ from \eqref{def:E11E12E13} apply unchanged (with $v_{0},V_{0}$ from \eqref{def:vVtheta ridge}), and Lemma \ref{lem:probE11,E12,E13} gives $\mathbb{P}(\mathcal{E}_{1,1}\cap\mathcal{E}_{1,2}\cap\mathcal{E}_{1,3})\ge 1-4\epsilon$.

Since $\tilde f_{1}^{ridge}(\boldsymbol{\theta},v)=\tilde f_{1}(\boldsymbol{\theta},v)-\lambda\boldsymbol{\theta}$, for $\boldsymbol{\theta}\in\partial\mathcal{B}_{\theta_{0}}(\boldsymbol{\theta}^{*})$,
\[
-\lambda\langle\boldsymbol{\theta},\boldsymbol{\theta}-\boldsymbol{\theta}^{*}\rangle
=-\lambda\|\boldsymbol{\theta}-\boldsymbol{\theta}^{*}\|_{2}^{2}-\lambda\langle\boldsymbol{\theta}^{*},\boldsymbol{\theta}-\boldsymbol{\theta}^{*}\rangle
\le -\lambda\theta_{0}^{2}+\lambda\|\boldsymbol{\theta}^{*}\|_{2}\theta_{0}.
\]
Combining this with \eqref{ineq:suptf1} and $\lambda=\lambda_{0}\alpha_{1}$, it holds as long as $n\ge Cd(c_{l}+\lambda_{0})^{-1}[1+\log(\epsilon^{-1})]$ that
\[
\sup_{v,\boldsymbol{\theta}}\langle\tilde f_{1}^{ridge}(\boldsymbol{\theta},v),\boldsymbol{\theta}-\boldsymbol{\theta}^{*}\rangle
\le L\theta_{0}\alpha_{1}\!\Bigl[-\tfrac{c_{l}+6V_{0}\lambda_{0}}{6V_{0}L}\theta_{0}+9\sqrt{d}\alpha_{1}+\tfrac{\lambda_{0}}{L}\|\boldsymbol{\theta}^{*}\|_{2}-\tfrac{\lambda_{0}}{L}\theta_{0}\Bigr]=0,
\]
the last equality by the definition of $\theta_{0}$ in \eqref{def:vVtheta ridge}.
\end{proof}

\subsection{Proof of Lemma \ref{lem:probE2 ridge}}\label{ss:appD probE2}

\begin{proof}
$\tilde f_{2}^{ridge}=\tilde f_{2}$, so Claim 1 of Lemma \ref{lem:probE2 re} carries over verbatim and gives the upper and lower concentration estimates of the form \eqref{ineq:tf2upper re} and \eqref{ineq:tf2lower re} for $\tilde f_{2}^{ridge}$.

For Claim 2, with $t=\sigma^{2}/v^{2}-1$, $\tilde B^{v}_{+}(v)\le a|t|^{\beta}+bt+c=\bar B^{v}_{+}(t)$ exactly as in \eqref{e:TilBvBtCom}, with positive
\[
a=O(\alpha_{2}^{\beta-1}),\quad b=1+\theta_{0}+\tfrac{(1+1/\theta_{0})c_{u}\theta_{0}^{2}}{(1-\alpha_{2})\sigma^{2}},\quad c=O(\theta_{0}+\alpha_{2}^{\beta-1}).
\]
For $C=C(m_{2\beta},\sigma,\beta,c_{u},L)$ large and $n\ge Cd[1+\log(\epsilon^{-1})]$, $b>c+3/4$ and $b-a(2c)^{\beta-1}>c+1/2$. Lemma \ref{lem:root bound}(i) gives a largest root $t_{+}$ with
\[
t_{+}\ge -\tfrac{c}{b-a(2c)^{\beta-1}}\ge -\tfrac{c}{c+1/2}\ge -c_{0}\!\bigl(\alpha_{2}^{\beta-1}+\tfrac{\theta_{0}}{\theta_{0}+1/2}\bigr),
\]
the last inequality by the definition of $c$ and $c_{0}$. Setting $v_{+}:=\sigma/\sqrt{1+t_{+}}\le V_{0}$ gives $\tilde B^{v}_{+}(v_{+})\le 0$, and monotonicity of $\tilde f_{2}^{ridge}$ in $v$ yields $\mathbb{P}(\mathcal{E}_{2}^{ridge+})\ge 1-\epsilon$. The bound on $\mathcal{E}_{2}^{ridge-}$ is similar.
\end{proof}


\begin{thebibliography}{00}
\bibitem{AVP20} Ankit Pensia, Varun Jog, and Po-Ling Loh. Robust regression with covariate filtering: Heavy tails and adversarial contamination. {\it arXiv preprint arXiv:2009.12976}, 2020.

\bibitem{AY22} Auddy, A., \& Yuan, M. (2022). On estimating rank-one spiked tensors in the presence of heavy tailed errors. {\it IEEE Transactions on Information Theory, 68}(12), 8053-8075.

\bibitem{BCL11} Belloni, A., Chernozhukov, V., \& Wang, L. (2011). Square-root ridge: pivotal recovery of sparse signals via conic programming. Biometrika, 98(4), 791-806.

\bibitem{BGS22} Babii, A., Ghysels, E., \& Striaukas, J. (2022). Machine learning time series regressions with an application to nowcasting. {\it Journal of Business \& Economic Statistics, 40}(3), 1094-1106.

\bibitem{BCL13} Bubeck, S., Cesa-Bianchi, N., \& Lugosi, G. (2013). Bandits with heavy tail. {\it IEEE Transactions on Information Theory, 59}(11), 7711-7717.

\bibitem{BMGS19} Bertrand,Q., Massias,M., Gramfort,A.,  \& Salmon, J. (2019). Handling correlated and repeated measurements with the smoothed multivariate square-root ridge. Advances in Neural Information Processing Systems 32 (NeurIPS 2019). 


\bibitem{Ca12} Catoni, O. (2012), Challenging the Empirical Mean and Empirical Variance: A Deviation Study. {\it Annales de I’Institut Henri Poincaré— Probabilités et Statistiques}, 48, 1148–1185.

\bibitem{CGM10} Croux, C., Gelper, S., \& Mahieu, K. (2010). Robust exponential smoothing of multivariate time series. {\it Computational statistics \& data analysis, 54}(12), 2999-3006.

\bibitem{CJLX21} Chen, P., Jin, X., Li, X. and Xu, L., 2021. A generalized Catoni’s M-estimator under finite $\alpha$-th moment assumption with $\alpha\in (1, 2)$. {\it Electronic Journal of Statistics, 15}(2), pp.5523-5544.

\bibitem{CMA94} Connor, J. T., Martin, R. D., \& Atlas, L. E. (1994). Recurrent neural networks and robust time series prediction. {\it IEEE transactions on neural networks, 5}(2), 240-254.

\bibitem{EoJo15} Eom,Y.H. and Jo, H.H. (2015),   
Tail-scope: Using friends to estimate heavy tails of degree distributions in large-scale complex networks.
{\it Scientific Reports}, vol. 5,  09752. (2015) 



\bibitem{FKSZ19} Fan, J., Ke, Y., Sun, Q., \& Zhou, W. X. (2019). FarmTest: Factor-adjusted robust multiple testing with approximate false discovery control. {\it Journal of the American Statistical Association, 114}(526), 1684-1696.



\bibitem{FLW17} Fan, J., Li, Q. \& Wang, Y. (2017). Estimation of high dimensional mean regression
in the absence of symmetry and light tail assumptions. {\it Journal of the Royal Statistical
Society, Series B, 79}(1), 247–265.

\bibitem{FLW18} Fan, J., Liu, H., \& Wang, W. (2018). Large covariance estimation through elliptical factor models. {\it Annals of statistics, 46}(4), 1383.

\bibitem{FR03} Finkenstadt, B., \& Rootzén, H. (Eds.). (2003). Extreme values in finance, telecommunications, and the environment. CRC Press.

\bibitem{Fr18} Frankowska, H. (2018). The Poincar\'e–Miranda theorem and viability condition. {\it Journal of Mathematical Analysis and Applications, 463}(2), 832-837.

\bibitem{FWZ19} Fan, J., Wang, W., \& Zhong, Y. (2019). Robust covariance estimation for approximate factor models. {\it Journal of econometrics, 208}(1), 5-22.

\bibitem{FWZ21} Fan, J., Wang, W., \& Zhu, Z. (2021). A shrinkage principle for heavy-tailed data: High-dimensional robust low-rank matrix recovery. {\it Annals of statistics, 49}(3), 1239.

\bibitem{FWZZ21} Fan, J., Wang, K., Zhong, Y., \& Zhu, Z. (2021). Robust high dimensional factor models with applications to statistical machine learning. {\it Statistical science: a review journal of the Institute of Mathematical Statistics, 36}(2), 303.

\bibitem{GMVX22} Guerrier, S., Molinari, R., Victoria-Feser, M. P., \& Xu, H. (2022). Robust two-step wavelet-based inference for time series models. {\it Journal of the American Statistical Association, 117}(540), 1996-2013.


\bibitem{Hu64} Huber, P. J. (1964). Robust estimation of a location parameter. {\it Annals of Mathematical Statistics, 35}(1), 73–101.

\bibitem{Hu73} Huber, P. J. (1973). Robust regression: asymptotics, conjectures and Monte Carlo. {\it The annals of statistics}, 799-821.

\bibitem{HR11} Huber, P. J., \& Ronchetti, E. M. (2011). {\it Robust statistics}. John Wiley \& Sons.

\bibitem{JP22}Jasper CH Lee and Paul Valiant. Optimal sub-gaussian mean estimation in $\mathbb{R}$. In {\it 2021 IEEE 62nd Annual Symposium on Foundations of Computer Science (FOCS)}, pages 672–683. IEEE, 2022.


\bibitem{JXX19}  Jia, J., Xie, F.,  \& Xu, L. (2019) Sparse Poisson regression with penalized weighted score function. {\it Electronic Journal of Statistics, 13}(2), 2898-2920.


\bibitem{JSF18}  Jiang, B., Sun, Q.,  \& Fan, J. (2018) Bernstein's inequalities for general Markov chains. {\it arXiv preprint arXiv:1805.10721}. 


\bibitem{KMRSZ19} Ke, Y., Minsker, S., Ren, Z., Sun, Q., \& Zhou, W. X. (2019). User-friendly covariance estimation for heavy-tailed distributions. {\it Statistical Science, 34}(3), 454-471.

\bibitem{LL20} Lecu\'e, G., \& Lerasle, M. (2020). Robust machine learning by median-of-means: theory and practice. {\it Annals of Statistics, 48}(2), 906-931.

\bibitem{LeVa22-0}
Lee, ~J.~C.~H., Valiant, L. (2022). Optimal Sub-Gaussian Mean Estimation in $\mathbb{R}$,  \textit{Proceedings of the 63rd IEEE Annual Symposium on Foundations of Computer Science (FOCS)}, pp. 674--683. 

\bibitem{LeVa22-1}
Lee, ~J.~C.~H., Valiant, L. (2022).
Optimal Sub-Gaussian Mean Estimation in Very High Dimensions, in \textit{13th Innovations in Theoretical Computer Science Conference (ITCS 2022)}, Art. no. 96.


\bibitem{LM19} Lugosi G., Mendelson S. (2019) "Sub-Gaussian estimators of the mean of a random vector," {\it The Annals of Statistics, Ann. Statist. 47}(2), 783-794, 

\bibitem{LM19-2} Lugosi, G., Mendelson, S. (2019) Mean Estimation and Regression Under Heavy-Tailed Distributions: A Survey. {\it Foundations of Computational Mathematics, 19}(5), 1145–1190. 

\bibitem{Ma89} Mammen, E. (1989). Asymptotics with increasing dimension for robust regression with applications to the bootstrap. {\it The annals of statistics,} 382-400.

\bibitem{Mi18} Minsker, S. (2018). Sub-Gaussian estimators of the mean of a random matrix with heavy-tailed entries. {\it The Annals of Statistics, 46}(6A), 2871-2903.

\bibitem{Mol22} Molstad, A.J. (2022). New Insights for the Multivariate Square-Root ridge.{\it Journal of Machine Learning Research, 23} (66):1-52.

\bibitem{PKH22} Pillutla, K., Kakade, S. M., \& Harchaoui, Z. (2022). Robust aggregation for federated learning. {\it IEEE Transactions on Signal Processing, 70,} 1142-1154.

\bibitem{PSBR20} Prasad, A., Suggala, A. S., Balakrishnan, S., \& Ravikumar, P. (2020). Robust estimation via robust gradient estimation. {\it Journal of the Royal Statistical Society Series B: Statistical Methodology, 82}(3), 601-627.


\bibitem{Qu21} Qu, L. (2021). A new approach to estimating earnings forecasting models: Robust regression MM-estimation. {\it International Journal of Forecasting, 37}(2), 1011-1030.

\bibitem{Su21} Sun, Q., (2021). Do we need to estimate the variance in robust mean estimation. {\it arXiv preprint arXiv:2107.00118.}

\bibitem{SZ12} Sun, T. and Zhang, C.H. (2012). Scaled sparse linear regression. {\it Biometrika, 99}(4), pp.879-898.

\bibitem{SZF20} Sun, Q., Zhou, W.X. \& Fan, J. (2020). Adaptive huber regression. {\it Journal of the American Statistical Association, 115}(529), pp.254-265.

\bibitem{VBRD14} Van de Geer, S., Bühlmann, P., Ritov, Y. A., \& Dezeure, R. (2014). On asymptotically optimal confidence regions and tests for high-dimensional models. {\it Annals of statistics, 42}(3), 1166-1202.

\bibitem{Ve10} Vershynin, R. (2010). Introduction to the non-asymptotic analysis of random matrices. {\it arXiv preprint arXiv:1011.3027.}

\bibitem{Ve18} Vershynin, R. (2018). {\it High-dimensional probability: An introduction with applications in data science} (Vol. 47). Cambridge university press.


\bibitem{WLXXZ24} Wang, Y., Li, G., Xiao, Z., Xu, L., \& Zhang, W. (2024). Robust estimation for high-dimensional time series with heavy tails. {\it arXiv preprint arXiv:2411.05217.}

\bibitem{WPL15} Wang, L., Peng, B., \& Li, R. (2015). A high-dimensional nonparametric multivariate test for mean vector. Journal of the American Statistical Association, 110(512), 1658-1669.


\bibitem{WR23} Wang, H., \& Ramdas, A. (2023). Catoni-style confidence sequences for heavy-tailed mean estimation. {\it Stochastic Processes and Their Applications, 163}, 168-202.


\bibitem{WZHCT21} Wang, Y., Zhong, X., He, F., Chen, H., \& Tao, D. (2021, October). Huber additive models for non-stationary time series analysis. In {\it International conference on learning representations.}

\bibitem{WZZZ21} Wang, L., Zheng, C., Zhou, W., \& Zhou, W. X. (2021). A new principle for tuning-free Huber regression. { \it Statistica Sinica, 31}(4), 2153-2177.


\bibitem{YM79} Yohai, V. J., \& Maronna, R. A. (1979). Asymptotic behavior of M-estimators for the linear model. {\it The Annals of Statistics,} 258-268.

\bibitem{ZBFL18} Zhou, W. X., Bose, K., Fan, J., \& Liu, H. (2018). A new perspective on robust M-estimation: Finite sample theory and applications to dependence-adjusted multiple testing. {\it Annals of statistics, 46}(5), 1904.






\end{thebibliography}
\end{document}